\documentclass[11pt,a4paper]{article}
\usepackage[a4paper,margin=26mm,headheight=14pt]{geometry}
\usepackage[T1]{fontenc}
\usepackage[utf8]{inputenc}
\usepackage{lmodern,microtype}
\usepackage{amsmath,amssymb,amsthm,mathtools}
\usepackage{xurl}
\usepackage[unicode,hidelinks]{hyperref}
\usepackage{fancyhdr}
\usepackage{bbm}
\usepackage{mathrsfs}

\newcommand{\eee}{{\rm e}}

\newcommand{\me}{\mathbb{E}}
\newcommand{\mn}{\mathbb{N}}
\newcommand{\mr}{\mathbb{R}}
\newcommand{\mmp}{\mathbb{P}}

\def\1{\mathbbm{1}}

\newcommand{\PP}{\mathbb P}

\newcommand{\ii}{\mathrm i}

\allowdisplaybreaks[2]
\numberwithin{equation}{section}
\newtheorem{theorem}{Theorem}[section]
\newtheorem{lemma}[theorem]{Lemma}
\newtheorem{proposition}[theorem]{Proposition}
\newtheorem{corollary}[theorem]{Corollary}
\theoremstyle{remark}

\newcommand{\Prob}{\mathbb P}
\newcommand{\F}{\mathcal F}
\newcommand{\Primes}{\mathcal P}

\newcommand{\dd}{\,\mathrm d}

\newcommand{\norm}[1]{\left\lVert#1\right\rVert}
\newcommand{\floor}[1]{\left\lfloor#1\right\rfloor}

\newcommand{\as}{\text{almost surely}}

\makeatletter
\newcommand{\rmflabelalias}[2]{%
  \expandafter\let\csname r@#1\expandafter\endcsname\csname r@#2\endcsname}
\AtBeginDocument{%
  \rmflabelalias{eq:martrepr}{eq:decomposition}%
  \rmflabelalias{eq:deterministic-Q-bounds}{eq:clockelementary}%
  \rmflabelalias{lem:deterministic-Q-bounds}{lem:clockelementary}%
  \rmflabelalias{a:eq:PNT}{eq:pnt}%
  \rmflabelalias{a:eq:orth}{eq:orth}%
  \rmflabelalias{a:eq:second}{eq:orthogon}%
  \rmflabelalias{b:eq:covinput}{eq:covinput}%
}
\makeatother

\begin{document}
\title{A law of the iterated logarithm for random multiplicative functions on integers with a large prime factor}\date{}
\author{Dariusz Buraczewski\footnote{Mathematical Institute, University of Wroclaw, Poland; e-mail address: dariusz.buraczewski@math.uni.wroc.pl} \ \ and \ \ Alexander Iksanov\footnote{Faculty of Computer Science and Cybernetics, Taras Shevchenko National University of Kyiv, Ukraine; e-mail address:
iksan@univ.kiev.ua}}
\maketitle

\begin{abstract}
Let $f$ be a Rademacher random multiplicative function supported on the squarefree integers. We investigate the almost sure fluctuations of the restricted sum $M^{>}(x)=\sum_{n\leq x;~ P(n)>x^{1/2}}f(n)$, where $P(n)$
is the largest prime factor of $n$. With $Q(x)$ being the conditional
variance of $M^{>}(x)$ given the values of $f$ at primes not exceeding $x^{1/2}$, we prove that, almost surely, the upper and lower limits of $M^{>}(x)(2Q(x)\log\log Q(x))^{-1/2}$ as $x\to\infty$ are $1$ and $-1$, respectively. {We also show}
that the almost sure cluster set of the normalized sum is $[-1,1]$.
The proof combines conditioning on small primes with estimates for random Euler products and their associated multiplicative chaos. Finally, combining our main result with Harper's lower bound for the conditional variance, we establish that the upper and lower limits of $M^{>}(x)/(x^{1/2}(\log\log x)^{1/4})$ are strictly positive and strictly negative, respectively, almost surely.
\end{abstract}



\section{Introduction and main results}\label{sec:model}

Let $\Primes$ be the set of primes. On a probability space $(\Omega,\F,\Prob)$, let $(f(p))_{p\in\Primes}$ be independent random variables with $\mmp\{f(p)=\pm1\}=1/2$. A positive integer is \emph{squarefree} if no square of a prime divides it. Define $f(1)=1$ and, for $n\ge 2$, $$f(n)=\begin{cases}
\displaystyle\prod_{p\mid n}f(p),&n\text{ squarefree},\\
0,&n\text{ not squarefree}.
 \end{cases}
$$ Equivalently, $f(n)=(\mu(n))^2\prod_{p\mid n}f(p)$, where $\mu$ is the M\"obius function. Thus, for instance, $f(6)=f(2)f(3)$, $f(12)=0$ and $f(30)=f(2)f(3)f(5)$. The function $f$ so defined is called a
\emph{Rademacher random multiplicative function}. Put $$M(x)=\sum_{1\le n\le x}f(n),\quad x\in\mr$$ with the usual convention that the sum over the empty set is equal to $0$. Plainly, $M(x)=0$ for $x<1$ and $M(x)=M(\floor{x})$ whenever $x\geq 1$.

Let $P(n)$ be the largest prime factor of an integer $n\ge 2$, and put $P(1)=1$. For $x\geq 0$, define
\begin{equation*}
M^{>}(x):=\sum_{n\le x;~ P(n)>x^{1/2}}f(n)\quad\text{and}\quad Q(x):=\sum_{x^{1/2}<p\le x}(M(x/p))^2.
\end{equation*}
All sums with the index $p$ are over primes. The sum $M^{>}$ is called a \emph{restricted sum}. The restriction retains a positive proportion of the integers up to $x$, while permitting a particularly useful separation of the prime signs.

\subsection{A short survey of relevant works}\label{sec:survey}

Wintner introduced this model in \cite{Wintner:1944} as a probabilistic counterpart of the M\"obius function. He proved that, for each $\varepsilon>0$, $M(x)=O(x^{1/2+\varepsilon})$ almost surely (a.s.), whereas $M(x)=O(x^{1/2-\varepsilon})$ fails a.s. The former limit relation
is a probabilistic counterpart of the fact that the Riemann hypothesis is equivalent to $\sum_{n\leq x}\mu(n)=O(x^{1/2+\varepsilon})$ for each $\varepsilon>0$. These results determine the power of $x$, but leave open the size of fluctuations on the scale $x^{1/2}$. The problem of finding a counterpart of the law of the iterated logarithm (LIL) for $M(x)$, properly normalized, goes back to Wintner.

Hal\'asz in \cite{Halasz:1983} strengthened the earlier upper bounds by proving that, for some $C>0$, $$|M(x)|=O\big(x^{1/2}\exp\{C(\log\log x\log\log\log x)^{1/2}\}\big), \quad x\to\infty\quad\text{a.s.}$$ He also proved that, for some $B>0$,
the bound $$M(x)=O\big(x^{1/2}\exp\{-B(\log\log x\log\log\log x)^{1/2}\}\big)$$ fails {a.s.} Significantly later, Basquin in \cite{Basquin:2012} and, independently, Lau, Tenenbaum and Wu in \cite{Lau+Tenenbaum+Wu:2013}
obtained, for each $\varepsilon>0$,\footnote{The published version of \cite{Lau+Tenenbaum+Wu:2013} states the exponent $3/2+\varepsilon$, whereas the exponent $2+\varepsilon$ is that of the corrected version cited in the bibliography.}
\begin{equation}\label{eq:wu}
|M(x)|=O\big(x^{1/2}(\log\log x)^{2+\varepsilon}\big)\quad\text{a.s.}
\end{equation}
Caich in \cite{Caich:2025+} reduced the exponent $2+\varepsilon$ to $3/4+\varepsilon$. More recently, Durkan and Pearce-Crump \cite{DurkanPearceCrump} and, independently, Verreault \cite{Verreault} proved that, for each $\varepsilon>0$,
$$|M(x)|=O\big(x^{1/2}(\log\log x)^{1/4+\varepsilon}\big),\quad x\to\infty\quad\text{a.s.}$$ {Actually, Verreault obtained the sharper bound}
$O\big(x^{1/2}(\log\log x)^{1/4}(\log\log\log x)^{1+\varepsilon}\big)$ for each $\varepsilon>0$.

In the opposite direction, Hal\'asz's lower bound was improved in \cite{Harper:2013AAP}, which showed
that, for each $\varepsilon>0$,
the bound $M(x)=O\big(x^{1/2}(\log\log x)^{-5/2-\varepsilon}\big)$ fails a.s. Recently, Harper in \cite{Harper:2023} proved that, for any fixed deterministic function $V$ with $\lim_{x\to\infty}V(x)=\infty$,
$$|M(x)|\geq \frac{x^{1/2}(\log\log x)^{1/4}}{V(x)}\quad\text{a.s.}$$ for arbitrarily large $x$. Together, these results determine the optimal exponent $1/4$ of $\log\log x$, but do not identify an exact deterministic LIL normalization, nor a limit constant. Taking $\varepsilon<1/4$ in the upper bound above gives $\lim_{n\to\infty}(n\log\log n)^{-1/2}M(n)= 0$ a.s., showing that the classical LIL normalization is too large for this model. A more extensive account of these developments is given in the accompanying paper \cite{Buraczewski+Iksanov:accompanying}.

Limit theorems for the restricted sums have also received considerable attention. Mastrostefano proved in~ Theorem~1.1 of~\cite{Mastrostefano:2022} that, for each $\varepsilon>0$, $$|M^{>}(x)|=O\big(x^{1/2}(\log\log x)^{1/4+\varepsilon}\big),\quad x\to\infty\quad\text{a.s.}$$

His result is also valid for Steinhaus random multiplicative functions (completely multiplicative functions whose prime values are independent and uniform on the unit circle). The proof combines conditioning on the small primes with estimates for the resulting random variance, which rely on Harper's low-moment estimates \cite{Harper:2020} (our proof of the upper bound in Theorem~\ref{thm:main} follows the same scheme). Large prime subsums also enter Harper's lower-bound argument, see Section~3.1 and Proposition~2 in \cite{Harper:2023}, where conditional covariance estimates allow one to find sufficiently many weakly correlated sums.
Thus, this restriction isolates a part of the multiplicative sum which already exhibits the fluctuations relevant to the full problem.

For Steinhaus random multiplicative functions, Hardy in \cite{Hardy:2025} proved that the restricted sum, multiplied by $x^{-1/2}(\log\log x)^{1/4}$, converges in distribution to a complex Gaussian scale mixture, with a mixing variable expressed through a critical Gaussian multiplicative chaos. In Theorem~2 of the same paper, Hardy also obtained an approximation of the conditional variance by an Euler product integral for both the Steinhaus and Rademacher models.
The corresponding distributional limit for the full Steinhaus sum was subsequently proved by Gorodetsky and Wong in \cite{Gorodetsky+Wong:2025}. These results describe convergence in distribution as $x\to\infty$. A.s.\  bounds of the order given by LIL,
up to multiplicative constants, were obtained by Shala \cite{Shala:2026} for sums of $f$ over sets on which these sums are approximately Gaussian. Our purpose is to determine the a.s.\ fluctuations of the Rademacher restricted sum as $x$ varies.

\subsection{Main results}

Denote by $\F_y=\sigma(f(p):~p\le y)$ the $\sigma$-algebra generated by the collection $(f(p))_{p\leq y}$. We write $\me_y$ and $\mmp_y$ for conditional expectation and conditional probability given $\F_y$, respectively.
If $n\le x$ has a prime factor $p>x^{1/2}$, then $n=pm$ with $m\le x/p<x^{1/2}<p$. In particular, $p$ is unique and $p\nmid m$. Consequently $f(pm)=f(p)f(m)$ and
\begin{equation}\label{eq:decomposition}
M^{>}(x) =\sum_{x^{1/2}<p\le x}f(p)M(x/p).
\end{equation}
Conditionally on $\F_{x^{1/2}}$, the coefficients $M(x/p)$ in this representation are fixed and the signs $(f(p))_{p>x^{1/2}}$ remain independent. Hence
\begin{equation*}
\me_{x^{1/2}}[M^{>}(x)]=0\quad\text{and}\quad \me_{x^{1/2}}\big[(M^{>}(x))^2\big]=Q(x)\quad\text{a.s.}
\end{equation*}

The restriction $p>x^{1/2}$ ensures that the coefficients $M(x/p)$ depend only on the signs $f(q)$ with $q\le x^{1/2}$. Thus, conditionally on $\F_{x^{1/2}}$, the representation
\eqref{eq:decomposition} is a sum of independent Rademacher variables with fixed coefficients and variance $Q(x)$. This gives exponential tail estimates at each fixed $x$.
To prove an LIL, we must also control the sum uniformly in $x$ and show that large excursions occur infinitely often. For the lower bound of $\limsup$, we compare sums at many different points.
These sums share prime signs, and conditional covariance estimates allow us to control their dependence, see Section~\ref{sec:discussion}.

The random variance retains the contribution of the small-prime signs to the size of the sum. This is consistent with the role of random variances in the distributional results just discussed and with
Harper's low-moment estimates \cite{Harper:2020}, which show that the second moment alone does not describe the typical size of a multiplicative sum. In the present article we use $Q(x)$ itself in the
normalization and obtain the classical signed LIL constants.

Here is our main result.
\begin{theorem}\label{thm:main}
We have
\begin{equation*}
\limsup_{x\to\infty}\frac{M^{>}(x)}{\big(2Q(x)\log\log Q(x)\big)^{1/2}} =1\quad\as,
\end{equation*}
with the corresponding lower limit being equal to $-1$.
\end{theorem}

The normalization is well defined for all sufficiently large $x$. Lemma~\ref{lem:clockelementary} entails $\log\log Q(x)\sim\log\log x$ as $x\to\infty$. Thus, $\log\log Q(x)$ in Theorem~\ref{thm:main} can be replaced by $\log\log x$.

For a random function $A$, denote by ${\rm Clust}_{x\to\infty}A(x)$ its cluster set as $x\to\infty$, understood realization-wise. Write $\mathcal C$ for the cluster set of the normalized function in
Theorem~\ref{thm:main}. The theorem gives $\sup\mathcal C=1$ and $\inf\mathcal C=-1$ a.s. Our next result identifies all its points.
\begin{corollary}\label{cor:cluster}
We have
\begin{equation}\label{eq:cluster}
 {\rm Clust}_{x\to\infty}
 \frac{M^{>}(x)}{\big(2Q(x)\log\log Q(x)\big)^{1/2}}
 =[-1,1]\quad\text{almost surely}.
\end{equation}
\end{corollary}

Theorem~\ref{thm:main} also yields information at the deterministic scale $x^{1/2}(\log\log x)^{1/4}$ of the bounds discussed in Section~\ref{sec:survey}.
{Using Harper's probabilistic lower bound for $Q(x)$, the argument proving Theorem~\ref{thm:main} also yields the following result at the deterministic scale $x^{1/2}(\log\log x)^{1/4}$.}
\begin{corollary}\label{cor:scale}
We have
\begin{equation*}
\limsup_{x\to\infty}\frac{M^{>}(x)}{x^{1/2}(\log\log x)^{1/4}}>0\quad\text{and}\quad\liminf_{x\to\infty}\frac{M^{>}(x)}{x^{1/2}(\log\log x)^{1/4}}<0\quad\as
\end{equation*}
\end{corollary}
Corollary~\ref{cor:scale} is proved in Section~\ref{sect:proofscale}. Together with Mastrostefano's upper bound, {our result}
determines the {scale of the a.s.\ extreme fluctuations}
of $M^{>}(x)$ up to a factor {of} $(\log\log x)^{\varepsilon}$.

The random normalization in Theorem~\ref{thm:main} reflects the small-prime environment. For each fixed $x$, $Q(x)$ is the predictable quadratic variation of the martingale obtained by revealing the large-prime signs in \eqref{eq:decomposition}. {With the normalization $(2Q(x)\log\log Q(x))^{1/2}$,} we recover the classical LIL constants $1$ and $-1$.
This is consistent with the role of self-normalization in martingale limit theory, see \cite{delaPena+Lai+Shao:2009,HallHeyde}. For the Steinhaus model, deterministically normalized restricted and full sums converge in distribution to Gaussian scale mixtures \cite{Gorodetsky+Wong:2025, Hardy:2025}. Theorem~\ref{thm:main} provides an a.s.\ analogue of this principle for the restricted Rademacher sum, with the {conditional}
variance {$Q(x)$} entering the normalization and the limiting constants remaining deterministic. The conditional variance is itself closely related to random Euler products, as discussed in Section~\ref{sec:chaos}.

\subsection{Discussion of the proofs}\label{sec:discussion}

\paragraph{The conditional variance and the Euler chaos.}\label{sec:chaos}
For $z\geq 2$ and $t\in\mr$, put $F_z(t):=\prod_{p\leq z}(1+f(p)p^{-1/2-\ii t})$. The random measures $|F_z(t)|^2\dd t/\me[|F_z(t)|^2]$ are closely related to multiplicative chaos, see \cite{Harper:2020,Gorodetsky+Wong:2025,HKK:2026}. Motivated by this connection, we call them the \emph{Euler chaos}. It enters Theorem~\ref{thm:main} in two ways.

First, the Euler chaos determines the size of the normalization. If $x^{1/2}<p\leq x$, then the coefficients $M(x/p)$ in \eqref{eq:decomposition} {are partial sums of the Dirichlet coefficients of $F_{x^{1/2}}$.}
{Through Plancherel-type identities and prime averaging, the conditional variance $Q(x)$ can be related to integrals of $|F_{x^{1/2}}(t)|^2$. This connection underlies Proposition~1 of \cite{Harper:2023}, which we use in the form \eqref{eq:HarperDet}, and Theorem~2 of \cite{Hardy:2025}.}

Second, the Euler chaos governs the dependence between the sums at different points. For the lower bound, we need many points $x<y$ at which the sums are nearly uncorrelated, conditionally on the small primes. By Lemma~\ref{lem:perron}, the conditional covariances are, up to a localization, Fourier transforms of the Euler chaos at the frequency $\log(y/x)$. They are small unless the chaos is concentrated on short intervals, and such concentration is excluded by barrier estimates for partial Euler products, see Sections~\ref{sect:proofnovel} and~\ref{sect:proofcov}. This mechanism underlies Harper's lower bound \cite{Harper:2023}.

\paragraph{Relation to general limit theorems.}
For sums of independent random variables and for martingales, general self-normalized LILs
are available, see \cite{Griffin+Kuelbs:1989} for independent identically distributed summands, \cite{Stout:1970,delaPena+Klass+Lai:2004} for martingales, and the monograph \cite{delaPena+Lai+Shao:2009} {for a systematic treatment of self-normalized processes.} LILs
for weighted sums of independent random variables (double arrays) are given, for instance, in \cite{Lai+Wei:1982}. We now explain why these results do not apply directly in our setting.

Theorem~\ref{thm:main} may
be viewed as an arithmetic counterpart of Stout's LIL for martingales normalized by their conditional variances \cite{Stout:1970} and of self-normalized LILs, see, for instance, \cite{delaPena+Lai+Shao:2009}.
However, Theorem~\ref{thm:main} is not a consequence of these results, since neither $(M(x), \F_x)_{x\geq 1}$ nor $(M^{>}(x), \F_x)_{x\geq 1}$ is a martingale.
A martingale structure emerges only when $x$ is fixed. Indeed, conditionally on $\F_{x^{1/2}}$, the coefficients $M(x/p)$ in \eqref{eq:decomposition} are fixed, and the partial sums obtained by successively revealing the signs $f(p)$ for $x^{1/2}<p\leq x$, in increasing order of $p$, form a martingale with terminal value $M^{>}(x)$ and terminal predictable quadratic variation $Q(x)$. This structure depends on $x$, since both the coefficients and the range of summation vary with $x$, and summands may enter or leave the sum, see \eqref{eq:S-entry-exit-short}. Thus, the classical martingale LIL does not directly yield the desired a.s.\ asymptotics as $x\to\infty$. To pass from this fixed-$x$ structure to an LIL, we need the interpolation argument of Section~\ref{proof:prop2.1} for the upper bound and the analysis of conditional covariances via random Euler products in Sections~\ref{sect:proofnovel} and~\ref{sect:proofcov} for the lower bound.

\paragraph{Outline of the proofs.}
{\sc Upper bound {for $\limsup$}} (Sections~\ref{sect:upper} and~\ref{proof:prop2.1}). On a geometric grid in each block $[\exp(2^r), \exp(2^{r+1})]$ for $r\in\mn:=\{1,2,\ldots\}$,
the conditional Hoeffding inequality and the Borel--Cantelli lemma give the upper bound (Proposition~\ref{prop:grid}). Between the grid points, the increments of $M^{>}$ and of the coefficient vectors are controlled uniformly in $x$ (Proposition~\ref{prop:interp}) by a maximal fourth-moment inequality based on hypercontractivity (Lemma~\ref{lem:fourth}), as in \cite{Lau+Tenenbaum+Wu:2013,Mastrostefano:2022}, and by an exponential bound for sums with predictable coefficients. These errors are negligible on the scale of the LIL by Lemma~\ref{lem:lowerclock}.

{\sc Lower bound {for $\limsup$}} (Sections~\ref{sect:lower}, \ref{sect:proofnovel} and~\ref{sect:proofcov}). We follow the architecture of Harper's proof in Section~3.1 of \cite{Harper:2023}. Conditionally on the small primes, we consider the sums at many points of a block. Their conditional variances are bounded from below with the help of Harper's Proposition~1, and their conditional covariances are small for all but few pairs of points by \eqref{eq:covinput}, a sharpened form of his Proposition~2. Selecting nearly uncorrelated sums and applying Lemma~\ref{lem:rademacher-max} gives the constant $1$. Compared with \cite{Harper:2023}, the new ingredients are {the covariance estimate} \eqref{eq:covinput} with an arbitrary exponent $\kappa$, the comparison of $Q(x)$ with the conditional variances up to a factor $1+o(1)$ and Lemma~\ref{lem:rademacher-max}, which allows us to work directly with {weighted Rademacher sums. Estimate \eqref{eq:covinput}, which is the key ingredient in the proof of the lower bound, is established in Section~\ref{sect:proofcov}. Its proof relies on several auxiliary number-theoretic estimates established in Section~\ref{sect:proofnovel}.}

{We expect that the methods developed here, particularly the uniform interpolation estimate and the selection of weakly correlated sums through estimates for random Euler products, can be adapted to other random sums with a similar conditionally independent structure, such as restricted sums of Steinhaus random multiplicative functions.}

\paragraph{The full sum.}
In the accompanying paper \cite{Buraczewski+Iksanov:accompanying} we prove the counterpart of Theorem~\ref{thm:main} for the full sum $M(x)$. The normalization there uses its full predictable quadratic variation
$$\sum_{p\le x}\Big(\sum_{\substack{m\le x/p\\P(m)<p}}f(m)\Big)^2$$ with respect to the filtration obtained by revealing the prime signs. We also show that this full quadratic variation has no deterministic
asymptotic equivalent, even in probability along a deterministic subsequence. Furthermore, we prove in \cite{Buraczewski+Iksanov:accompanying} that the a.s.\ upper limit of $$\frac{|M(n)|}{n^{1/2}(\log\log n)^{1/4}}$$ cannot be a finite deterministic constant. The restricted sum theorem isolates the conditional independence and covariance analysis needed for this approach. The full sum lower bound requires an additional analysis of the
contributions from smaller prime factors, since its quadratic variation is larger than $Q(x)$.

\section{Proof of Theorem~\ref{thm:main} \texorpdfstring{--}{--} the upper bound}\label{sect:upper}
The first part of the paper is concerned with the upper bound for $\limsup$.
\begin{proposition}\label{prop:upper bound}
We have
\begin{equation}\label{eq:upper sup}
\limsup_{x\to\infty}\frac{|M^{>}(x)|}{(2Q(x)\log\log Q(x))^{1/2}}\leq 1\quad\text{{\rm a.s.}}
\end{equation}
\end{proposition}

We begin with the main idea of the proof of Proposition \ref{prop:upper bound}. We work on logarithmic blocks $[\eee^T, \eee^{2T}]$, where $T=2^r$ for $r\in\mn$. We also put $\mn_0:=\mn\cup\{0\}$. In what follows, we refer to such $T$ as dyadic $T$. On each block we choose a geometric grid with logarithmic spacing $T^{-\alpha}$. At a fixed grid point $z$, conditioning on
$\F_{\sqrt z}$ makes the coefficients $M(z/p)$ in \eqref{eq:martrepr} fixed, while the signs $f(p)$ for $p>z^{1/2}$ remain independent. The conditional exponential bound of
Lemma~\ref{lem:expo}, followed by a union bound over the grid, gives the desired upper estimate there, see Proposition~\ref{prop:grid}.

The main difficulty is extending this estimate to each real $x$ in the block. We choose a nearby grid point $z\le x$ and control both $M^{>}(x)-M^{>}(z)$ and the distance between their coefficient vectors. For $p\in\Primes$ and $x\ge 2$, put $a_p(x):=M(x/p)\1_{\{p>x^{1/2}\}}$. Then $$M^{>}(x)=\sum_p f(p)a_p(x)\quad\text{and}\quad Q(x)=\sum_p(a_p(x))^2.$$ Further, we put
\begin{equation}\label{eq:distance}
\mathbf a(x):=(a_p(x))_{p\in\Primes}\quad\text{and}\quad d(x,z):=\norm{\mathbf a(x)-\mathbf a(z)}_2=\Big(\sum_p (a_p(x)-a_p(z))^2\Big)^{1/2}
\end{equation}
for $x,z\geq 2$. For fixed $x$, the vector $\mathbf a(x)$ has a finite support, since $a_p(x)=0$ when $p>x$. The large prime decomposition also gives
\begin{equation}\label{eq:avector}
\mathbf a(x)= \sum_{\substack{2\le n\le x\\ n\text{ squarefree}\\ P(n)>x^{1/2}}}f\!\left(\frac{n}{P(n)}\right)e_{P(n)},
\end{equation}
where $e_p$ is the standard unit vector with $1$ in the coordinate indexed by the prime $p$ and $0$ in all other coordinates.
In particular, $Q(x)=\norm{\mathbf a(x)}_2^2$, and an application of the triangle inequality yields
$\big|(Q(x))^{1/2}-(Q(z))^{1/2}\big|\le d(x,z)$. A finer auxiliary grid then makes the control uniform in $x$, see Proposition~\ref{prop:interp}. Finally, the lower bound for $Q(x)$ in Lemma~\ref{lem:lowerclock} shows that the interpolation errors are negligible on the LIL scale.

The structure of the argument is that of the classical proofs of the upper half of the LIL: exponential bounds along a geometric sequence, followed by interpolation. The same structure is used by Mastrostefano \cite{Mastrostefano:2022} for the restricted sum at the deterministic scale $x^{1/2}(\log\log x)^{1/4+\varepsilon}$. Two features are specific to the random normalization.
{Both $M^{>}$ and $Q$ must be interpolated, which necessitates controlling the distance $d(x,z)$. Moreover, the interpolation errors must be small compared with $(Q(x)\log\log Q(x))^{1/2}$, which may be much smaller than $x^{1/2}$.}

Fix $0<\beta<1/2$ and put $\alpha=\beta/2$. For each dyadic $T$, define
\begin{equation}\label{eq:coarse}
C_T:=\{\eee^{T+kT^{-\alpha}}:~0\le k\le\lfloor T^{1+\alpha} \rfloor\}\cup\{\eee^{2T}\}.
\end{equation}
Thus $|C_T|=O(T^{1+\alpha})$, and each $x\in[\eee^T, \eee^{2T}]$ has a grid point below it at logarithmic distance at most $T^{-\alpha}$. Here and hereafter, for a finite discrete set $A$, we write interchangeably $|A|$ or $\#A$ for the number of its elements.
\begin{proposition}\label{prop:grid}
For each fixed $0<\beta<1/2$, almost surely, for all sufficiently large dyadic $T$ and each $z\in C_T$,
\begin{equation*}
|M^{>}(z)|\leq (1+\beta)(2Q(z)\log\log Q(z))^{1/2}.
\end{equation*}
\end{proposition}
For a proof of Proposition \ref{prop:grid}, we need an observation regarding orthogonality and two auxiliary results. The sequence $(f(n))_{n\geq 1}$ satisfies
\begin{equation}\label{eq:orth}
\me[f(n)f(m)]=(\mu(n))^2\1_{\{n=m\}}.
\end{equation}
This shows that the variables $f(1)$, $f(2)$, $f(3),\ldots$ are orthogonal. Hence
\begin{equation}\label{eq:orthogon}
\me[(M(x))^2]=\sum_{1\leq n\leq x}\me [(f(n))^2]=\sum_{1\leq n\leq x}(\mu(n))^2\leq x,\quad x\geq 0
\end{equation}
and thereupon
\begin{equation*}
\me [Q(x)]=\sum_{x^{1/2}<p\le x}\sum_{m\le x/p}(\mu(m))^2=\#\{n\le x:~n\text{ squarefree},\ P(n)>x^{1/2}\}\le x,\quad x\geq 0.
\end{equation*}
Here, the second equality follows from the unique representation $n=P(n)m$. Indeed, if $x^{1/2}<p\leq x$ and $m\leq x/p$, then $m<x^{1/2}<p$. Thus $p$ does not divide $m$, and $pm$ is squarefree if, and only if, $m$ is squarefree. Moreover, $P(pm)=p$.

For later use, we note that, by the Rademacher--Menshov maximal inequality for {partial sums of} orthogonal random variables (see, for instance, Theorem 2.3.1 on p.~18 in~\cite{Stout:1974}),
$$\me\big[\max_{u\leq v}(M(u))^2\big]=\me\big[\max_{1\leq u\leq \lfloor v\rfloor}(M(u))^2\big]\leq v\Big(\frac{\log (4v)}{\log 2}\Big)^2.$$ Furthermore, by Markov's inequality,
\begin{equation}\label{eq:maxM}
\mmp\{\max_{u\leq b_r}|M(u)|>b_r^{2/3}\}=O\big(b_r^{-1/3}(\log b_r)^2\big)=o(1),\quad r\to\infty
\end{equation}
for any sequence $(b_r)_{r\geq 1}$ diverging to $+\infty$.

Recall that, with $y(s)>0$ for all sufficiently large $s$, $x(s)=\Theta(y(s))$ as $s\to\infty$ means that $|x(s)|/y(s)$ is bounded and bounded away from zero for all sufficiently large $s$.
\begin{lemma}\label{lem:clockelementary}
For all sufficiently large $x$,
\begin{equation}\label{eq:clockelementary}
\frac{x}{3\log x}\le Q(x)\le 2x^{3/2}.
\end{equation}
Consequently, uniformly in $\eee^T \leq x \leq \eee^{2T}$,
\begin{equation}\label{eq:logclock}
\log Q(x)=\Theta(T)\quad\text{and}\quad \log\log Q(x)=\log T+O(1)=\log\log x+O(1).
\end{equation}
\end{lemma}
\begin{proof}
For $x>4$, we have $(x/2, x]\subset (x^{1/2}, x]$, whence $Q(x)\geq \sum_{x/2<p\leq x}(M(\lfloor x/p\rfloor))^2$. Under the last sum, $\lfloor x/p\rfloor=1$ and consequently $(M(\lfloor x/p\rfloor))^2=(f(1))^2=1$. With $\pi(x)$ denoting the number of prime numbers that do not exceed $x$, this proves that
\begin{equation*}
Q(x)\geq \pi(x)-\pi(x/2)\geq x/(3\log x)
\end{equation*}
for all sufficiently large $x$. Here, the last inequality is secured by the prime number theorem. Since $|M(u)|\le u$ for $u\ge 1$, $$Q(x)\le x^2\sum_{p>x^{1/2}}p^{-2} \le x^2\sum_{n>x^{1/2}}n^{-2}\leq 2x^{3/2}$$ for all sufficiently large $x$. The last part of the lemma follows from \eqref{eq:clockelementary}.
\end{proof}

For each prime $p$, put $\mathcal F_{p^-}:=\sigma(f(q):~q\in\mathcal P,~q<p)$.

\begin{lemma}\label{lem:expo}
For each fixed sufficiently large $x$ and any positive $\F_{x^{1/2}}$-measurable random variable $v$,
\begin{equation}\label{eq:conditional-tail}
\PP_{x^{1/2}}\big\{|M^{>}(x)|\ge v\big\} \le 2\exp\big(-v^2/(2Q(x))\big)\quad\text{almost surely}.
\end{equation}
Also, if $(b_p)_p$ is a family of real random variables such that $b_p=0$ for all primes $p$ outside a deterministic
finite set $\mathcal P_1$, with $b_p$ measurable with respect to $\F_{p^-}$ for $p\in\mathcal P_1$, then, for deterministic $t,v>0$,
\begin{equation}\label{eq:predictable-tail}
\PP\Big\{\Big|\sum_p f(p)b_p\Big|\geq t,~ \sum_p b_p^2\le v\Big\} \le 2\exp\big(-t^2/(2v)\big).
\end{equation}
\end{lemma}
\begin{proof}
{Inequality~\eqref{eq:conditional-tail} follows from Hoeffding's inequality, see, for instance, Lemma~2.3 in~\cite{Aymone etal:2020}. Inequality~\eqref{eq:predictable-tail} follows from Theorem~6.1 in~\cite{delaPena:1999}, see also Theorem~1.3 and formula (1.4) in~\cite{Bercu+Touati:2008}. Both inequalities can be derived using the standard supermartingale argument underlying Hoeffding's inequality. For completeness, we provide the details.}

Noting that, conditionally on  $\F_{x^{1/2}}$, the variables $f(p)$ for $p>x^{1/2}$ are independent and that $v$ and the coefficients $M(x/p)$ are $\F_{x^{1/2}}$-measurable, hence fixed upon conditioning, an application of the Hoeffding inequality yields $$\mmp_{x^{1/2}}\big\{M^{>}(x)\geq v\big\} \leq \exp\big(-v^2/(2Q(x))\big)\quad\text{a.s.}$$ Applying the same argument to $-M^{>}(x)$ and then adding the two conditional bounds we arrive at \eqref{eq:conditional-tail}.

To prove \eqref{eq:predictable-tail}, fix $\theta>0$. Since $b_p$ is $\F_{p^-}$-measurable and $f(p)$ is independent of $\F_{p^-}$, we have $$\me_{p^-}\big[\exp\big(\theta f(p)b_p-\theta^2b_p^2/2\big)\big]=\exp(-\theta^2b_p^2/2)\cosh(\theta b_p)\leq 1,\quad p\in\mathcal P_1.$$ Here, $\me_{p^-}$ denotes the conditional expectation given $\F_{p^-}$. With this at hand, successive conditioning gives $$\me\Big[\exp\Big(\theta\sum_p f(p)b_p-\frac{\theta^2}{2}\sum_p b_p^2\Big)\Big]\leq 1.$$ Consequently, by Markov's inequality,
\begin{multline*}
\mmp\Big\{\sum_p f(p)b_p\geq t,~ \sum_p b_p^2\leq v\Big\}\leq \mmp\Big\{\exp\Big(\theta \sum_p f(p)b_p-(\theta^2/2)\sum_p b_p^2\Big)\geq \exp(\theta t-\theta^2v/2)\Big\}\\\leq \exp(-\theta t+\theta^2v/2)
\me\Big[\exp\Big(\theta \sum_p f(p)b_p-(\theta^2/2)\sum_p b_p^2\Big)\Big]\leq \exp(-\theta t+\theta^2v/2).
\end{multline*}
Choosing $\theta=t/v$ we obtain the upper bound $\exp(-t^2/(2v))$. Applying the same argument to $(-b_p)_p$ and adding the resulting bounds proves \eqref{eq:predictable-tail}.
\end{proof}
\begin{proof}[Proof of Proposition~\ref{prop:grid}]
For all sufficiently large $T$ and each $z\in C_T$, the variable $(1+\beta)(2Q(z)\log\log Q(z))^{1/2}$ is positive and $\F_{z^{1/2}}$-measurable. Using \eqref{eq:conditional-tail} for the first inequality and \eqref{eq:logclock} for the second, we infer $$\mmp_{z^{1/2}}\big\{|M^{>}(z)|>(1+\beta)(2Q(z)\log\log Q(z))^{1/2}\big\} \leq 2 (\log Q(z))^{-(1+\beta)^2}\leq C_\beta T^{-(1+\beta)^2}$$
uniformly for $z\in C_T$, for a deterministic constant $C_\beta>0$. Taking expectations and then a union bound over the deterministic grid $C_T$ yields
$$\mmp\big\{\exists z\in C_T:~ |M^{>}(z)|>(1+\beta)(2Q(z)\log\log Q(z))^{1/2}\big\}\le C_\beta' T^{1+\alpha-(1+\beta)^2}.$$ The exponent is negative since $\alpha=\beta/2<(1+\beta)^2-1$. For $T=2^r$, these probabilities are summable in $r$. Invoking the direct part of the Borel-Cantelli lemma completes the proof.
\end{proof}

Lemma~\ref{lem:lowerclock} bounds $Q$ uniformly from below. Its proof is given, together with that of formula \eqref{eq:var v+w}, in Section~\ref{sect:aux}.
\begin{lemma}\label{lem:lowerclock}
For each fixed $b>0$, almost surely, for all sufficiently large dyadic $T$, simultaneously for all $x\in[\eee^T, \eee^{2T}]$,
\begin{equation*}
Q(x)\ge xT^{-b}.
\end{equation*}
\end{lemma}

The following proposition provides the uniform interpolation estimate. We postpone its proof to Section \ref{proof:prop2.1}.
\begin{proposition}\label{prop:interp}
Fix $\alpha\in (0,1/4)$. Almost surely, for all sufficiently large dyadic $T$ and each $x\in[\eee^T, \eee^{2T}]$, there is a $z=z_T(x)\in C_T$ such that $z\le x$, $\log(x/z)\leq 2T^{-\alpha}$, and
\begin{equation*}
 |M^{>}(x)-M^{>}(z)|+d(x,z)\leq 2x^{1/2}\,T^{-\alpha/12}.
\end{equation*}
\end{proposition}

We are ready to prove the main result of the first part of the paper.
\begin{proof}[Proof of Proposition \ref{prop:upper bound}]
Fix $\alpha\in (0,1/4)$ and put $\beta=2\alpha$. Given $x\in [\eee^T, \eee^{2T}]$, let $z=z_T(x)$ be as in Proposition~\ref{prop:interp}. Throughout the remainder of the proof, $T$ ranges over dyadic
values. All inequalities hold a.s.\ for all sufficiently large $T$, and all asymptotic relations hold a.s.\ as $T\to\infty$, uniformly in $x\in[\eee^T,\eee^{2T}]$.

Using the inequality which has already appeared in the paragraph following \eqref{eq:avector} we obtain $$\big|(Q(x))^{1/2}-(Q(z))^{1/2}\big| \le d(x,z)\leq 2x^{1/2}T^{-\alpha/12}.$$ Applying  Lemma~\ref{lem:lowerclock} with $b=\alpha/12$ and Proposition \ref{prop:interp}, we infer $$\Big|\Big(\frac{Q(z)}{Q(x)}\Big)^{1/2}-1\Big|=\frac{\big|(Q(z))^{1/2}-(Q(x))^{1/2}\big|}{(Q(x))^{1/2}}\leq \frac{d(x,z)}{(Q(x))^{1/2}}\leq
\frac{2x^{1/2}T^{-\alpha/12}}{x^{1/2}T^{-\alpha/24}}=2T^{-\alpha/24}~\to~0.$$ By \eqref{eq:logclock}, $\log\log Q(z)/\log\log Q(x)=1+o(1)$.
By Lemma~\ref{lem:lowerclock}, with $b=\alpha/12$, we have $Q(x)\geq xT^{-\alpha/12}$. In particular, $\log Q(x)\geq \log x- (\alpha/12)\log T \geq T/2$, whence $\log\log Q(x)\geq \log(T/2)\geq (\log T)/2$ for all sufficiently large $T$. Using these inequalities to bound the denominator and Proposition \ref{prop:interp} to bound the numerator, we obtain $$\frac{|M^{>}(x)-M^{>}(z)|}{(2Q(x)\log\log Q(x))^{1/2}} \leq \frac{2x^{1/2}T^{-\alpha/12}}{x^{1/2}T^{-\alpha/24}(\log T)^{1/2}}=2\frac{T^{-\alpha/24}}{(\log T)^{1/2}}~\to~0.$$ Finally, writing, for $x\in [\eee^T,\eee^{2T}]$,
\begin{multline*}
\frac{|M^{>}(x)|}{(2Q(x)\log\log Q(x))^{1/2}}\\\leq \frac{|M^{>}(z)|}{(2Q(z)\log\log Q(z))^{1/2}}\Big(\frac{Q(z)\log\log Q(z)}{Q(x)\log\log Q(x)}\Big)^{1/2}+\frac{|M^{>}(x)-M^{>}(z)|}{(2Q(x)\log\log Q(x))^{1/2}}
\end{multline*}
and applying the preceding asymptotic relations and Proposition~\ref{prop:grid} yields $$\limsup_{r\to\infty}\sup_{x\in[\eee^{2^r},\eee^{2^{r+1}}]}\frac{|M^{>}(x)|}{(2Q(x)\log\log Q(x))^{1/2}}\leq 1+\beta\quad\text{a.s.}$$ Since the intervals $\big[\eee^{2^r}, \eee^{2^{r+1}}\big]$ cover all sufficiently large $x$, we obtain $$\limsup_{x\to\infty}\frac{|M^{>}(x)|}{(2Q(x)\log\log Q(x))^{1/2}}\leq1+\beta\quad\text{a.s.}$$ Taking a countable sequence $\alpha\to 0+$, with $\beta=2\alpha$, completes the proof.
\end{proof}

\section{Proof of Proposition~\ref{prop:interp}}\label{proof:prop2.1}

{\sc Plan of the proof.} Besides the grid $C_T$, we use a much finer grid $D_T$. Given $x$, let $y$ be the largest point of $D_T$ with $y\leq x$, and let $z$ be the largest point of $C_T$ with $z\leq y$. The passage from $y$ to $x$ takes place within one short cell of $D_T$ and is controlled by the fourth-moment bound of Lemma~\ref{lem:fourth}.
{Since $M^{>}(y)-M^{>}(z)=\sum_p f(p)(a_p(y)-a_p(z))$ has predictable coefficients, the passage from $z$ to $y$ is controlled by the exponential bound \eqref{eq:predictable-tail}. The sum of the squared coefficients, $(d(z,y))^2$, is estimated in Lemmas~\ref{lem:distance}, \ref{lem:shifts} and~\ref{lem:R}.}

We start with a number of auxiliary results.
Throughout this section $T=2^r$ tends to $\infty$ through dyadic values and $\alpha\in (0, 1/4)$ is fixed. Put $N=\lfloor \eee^{2T}\rfloor$ and $\delta=T^{-250}$.

For $u\ge 0$ and $s\geq 0$, put
\begin{equation*}
M_N(u)=\sum_{n\le \min(u,N)}f(n)\quad\text{and}\quad A_N(s)=\int_0^\infty \frac{(M_N(\eee^su)-M_N(u))^2}{u^2} \dd u.
\end{equation*}
Also, for $z\in [\eee^T, \eee^{2T}]$, put
\begin{equation*}
R_T(z)=\sum_{p>z^{1/2}}\frac1{\delta p} \int_p^{(1+\delta)p}(M(z/p)-M(z/t))^2\dd t.
\end{equation*}
This sum is finite because $M(z/p)=M(z/t)=0$ for $t>p>z$.

Put $\pi(x):=\#\{p:~p\leq x\}$ for $x\geq 0$. As a preparation for the proof of Lemma \ref{lem:distance} given below, we note that, by formula (6.14) on p.~179 in~ \cite{MV},
\begin{equation}\label{eq:pix}
\pi(x)=\int_2^x\frac{{\rm d}y}{\log y}+R(x),\quad x\geq 2,
\end{equation}
where, for some positive constants $c$ and $c_1$,
\begin{equation}\label{eq:Restim}
|R(x)|\leq c_1x\eee^{-c(\log x)^{1/2}},\quad x\geq 2.
\end{equation}

\begin{lemma}\label{lem:distance}
Almost surely, for all sufficiently large dyadic $T$, and each $\eee^T\le x\le y\le \eee^{2T}$ with $\theta=\log(y/x)\le T^{-\alpha}$,
\begin{equation}\label{eq:distancebound}
\frac{(d(x,y))^2}{x}\leq 10\frac{A_N(\theta)}{T}+4\frac{R_T(x)}{x}+8\frac{R_T(y)}{y}+C(\omega)\frac{(\log T)^6}{T^{1+\alpha}}.
\end{equation}
Here, $C(\omega)$ is an almost surely finite random constant independent of $x$, $y$ and $T$.
\end{lemma}
\begin{proof}
Fix $\lambda>0$. We claim that, for all sufficiently large $T$, all $v\in [\eee^{T/2}/2, 2\eee^{2T}]$ and all $\ell\in [0, vT^{-\lambda}]$,
\begin{equation}\label{eq:primecount}
\#\{p:~v\leq p\leq v+\ell\} \leq 5 vT^{-\lambda-1}.
\end{equation}
Indeed, for all sufficiently large $T$, we have $\log v\geq T/3$ and $v+\ell\leq 2v$. Hence, by \eqref{eq:Restim},
\begin{multline*}
\#\{p:~v< p\leq v+\ell\}=\pi(v+\ell)-\pi(v)=\int_v^{v+\ell}\frac{\dd y}{\log y}+R(v+\ell)-R(v)\\\leq \frac{\ell}{\log v}+c_1(v+\ell)\eee^{-c(\log(v+\ell))^{1/2}}+c_1v\eee^{-c(\log v)^{1/2}}\\\leq \frac{3\ell}{T}+3c_1v\eee^{-c'T^{1/2}}\leq 3v T^{-\lambda-1}+3c_1v\eee^{-c'T^{1/2}},
\end{multline*}
where $c':=3^{-1/2}c$. Since $\lim_{T\to\infty} T^{\lambda+1}\eee^{-c'T^{1/2}}=0$, the last expression is at most $4vT^{-\lambda-1}$ for all sufficiently large $T$. Including the possible prime at the left endpoint $v$ adds
at most one. As $vT^{-\lambda-1}\geq \eee^{T/2}T^{-\lambda-1}/2 \to \infty$ as $T\to\infty$, this proves \eqref{eq:primecount}.

Write, for $2\leq x\leq y$,
\begin{equation*}
(d(x,y))^2=A(x,y)+B(x,y),
\end{equation*}
where $$A(x,y)=\sum_{p>y^{1/2}}|M(y/p)-M(x/p)|^2\quad\text{and}\quad B(x,y)=\sum_{x^{1/2}<p\leq y^{1/2}} |M(x/p)|^2.$$ We first estimate $A(x,y)$. Using twice the inequality $(a+b)^2\leq 2(a^2+b^2)$, $a,b\in\mr$, we obtain, for $t\in [p, (1+\delta)p]$, $$(M(y/p)-M(x/p))^2\leq 2(M(y/t)-M(x/t))^2+4(M(y/p)-M(y/t))^2+4(M(x/p)-M(x/t))^2.$$ Now we divide the latter inequality by $\delta p$, integrate it over $t\in [p, (1+\delta)p]$ and then sum over $p>y^{1/2}$. This gives
\begin{multline*}
A(x,y)\leq 2\sum_{p>y^{1/2}}(\delta p)^{-1}\int_p^{(1+\delta)p}(M(y/t)-M(x/t))^2\,{\rm d}t+4R_T(y)+4R_T(x)\\\leq 2\int_{y^{1/2}}^y (M(y/t)-M(x/t))^2 \sum_{t/(1+\delta)\le p\le t}
(\delta p)^{-1}\,{\rm d}t+4R_T(y)+4R_T(x)
\end{multline*}
having utilized for the last inequality Fubini's theorem, the fact that the inequalities $p>y^{1/2}$ and $t\in [p,(1+\delta)p]$ entail $t>y^{1/2}$ and $p\in [t/(1+\delta),t]$, and also $M(y/t)-M(x/t)=0$ for $t>y$. Further, for $t\in [y^{1/2},y]$,
$$\sum_{t/(1+\delta)\le p\le t}(\delta p)^{-1}\leq (1+\delta)(\delta t)^{-1}\#\{p:~t/(1+\delta)\leq p\leq t\}\leq 5/T.$$ The last inequality is justified by \eqref{eq:primecount}, with $\lambda=250$, $v=t/(1+\delta)$ and $\ell=\delta t/(1+\delta)=\delta v=vT^{-250}$. The formula applies because $v\in [\eee^{T/2}/2, \eee^{2T}]$ for all sufficiently large $T$. Substituting this estimate into the preceding bound for $A(x,y)$ gives
\begin{equation*}
A(x,y)\le 10T^{-1}\int_{y^{1/2}}^y (M(y/t)-M(x/t))^2\,{\rm d}t+4R_T(x)+4R_T(y).
\end{equation*}
Both arguments of $M$ in this integral are at most $y^{1/2}\leq \eee^T<N$. Hence, $M$ may be replaced by $M_N$. Changing the variable $u=x/t$ and enlarging the domain of integration, we obtain
\begin{multline*}
\int_{y^{1/2}}^y (M(y/t)-M(x/t))^2\,\dd t=x\int_{x/y}^{x/y^{1/2}}\frac{(M_N(\eee^\theta u)-M_N(u))^2}{u^2}\,\dd u\\ \leq x \int_0^\infty \frac{(M_N(\eee^\theta u)-M_N(u))^2}{u^2}\,\dd u=x A_N(\theta),
\end{multline*}
where $\theta=\log(y/x)$. Thus,
\begin{equation}\label{eq:A-bound}
A(x,y)\leq 10T^{-1}x A_N(\theta)+4R_T(x)+4R_T(y).
\end{equation}

Now we work towards estimating $B(x,y)$. Relation \eqref{eq:wu}, with $\varepsilon=1$, implies that, a.s., after enlarging a finite random constant if necessary,
we have
\begin{equation}\label{eq:LTW}
|M(u)|\leq C_1(\omega)u^{1/2}\,(1+\log\log(u+3))^3, \quad u\geq 1.
\end{equation}
The inequality $x/p \leq x^{1/2}\leq \eee^T$ holds whenever $x^{1/2}<p\le y^{1/2}$. Consequently, there exists an absolute constant $c_1>0$ such that, for all sufficiently large $T$, $$1+\log\log(x/p+3)\leq 1+\log\log(\eee^T+3)\leq {c_1}\log T.$$

Observe that $x/p\geq 1$ as a consequence of $p\leq y^{1/2}\leq\eee^T\leq x$. Using now \eqref{eq:LTW} we infer $(M(x/p))^2 \leq C_2(\omega) p^{-1}x (\log T)^6$ a.s., where $C_2(\omega):={c_1^6}(C_1(\omega))^2$ and thereupon
\begin{equation}\label{eq:B-first}
B(x,y)\leq  C_2(\omega)x(\log T)^6\sum_{x^{1/2}<p\le y^{1/2}}p^{-1}.
\end{equation}
Further, $$\sum_{x^{1/2}<p\le y^{1/2}}p^{-1}\leq \sum_{x^{1/2}<p\le \eee^{T^{-\alpha}/2} x^{1/2}}p^{-1}\leq x^{-1/2}\#\{p:~x^{1/2}< p\leq \eee^{T^{-\alpha}/2} x^{1/2}\}\leq 5T^{-1-\alpha}.$$
Here, the first inequality holds because $y/x\leq \exp(T^{-\alpha})$ entails $y^{1/2}\leq \exp(T^{-\alpha}/2)x^{1/2}$. The last inequality is justified by \eqref{eq:primecount}, with $\lambda=\alpha$, $v=x^{1/2}$ and $\ell=x^{1/2}\big(\exp(T^{-\alpha}/2)-1\big)$. The formula applies because $v\in[\eee^{T/2},\eee^T]$ and, for all sufficiently large $T$, $\exp(T^{-\alpha}/2)-1\leq T^{-\alpha}$, whence $\ell\leq vT^{-\alpha}$. Substituting the estimate for the prime sum into \eqref{eq:B-first}, we obtain
\begin{equation}\label{eq:B-bound}
B(x,y)\leq C(\omega)x T^{-1-\alpha}(\log T)^6,
\end{equation}
where $C(\omega):=5C_2(\omega)$. Finally, combining \eqref{eq:A-bound}, \eqref{eq:B-bound} and $$\frac{R_T(y)}{x}=\frac{y}{x}\frac{R_T(y)}{y}\leq \eee^{T^{-\alpha}}\frac{R_T(y)}{y}\leq 2\frac{R_T(y)}{y}$$ for sufficiently large $T$, we arrive at \eqref{eq:distancebound}.
\end{proof}
\begin{lemma}\label{lem:shifts}
Almost surely, for all sufficiently large dyadic $T$,
\begin{equation}\label{eq:shiftsas}
\sup_{0\le\theta\le T^{-\alpha}}\frac{A_N(\theta)}{T}\leq T^{-\alpha/2},
\end{equation}
where $\alpha\in (0,1/4)$ is fixed.
\end{lemma}
\begin{proof}
For $u,s\geq 0$, put $D_s(u):=M_N(\eee^su)-M_N(u)$. If $0\le\theta\le T^{-\alpha}$ and $T^{-\alpha} \leq t \leq 2T^{-\alpha}$, then $D_\theta(u)=D_t(u)-D_{t-\theta}(\eee^\theta u)$. The inequality $(a-b)^2\leq 2a^2+2b^2$, $a,b\in\mr$, in combination with the change of variable $v=\eee^\theta u$, gives $$A_N(\theta)\leq 2 A_N(t)+2\eee^\theta A_N(t-\theta).$$ Integrating this inequality over $t\in[T^{-\alpha},2T^{-\alpha}]$ and then multiplying both sides by $T^\alpha$, we obtain, simultaneously for all $\theta\in [0,T^{-\alpha}]$, $$A_N(\theta)\leq 2T^\alpha\int_{T^{-\alpha}}^{2T^{-\alpha}}A_N(t)\,{\rm d}t+2\eee^{\theta}T^\alpha\int_{T^{-\alpha}-\theta}^{2T^{-\alpha}-\theta}A_N(s)\,{\rm d}s\leq c_\ast T^\alpha \int_0^{2T^{-\alpha}}A_N(s)\,{\rm d}s$$ and thereupon
\begin{equation}\label{eq:expected-sup-shifts}
\me\big[\sup_{0\le\theta\le T^{-\alpha}}A_N(\theta)\big]\leq c_\ast T^\alpha\int_0^{2T^{-\alpha}}\me [A_N(s)]\,{\rm d}s
\end{equation}
for all sufficiently large $T$, where $c_\ast:=2(1+\eee^{1/2})$. We have used that both $[T^{-\alpha}, 2T^{-\alpha}]$ and $[T^{-\alpha}-\theta, 2T^{-\alpha}-\theta]$ are contained in $[0,2T^{-\alpha}]$, and that $\eee^\theta\le \eee^{1/2}$ for all sufficiently large $T$.

To estimate the expectation, we start with a representation $$D_s(u)=M_N(\eee^su)-M_N(u)=\sum_{n\leq N}f(n)\1_{\{u<n\leq\eee^su\}}=\sum_{n\leq N}f(n)\1_{\{n\eee^{-s}\leq u<n\}}.$$ Invoking \eqref{eq:orth} we then obtain $$\me[(D_s(u))^2]=\sum_{n,m\leq N}\me[f(n)f(m)]\1_{\{n\eee^{-s}\leq u<n\}} \1_{\{m\eee^{-s}\leq u<m\}}=\sum_{n\leq N}(\mu(n))^2 \1_{\{n\eee^{-s}\leq u<n\}}.$$ In view of Fubini's theorem and
$$\sum_{n=1}^m n^{-1}=1+\sum_{n=2}^m n^{-1}\leq 1+\int_1^m u^{-1}\dd u=1+\log m,\quad m\in\mn,$$ we conclude that
\begin{multline*}
\me[A_N(s)]=\int_0^\infty \frac{\me[D_s(u)^2]}{u^2}\,\dd u=\sum_{n\leq N}(\mu(n))^2\int_{n\eee^{-s}}^n\frac{\dd u}{u^2}=(\eee^s-1)\sum_{n\leq N}\frac{\mu^2(n)}n \\\leq (1+\log N)(\eee^s-1)\leq 3T(\eee^s-1),
\end{multline*}
where the last inequality follows from $N=\lfloor\eee^{2T}\rfloor$ and $1+\log N\leq 1+2T\leq 3T=3\cdot 2^r$ for $r\in\mn$. Substituting this bound into \eqref{eq:expected-sup-shifts} yields
\begin{multline*}
T^{-1}\me\Big[\sup_{0\leq\theta\leq T^{-\alpha}}A_N(\theta)\Big]\leq 3c_\ast T^\alpha \int_0^{2T^{-\alpha}}(\eee^s-1)\,\dd s =3c_\ast T^\alpha \big(\eee^{2T^{-\alpha}}-1-2T^{-\alpha}\big)\leq 9c_\ast T^{-\alpha}
\end{multline*}
for all sufficiently large $T$. By Markov's inequality, $$\mmp\Big\{T^{-1}\sup_{0\leq\theta\leq T^{-\alpha}}A_N(\theta)>T^{-\alpha/2}\Big\}\leq 9c_\ast T^{-\alpha/2}.$$ Since $T=2^r$, these probabilities are summable in $r$. An application of the direct part of the Borel-Cantelli lemma completes the proof of \eqref{eq:shiftsas}.
\end{proof}

\begin{lemma}\label{lem:R}
Let $G_T\subset [\eee^T, \eee^{2T}]$ be a deterministic set such that $|G_T|=O(T^{122})$ as $T\to\infty$. Then, almost surely, for all sufficiently
large dyadic $T$,
\begin{equation}\label{eq:Ras}
\max_{z\in G_T}\frac{R_T(z)}{z}\leq T^{-\alpha/2},
\end{equation}
where $\alpha\in (0,1/4)$ is fixed.
\end{lemma}
\begin{proof}
Fix a deterministic $z\in[\eee^T, \eee^{2T}]$. For $p\le t\le(1+\delta)p$, we obtain with the help of \eqref{eq:orth}
$$\me\big[(M(z/p)-M(z/t))^2\big]=\sum_{z/t<n \le z/p}(\mu(n))^2 \le \#\{n:~ ((1+\delta)p)^{-1}z<n\leq p^{-1}z\}.$$ This, in combination with Fubini's theorem, gives
\begin{multline*}
\me[R_T(z)]=\sum_{p>z^{1/2}}(\delta p)^{-1}\int_p^{(1+\delta)p}\me\big[(M(z/p)-M(z/t))^2\big]{\rm d}t\\\leq \sum_{p>z^{1/2}}\#\{n:~ (1+\delta)^{-1}z<np\leq z\}=\#\{m:~(1+\delta)^{-1}z<m\leq z,~ P(m)>z^{1/2}\}\\\leq \#\{m:~(1+\delta)^{-1}z <m\le z\}\leq \delta z+1,
\end{multline*}
whence $z^{-1}\me [R_T(z)]\leq \delta+\eee^{-T}$. To explain the last equality, observe that each pair $(p,n)$ counted in the prime sum gives an integer $m=pn$ satisfying $(1+\delta)^{-1}z<m\leq z$. Moreover, $p>z^{1/2}$ and $n\leq z/p<z^{1/2}<p$ ensure that $p$ is the unique prime divisor of $m$ exceeding $z^{1/2}$. Thus distinct counted pairs give distinct integers $m$. Conversely, each integer $m$ counted on the right-hand side of the equality gives the pair $(p,n)=(P(m),m/P(m))$ counted on the left.

Since $G_T$ is deterministic, Markov's inequality and a union bound yield
\begin{multline*}
\mmp\{\max_{z\in G_T}(z^{-1}R_T(z))>T^{-\alpha/2}\}\leq \sum_{z\in G_T}\mmp\{z^{-1}R_T(z)>T^{-\alpha/2}\}\leq \big(T^{-250}+\eee^{-T}\big)|G_T|T^{\alpha/2}\\=\big(T^{-250}+\eee^{-T}\big)O\big(T^{122+\alpha/2}\big).
\end{multline*}
These bounds are summable over $T=2^r$. An application of the direct part of the Borel-Cantelli lemma completes the proof of \eqref{eq:Ras}.
\end{proof}

Lemma \ref{eq:aux} is a special case of Lemma 2 on pp.~130--131 in~ \cite{Longnecker+Serfling:1977}, with $\nu=4$ and $\gamma=2$. Although their result is stated for real random variables, the same argument applies to random vectors, with the absolute value replaced by a norm.
\begin{lemma}\label{eq:aux}
For $m\in\mn$, let $Y_1,\ldots,Y_m$ be random scalars or random vectors and let $u_1,\ldots,u_m$ be nonnegative deterministic numbers. Suppose that, for each collection of consecutive indices $I=\{a,a+1,\ldots,b\}\subseteq\{1,\ldots,m\}$,
\begin{equation*}
\me\Big[\Big\|\sum_{j\in I}Y_j\Big\|^4\Big]\leq \Big(\sum_{j\in I}u_j\Big)^2.
\end{equation*}
Then
\begin{equation}\label{eq:max-fourth-short}
\me\Big[\max_{0\leq k\leq m}\Big\|\sum_{j\leq k}Y_j\Big\|^4\Big]\leq
C\Big(\sum_{j=1}^mu_j\Big)^2
\end{equation}
for an absolute constant $C>0$.
\end{lemma}

The proof of Lemma~\ref{lem:fourth} uses hypercontractive moment bounds, as in Lemmas~2.2 and~2.3 of \cite{Lau+Tenenbaum+Wu:2013}, together with the maximal inequality \eqref{eq:max-fourth-short}. Lemma~2.2 of \cite{Lau+Tenenbaum+Wu:2013} contains a fourth-moment bound for the increments $M(v)-M(u)$ of the full sum, of the same form as \eqref{eq:fourth}, while Lemma~2.3 provides an interpolation statement for $M$ analogous to Proposition~\ref{prop:interp}. For the restricted sum at the deterministic scale, the entering and leaving summands are controlled by similar tools in the proof of Theorem~1.1 of \cite{Mastrostefano:2022}. A related argument for full weighted Steinhaus sums appears in Lemma~2 of \cite{Hardy:2024}. Here, we treat $M^{>}$ and simultaneously control the coefficient distance $d$, accounting for both the appearance and removal of summands.
\begin{lemma}\label{lem:fourth}
There is an absolute constant $C>0$ such that, for all sufficiently large
$u$ and $u+u^{1/2} \leq v \leq 3u/2$,
\begin{equation}\label{eq:fourth}
\me\big[\sup_{u\le x\le v}\big(|M^{>}(x)-M^{>}(u)|+d(u,x)\big)^4\big] \leq Cu^{2/3} (v-u)^{4/3}(\log u)^{52/3}.
\end{equation}
\end{lemma}
\begin{proof}
We divide the proof into five steps.

\noindent {\sc Step 1}.
Recall that, for an integer $n\geq 2$, $P(n)$ denotes its largest prime factor.
Define $$\mathcal A:=\big\{n\geq 2:~n~\text{is squarefree and } n<(P(n))^2\big\}.$$ The set ${\mathcal A}$ consists precisely of the squarefree integers which can occur in (that is, give a nonzero contribution to) $M^{>}(x)$ for some $x$. Indeed, recalling that $f(n)=0$ unless $n$ is squarefree, the integer $n$ is present in $M^{>}(x)$ if, and only if, $n\leq x< (P(n))^2$. Consequently,
\begin{equation*}
M^{>}(x)=\sum_{n\in\mathcal A}f(n)\,\1_{\{n\leq x<(P(n))^2\}}.
\end{equation*}
For $n\in\mathcal A$ and $u\leq x$, $$\1_{\{n\leq x<(P(n))^2\}}-\1_{\{n\leq u<(P(n))^2\}}=\1_{\{u<n\leq x\}}-
\1_{\{u<{(P(n))^2}\leq x\}}$$ and thereupon
\begin{equation}\label{eq:S-entry-exit-short}
M^{>}(x)-M^{>}(u)=E(u,x)-L(u,x),
\end{equation}
where $$E(u,x):=\sum_{\substack{n\in\mathcal A\\u<n\leq x}}f(n)\quad\text{and}\quad L(u,x):=\sum_{\substack{n\in\mathcal A\\u<(P(n))^2\leq x}}f(n).$$

The coefficient vector $\mathbf a(x)$ admits the analogous representation, see \eqref{eq:distance} for the definition and \eqref{eq:avector} for the alternative representation:
\begin{equation}\label{eq:a-active-short}
\mathbf a(x)=\sum_{n\in\mathcal A}f\!\Big(\frac{n}{P(n)}\Big)e_{P(n)}
\1_{\{n\leq x<(P(n))^2\}}.
\end{equation}
Indeed, for fixed prime $p$, the integers $n\in\mathcal A$ satisfying
$P(n)=p$ have the form $n=pm$ with squarefree $m<p$. Their contribution to the coordinate indexed by $p$ of the right-hand side of \eqref{eq:a-active-short} is $$\1_{\{x<p^2\}}\sum_{m\le x/p}f(m)=\1_{\{p>x^{1/2}\}}M(x/p)=a_p(x).$$ This proves \eqref{eq:a-active-short}.

Finally,
\begin{equation}\label{eq:a-entry-exit-short}
\mathbf a(x)-\mathbf a(u)= E_{\mathbf a}(u,x)-L_{\mathbf a}(u,x),
\end{equation}
where $$E_{\mathbf a}(u,x):=\sum_{\substack{n\in\mathcal A\\u<n\le x}}
f\!\Big(\frac{n}{P(n)}\Big)e_{P(n)}\quad \mbox{and} \quad L_{\mathbf a}(u,x):=
\sum_{\substack{n\in\mathcal A\\u<(P(n))^2\le x}}f\!\Big(\frac{n}{P(n)}\Big)e_{P(n)}.$$

\noindent {\sc Step 2.}
Formula~(2.2) in \cite{Lau+Tenenbaum+Wu:2013}, applied with $m=4$, or equivalently Lemma~2.1 in \cite{Mastrostefano:2022}, applied with $m=2$ and using $d_3(n)=3^{\omega(n)}$ for squarefree $n$, yields the following bound. For each finitely supported family of real coefficients $(c_n)$ indexed by squarefree integers,
\begin{equation}\label{eq:fourth-moment-short}
\me\Big[\Big|\sum_n c_nf(n)\Big|^4\Big]\leq \Big(\sum_n 3^{\omega(n)}c_n^2\Big)^2,
\end{equation}
where $\omega(n)$ denotes the number of distinct prime factors of $n$. With some extra work, inequality \eqref{eq:fourth-moment-short} can also be derived from Theorem~3 on p.~376 in~\cite{Bonami}, applied with $p=2$, $q=4$ and $r=3^{-1/2}$, or from the $(2,4)$-Hypercontractivity Theorem in Section~9.2 of \cite{ODonnell2014}, see also Lemma~2 in~\cite{Halasz:1983}.

We shall also need a vector-valued version of \eqref{eq:fourth-moment-short}.  For each finitely supported family of real coefficients $c_n$ indexed by squarefree integers $n\in
\mathcal A$,
\begin{equation}\label{eq:vector-fourth-short}
\me\Big[\Big\|\sum_n c_n f\!\Big(\frac{n}{P(n)}\Big)e_{P(n)}\Big\|_2^4\Big]\leq \Big(\sum_n 3^{\omega(n)}c_n^2\Big)^2.
\end{equation}
To prove this, for each prime $p$, let $V_p$ denote the coordinate indexed by  $p$ of the random vector on the left-hand side of \eqref{eq:vector-fourth-short}, that is, $\sum_n c_n f\!(n/P(n))e_{P(n)}=\sum_p V_p e_p$. Explicitly, $$V_p=\sum_{n:~P(n)=p}c_n f(n/p).$$ Applying \eqref{eq:fourth-moment-short} to the coefficients $c_n$ appearing in the latter sum gives
\begin{equation}\label{eq:Vp-fourth}
\big(\me\big[|V_p|^4\big]\big)^{1/2}\leq \sum_{n:~P(n)=p}3^{\omega(n/p)}c_n^2=\sum_{n:~P(n)=p}3^{\omega(n)-1}c_n^2.
\end{equation}
We have used for the equality that, for a squarefree integer $n\geq 2$, $\omega(n/P(n))=\omega(n)-1$. Combining the Minkowski inequality with \eqref{eq:Vp-fourth}, we obtain
\begin{multline*}
\Big(\me \Big\|\sum_n c_n f\!\Big(\frac{n}{P(n)}\Big)e_{P(n)}\Big\|_2^4\Big)^{1/2}=
\Big(\me\Big[\Big(\sum_p |V_p|^2\Big)^2\Big]\Big)^{1/2}\\
\leq \sum_p \big(\me\big[|V_p|^4\big]\big)^{1/2} \leq \sum_p \sum_{n:~P(n)=p}
3^{\omega(n)-1}c_n^2=\sum_n 3^{\omega(n)-1}c_n^2.
\end{multline*}
Squaring both sides and using $1/3<1$ proves \eqref{eq:vector-fourth-short}.

\noindent {\sc Step 3}. Now we prove that
\begin{equation}\label{eq:entry-short}
\me \big[\sup_{u\leq x\leq v}|E(u,x)|^4\big]+\me\big[\sup_{u\leq x\leq v}\|E_{\mathbf a}(u,x)\|_2^4\big]\leq C_1u^{2/3} (v-u)^{4/3}(\log u)^{52/3}
\end{equation}
for an absolute constant $C_1>0$.

Let $d_k(n)$ denote the $k$-fold divisor function, that is, $$d_k(n):=\#\big\{(a_1,\ldots,a_k)\in\mn^k:a_1\cdot\ldots \cdot a_k=n\big\}.$$ If $n$ is squarefree, then $d_k(n)=k^{\omega(n)}$ because each prime factor of $n$ can be assigned independently to one of the $k$ factors. For each fixed integer $k\ge 2$, according to Exercise~18(c) on p.~43 in \cite{MV},
\begin{equation}\label{eq:divisor-short}
\sum_{n\leq z} d_k(n) \leq c_0 z(\log z)^{k-1},\quad z\ge 2
\end{equation}
for a constant $c_0=c_0(k)>0$.

Define $\mathcal E:=\{n\in\mathcal A:~u<n\leq v\}$ and observe that $|\mathcal E|\leq v-u+1 \leq 2(v-u)$. Combining H\"older's inequality with \eqref{eq:divisor-short}, we then obtain
\begin{multline}
\sum_{n\in\mathcal E}3^{\omega(n)}\leq |\mathcal E|^{2/3}\Big(\sum_{\substack{n\leq v\\n\ {\rm squarefree}}}27^{\omega(n)}\Big)^{1/3}\leq |\mathcal E|^{2/3}\Big(\sum_{n\leq v}d_{27}(n)\Big)^{1/3}\\\leq
c_1 u^{1/3} (v-u)^{2/3}(\log u)^{26/3},\label{eq:w}
\end{multline}
where $c_1:=2^9 3^{1/3}(c_0(27))^{1/3}$, having utilized $\log (3/2)+\log u\leq 2\log u$ for $u\geq 3/2$.

Write $\mathcal E=\{n_1,\ldots,n_m\}$, where $n_1<\ldots<n_m$. For $x\in[u,v]$, put $$k(x):=\#\{j\in\{1,\ldots,m\}:~n_j\leq x\}.$$ Then $$E(u,x)=\sum_{j=1}^{k(x)}f(n_j)~\text{and}~E_{\mathbf a}(u,x)=\sum_{j=1}^{k(x)}f(n_j/P(n_j))e_{P(n_j)},$$ where an empty sum is understood to be zero. Consequently,
\begin{align*}
\sup_{u\leq x\leq v}|E(u,x)|
&=\max_{0\leq k\leq m}\Big|\sum_{j=1}^k f(n_j)\Big|,\\
\sup_{u\leq x\leq v}\|E_{\mathbf a}(u,x)\|_2
&=\max_{0\leq k\leq m}\Big\|\sum_{j=1}^k f(n_j/P(n_j))e_{P(n_j)}\Big\|_2.
\end{align*}
Assume that $m\geq 1$. For $1\leq j\leq m$, put
$Y_j:=f(n_j)$ and $u_j:=3^{\omega(n_j)}$. By \eqref{eq:fourth-moment-short}, for all $1\leq a\leq b\leq m$, $$\me\Big[\Big|\sum_{j=a}^bY_j\Big|^4\Big]
\leq\Big(\sum_{j=a}^bu_j\Big)^2.$$ Hence, \eqref{eq:max-fourth-short} and \eqref{eq:w} give
\begin{multline*}
\me\Big[\sup_{u\leq x\leq v}|E(u,x)|^4\Big]
=\me\Big[\max_{0\leq k\leq m}
\Big|\sum_{j=1}^kY_j\Big|^4\Big]
\leq C\Big(\sum_{j=1}^mu_j\Big)^2
=C\Big(\sum_{n\in\mathcal E}
3^{\omega(n)}\Big)^2\\\leq Cc_1^2 u^{2/3} (v-u)^{4/3}(\log u)^{52/3}.
\end{multline*}
Applying \eqref{eq:max-fourth-short} to $(f(n_j/P(n_j))e_{P(n_j)})_{1\leq j\leq m}$,
with the same weights $u_j$, and using \eqref{eq:vector-fourth-short}, gives the same bound for the second expectation in \eqref{eq:entry-short}.
Thus, \eqref{eq:entry-short} holds with $C_1:=2Cc_1^2$. If $m=0$, both $x\mapsto E(u,x)$ and $x\mapsto E_{\mathbf a}(u,x)$ vanish on $[u,v]$, and \eqref{eq:entry-short} holds trivially.

\noindent {\sc Step 4}. We next prove that
\begin{equation}\label{eq:exit-short}
\me\big[\sup_{u\leq x\leq v} |L(u,x)|^4\big]+\me \big[\sup_{u\leq x\leq v}\|L_{\mathbf a}(u,x)\|_2^4\big]\leq C_2 (v-u)^2 (\log u)^4
\end{equation}
for an absolute constant $C_2>0$.

Define $\mathcal L:=\{n\in\mathcal A:~u<(P(n))^2\leq v\}$. If $n\in\mathcal L$ and $p=P(n)$, then $u^{1/2}<p\leq v^{1/2}$. Conversely, once such a prime $p$ is fixed, the relevant integers have the form $n=pm$, with $m$ squarefree and $m<p$. Hence $$\sum_{n\in\mathcal L}3^{\omega(n)}=3\sum_{u^{1/2}<p\leq v^{1/2}}\sum_{\substack{m<p\\m\ {\rm squarefree}}}3^{\omega(m)}.$$ Trivially,
\begin{equation}\label{eq:exit-prime-count}
\#\{p:~u^{1/2}<p\leq v^{1/2}\}\leq v^{1/2}-u^{1/2}+1\leq u^{-1/2}(v-u)/2+1.
\end{equation}
For each such prime $p$, using \eqref{eq:divisor-short} with $k=3$ and $p\leq (3/2)^{1/2}u^{1/2}$, we infer $$\sum_{\substack{m<p\\m\ {\rm squarefree}}}3^{\omega(m)}\leq \sum_{m<p}d_3(m)\leq c_0(3)p(\log p)^2\leq c_3u^{1/2}(\log u)^2,$$ where $c_3=(3/2)^{{5/4}}c_0(3)$. The last inequality holds for $u\geq (3/2)^{1/2}$ and, in particular, under the standing assumptions which imply that $u\geq 4$. Combining this with \eqref{eq:exit-prime-count} and using $u^{1/2}\leq v-u$, we obtain
\begin{equation}\label{eq:WL-short}
\sum_{n\in\mathcal L}3^{\omega(n)}\leq 3c_3 (v-u+u^{1/2})(\log u)^2\leq 6c_3 (v-u)(\log u)^2.
\end{equation}

Write the primes in $(u^{1/2},v^{1/2}]$ as $p_1<\ldots<p_s$ and put $$Y_j:=\sum_{\substack{n\in\mathcal L\\P(n)=p_j}}f(n),\quad \boldsymbol Y_j:=
\sum_{\substack{n\in\mathcal L\\P(n)=p_j}}f(n/p_j)e_{p_j}~\text{and}~u_j:=\sum_{\substack{n\in\mathcal L\\P(n)=p_j}}3^{\omega(n)}.$$ Assume that $s\geq 1$. Then $L(u,x)$ and $L_{\mathbf a}(u,x)$ are the partial
sums of $(Y_j)$ and $(\boldsymbol Y_j)$, respectively, taken over those indices $j$ for which $p_j^2\leq x$. The bounds \eqref{eq:fourth-moment-short} and \eqref{eq:vector-fourth-short} apply to every consecutive
collection of these groups. Invoking \eqref{eq:max-fourth-short}, with weights $u_j$, for these two sequences and then using \eqref{eq:WL-short}, we obtain \eqref{eq:exit-short}, with $C_2:=72Cc_3^2$,
where $C$ is the constant in \eqref{eq:max-fourth-short}. If $s=0$, both sums vanish, and the conclusion holds trivially.

\noindent {\sc Step 5.} Recall that $d(u,x)=\|\mathbf a(x)-\mathbf a(u)\|_2$. It follows from \eqref{eq:S-entry-exit-short} and \eqref{eq:a-entry-exit-short} that
\begin{align*}
|M^{>}(x)-M^{>}(u)|+\|\mathbf a(x)-\mathbf a(u)\|_2 \leq |E(u,x)|+|L(u,x)|+\|E_{\mathbf a}(u,x)\|_2+\|L_{\mathbf a}(u,x)\|_2.
\end{align*}
Using the inequality $(a_1+\ldots+a_4)^4\le 4^3\sum_{j=1}^4 a_j^4$ and then \eqref{eq:entry-short} and the consequence of \eqref{eq:exit-short}
\begin{multline*}
\me\big[\sup_{u\leq x\leq v} |L(u,x)|^4\big]+\me \big[\sup_{u\leq x\leq v}\|L_{\mathbf a}(u,x)\|_2^4\big]\leq C_2 (v-u)^2 (\log u)^4\\\leq 2^{-2/3}C_2 u^{2/3}(v-u)^{4/3}(\log u)^{52/3},
\end{multline*}
we arrive at \eqref{eq:fourth}, with $C:=4^3(C_1+2^{-2/3}C_2)$. The proof of Lemma \ref{lem:fourth} is complete.
\end{proof}

\begin{proof}[Proof of Proposition~\ref{prop:interp}]
We introduce a very fine deterministic grid inside $[\eee^T, \eee^{2T}]$ given by
\begin{equation*}
D_T=\{\eee^{T+jT^{-120}}:~j=0,1,\ldots,T^{121}\}.
\end{equation*}
The grid contains both endpoints of $[\eee^T,\eee^{2T}]$ and splits this interval into $T^{121}$ closed intervals with disjoint interiors and equal logarithmic length $T^{-120}$, which we call cells.

We first compare arbitrary points with the endpoints of their cells. To this end, we show that each cell $[u,v]$ satisfies the standing assumptions of Lemma \ref{lem:fourth}. Indeed, $$u^{-1/2}(v-u)=u^{1/2}\big(\eee^{T^{-120}}-1\big)\geq \eee^{T/2}T^{-120}~\to~\infty,\quad T\to\infty.$$ Moreover, for all sufficiently large $T$, $v/u=\eee^{T^{-120}}\leq 3/2$ and thereupon $u+u^{1/2}\leq v\leq 3u/2$. Hence, Lemma~\ref{lem:fourth} applies and yields, in combination with Markov's inequality
\begin{multline*}
\mmp\big\{\sup_{u\leq x\leq v}(|M^{>}(x)-M^{>}(u)|+d(u,x))>u^{1/2}\,T^{-2}\big\}\\ \leq 2^{52/3}C\frac{u^{2/3}\big(u\big(\eee^{T^{-120}}-1\big)\big)^{4/3}T^{52/3}}{u^2T^{-8}}=O\big(T^{-134}\big).
\end{multline*}
A union bound then gives $$\mmp\big\{\max_{[u,v]~\text{a cell}}u^{-1/2}\sup_{u\leq x\leq v}(|M^{>}(x)-M^{>}(u)|+d(u,x))>T^{-2}\big\}=T^{121}O\big(T^{-134}\big)=O\big(T^{-13}\big).$$ Since the right-hand side is summable over dyadic $T$, we conclude, by the direct part of the Borel-Cantelli lemma, that, a.s., for all sufficiently large $T$ and simultaneously for all cells,
\begin{equation}\label{eq:finecells}
\sup_{u\le x\le v}(|M^{>}(x)-M^{>}(u)|+d(u,x))\le u^{1/2}\,T^{-2}.
\end{equation}

For $y\in D_T$, let $c=c_T(y)$ be the largest
point of $C_T$, see \eqref{eq:coarse} for the definition, with $c\leq y$. The deterministic pair $(c,y)$ satisfies $\log(y/c)\leq T^{-\alpha}$. Indeed, if $y=\eee^{2T}$, then $c=\eee^{2T}$. If $y=\exp(T+jT^{-120})$ for some $j\in\mn_0$, $j\leq T^{121}-1$, then $c=\exp(T+\lfloor jT^{-120+\alpha}\rfloor T^{-\alpha})$, so that $\log (y/c)=jT^{-120}-\lfloor jT^{-120+\alpha}\rfloor T^{-\alpha}\leq T^{-\alpha}$. Combining Lemmas~\ref{lem:distance}, \ref{lem:shifts} and Lemma~\ref{lem:R} (applied to $G_T=C_T\cup D_T$), we obtain, a.s., for all sufficiently large $T$, simultaneously for all $y\in D_T$,
\begin{equation}\label{eq:dxy}
\frac{(d(c,y))^2}{c}\leq 22T^{-\alpha/2}+C(\omega)T^{-1-\alpha}(\log T)^6 \leq T^{-\alpha/3}.
\end{equation}
For each fixed pair $(c,y)$, put $b_p:=a_p(y)-a_p(c)$. Whenever $a_p(y)$ is nonzero, we have $p>y^{1/2}$ and $y/p<p$, so $a_p(y)$ is $\mathcal F_{p^-}$-measurable.
The same argument applies to $a_p(c)$. Hence, $b_p$ is $\mathcal F_{p^-}$-measurable for each prime $p$. Moreover, $b_p=0$ outside the deterministic finite set $\mathcal P_1:=\mathcal P\cap(c^{1/2},y]$.  Applying \eqref{eq:predictable-tail}, with the deterministic choices $t=c^{1/2}\,T^{-\alpha/12}$ and $v=cT^{-\alpha/3}$, we obtain
$$\mmp\big\{|M^{>}(y)-M^{>}(c)|>c^{1/2}\,T^{-\alpha/12},~(d(c,y))^2\le cT^{-\alpha/3}\big\} \leq 2\exp\big(-T^{\alpha/6}/2\big).$$ From this and a union bound it follows that
\begin{multline*}
\mmp\Big\{\bigcup_{y\in D_T}\big\{|M^{>}(y)-M^{>}(c_T(y))|>c_T(y)^{1/2}T^{-\alpha/12},~ d(c_T(y),y)^2\leq c_T(y)T^{-\alpha/3}\big\}\Big\}\\\leq 2(T^{121}+1)\exp\big(-T^{\alpha/6}/2\big).
\end{multline*}
These bounds are summable over $T=2^r$. By the direct part of the Borel--Cantelli lemma, the events on the left only occur finitely often a.s. Recalling \eqref{eq:dxy}, we conclude that a.s., for all sufficiently large $T$ and all $y\in D_T$,
\begin{equation}\label{eq:gridincrements}
|M^{>}(y)-M^{>}(c_T(y))|\leq (c_T(y))^{1/2} T^{-\alpha/12}\quad\text{and}\quad d(c_T(y),y)\leq (c_T(y))^{1/2}T^{-\alpha/6}.
\end{equation}
Here, the second inequality follows from \eqref{eq:dxy}, whereas the first inequality follows from the second and the preceding Borel-Cantelli argument.

Finally, for $x\in[\eee^T, \eee^{2T})$, choose the largest $y\in D_T$ with $y\le x$ and put $z=c_T(y)$. Then $z\leq y\leq x$, $\log(x/z)\leq 2T^{-\alpha}$ and the points $x$ and $y$ lie in the same cell of $D_T$. By \eqref{eq:finecells}, \eqref{eq:gridincrements}, and the triangle inequality,
\begin{multline*}
|M^{>}(x)-M^{>}(z)|+d(x,z)\leq |M^{>}(x)-M^{>}(y)|+d(x,y)+|M^{>}(y)-M^{>}(c_T(y))|+d(y,c_T(y))\\\leq y^{1/2}T^{-2}+ (c_T(y))^{1/2}\big(T^{-\alpha/12}+ T^{-\alpha/6}\big)\leq 2x^{1/2}T^{-\alpha/12}.
\end{multline*}
If $x=\eee^{2T}$, then $z=y=x$, and the latter inequality holds trivially. The proof of Proposition~\ref{prop:interp} is complete.
\end{proof}

\section{Proof of Theorem \ref{thm:main} -- the lower bound}\label{sect:lower}

\subsection{Our strategy}\label{sect:strat}

The second part of the paper is concerned with the lower bound for $\limsup$ and the upper bound for $\liminf$.
\begin{proposition}\label{prop:lower bound}
We have
\begin{equation}\label{eq:lower sup}
\limsup_{x\to\infty}\frac{M^{>}(x)}{(2Q(x)\log\log Q(x))^{1/2}}\geq 1\quad\text{{\rm a.s.}}
\end{equation}
and $$\liminf_{x\to\infty}\frac{M^{>}(x)}{(2Q(x)\log\log Q(x))^{1/2}}\leq -1\quad\text{{\rm a.s.}}$$
\end{proposition}

Theorem~\ref{thm:main} is an immediate consequence of Propositions~\ref{prop:upper bound} and~\ref{prop:lower bound}.

Fix $\lambda\in(0,1)$ and, for $r\in\mn$, put
$$T_r:=4^r,\quad Y_r:=\exp(T_r+T_r^{1-\lambda}),\quad a_r:=Y_r^{1/2}\quad\text{and}\quad \mathcal H_r:=\mathcal F_{a_r}.$$
Define the deterministic set
$$\mathcal X_r:=\{\exp(T_r+k):~k\in\mathbb Z,~0\leq k\leq\lfloor T_r^{1-\lambda}\rfloor\}.$$
For $x\in\mathcal X_r$, put
\begin{equation}\label{eq:UV}
\begin{aligned}
U_r(x)&:=\sum_{x^{1/2}<p\leq a_r}f(p)M(x/p),&\qquad u_r(x)&:=\sum_{x^{1/2}<p\leq a_r}(M(x/p))^2,\\
V_r(x)&:=\sum_{a_r<p\leq x}f(p)M(x/p),&\qquad v_r(x)&:=\sum_{a_r<p\leq x}(M(x/p))^2.
\end{aligned}
\end{equation}
Then $M^{>}(x)=U_r(x)+V_r(x)$ and $Q(x)=u_r(x)+v_r(x)$.
Each $x\in\mathcal X_r$ satisfies $x\leq Y_r=a_r^2$. For $p>a_r$, this entails $x/p<a_r$. Hence $v_r(x)$ and all the coefficients $M(x/p)$ appearing in $V_r(x)$ are $\mathcal H_r$-measurable. Conditionally on $\mathcal H_r$, the variables $(f(p))_{a_r<p\leq Y_r}$ remain independent Rademacher variables. In particular,
$$\me_{a_r}[V_r(x)]=0\quad\text{and}\quad\me_{a_r}[(V_r(x))^2]=v_r(x).$$
For the primes appearing in $U_r(x)$ and $u_r(x)$, we have $p\leq a_r$ and $x/p<x^{1/2}\leq a_r$. Hence $U_r(x)$ and $u_r(x)$ are also $\mathcal H_r$-measurable.

The lower bound for the $\limsup$ is driven by the principal term $V_r$, comprising all primes exceeding $a_r$. Formula \eqref{eq:covinput} in  Proposition~\ref{prop:novel} is a sharpened form of Proposition~2 in \cite{Harper:2023}, and its proof is the most technically demanding part of the paper. With probability tending to one, it controls, for each $x\in\mathcal X_r$, the number of points $y\neq x$ for which $|\me_{a_r}[V_r(x)V_r(y)]|$ is large relative to $(xy)^{1/2}$. Lemma~\ref{lem:prune} provides a large $\mathcal H_r$-measurable subset $\mathcal G_r\subseteq\mathcal X_r$ on which $U_r(x)$ is negligible relative to the intended normalization and $Q(x)/v_r(x)$ is close to one, on suitable events of probability tending to one. A greedy selection then gives a sufficiently large $\mathcal H_r$-measurable subset $\mathcal S_r\subseteq\mathcal G_r$ whose normalized sums have small pairwise conditional covariances. Lemma~\ref{lem:rademacher-max} shows that the maximum of $n$ weighted Rademacher sums with unit variances, sufficiently small coefficients and vanishing covariance upper bounds exceeds $(1-d)(2\log n)^{1/2}$ with probability tending to one, for each fixed $d\in(0,1)$. Applying this lemma conditionally on $\mathcal H_r$ directly to the sums $V_r(x)$ indexed by $\mathcal S_r$, and using the bound for $U_r(x)$, gives the required conditional exceedance estimate for $M^{>}$. The conditional Borel-Cantelli lemma then yields infinitely many exceedances a.s. The assertion for the $\liminf$ follows by reversing the signs of both terms.

\subsection{Auxiliary results}\label{sect:aux}

We start with a number of preparatory results, needed for the proof of Proposition \ref{prop:lower bound}.

For $x,y\in\mathcal X_r$, put $$C_r(x,y):=\me_{a_r}[V_r(x)V_r(y)]=\sum_{a_r<p\leq \min(x,y)}M(x/p)M(y/p).$$
\begin{proposition}\label{prop:novel}
There is $c_0>0$ such that
\begin{equation}\label{eq:var v+w}
\lim_{r\to\infty}\mmp\Big\{\inf_{x\in\mathcal X_r}\frac{v_r(x)}{x}\geq \frac{c_0}{(\log\log a_r)^{11/20}}\Big\}=1.
\end{equation}
Also, for each fixed $\kappa\in (0,1)$,
\begin{equation}\label{eq:covinput}
\lim_{r\to\infty}\mmp\Big\{\max_{x\in\mathcal X_r}\#\Big\{y\in\mathcal X_r\setminus\{x\}:~|C_r(x,y)|>(\log\log a_r)^{-3/5}(xy)^{1/2}\Big\}\leq (\log a_r)^\kappa\Big\}=1.
\end{equation}
\end{proposition}

Here, we only give a proof of \eqref{eq:var v+w}, for it rests on a result already available in the literature. Since the proof of \eqref{eq:covinput} is long, it is postponed until Section \ref{sect:proofcov}. The auxiliary estimates are developed in Section~\ref{sect:proofnovel}.
\begin{proof}[Proof of \eqref{eq:var v+w} and Lemma \ref{lem:lowerclock}]
According to Proposition 1 in \cite{Harper:2023}, there exist absolute constants $c_H,C_H>0$ such that, for all sufficiently large $X$,
\begin{equation}\label{eq:HarperDet}
\frac{1}{z}\sum_{X<p\le z}(M(z/p))^2\geq \frac{c_H}{\log X}\int_1^{X^{1/7}}\frac{(M(u))^2}{u^2}\,{\rm d}u, \quad X^{8/7}\leq z\le X^{4/3}
\end{equation}
and, for $1\le W\le(\log\log X)^{1/100}$,
\begin{equation}\label{eq:HarperProb}
\mmp\left\{\frac1{\log X}\int_1^{X^{1/7}}\frac{(M(u))^2}{u^2}\,{\rm d}u
\geq \frac{\eee^{-2.2W}}{(\log\log X)^{1/2}}\right\}\geq 1-C_H\eee^{-0.1W}.
\end{equation}

We first treat \eqref{eq:var v+w}. For all sufficiently large $r$, each $x\in\mathcal X_r$ satisfies
$$x^{3/4}>a_r,\quad x^{3/28}\ge a_r^{1/7}\quad\text{and}\quad \log(x^{3/4})\le (3/2)\log a_r.$$ Indeed, the first two inequalities follow, respectively, from $(3/4)T_r>(1/2)(T_r+T_r^{1-\lambda})$ and $(3/28) T_r\ge (1/14)(T_r+T_r^{1-\lambda})$, both of which hold eventually. The third is a consequence of $x\leq a_r^2$, which is secured by $x\in\mathcal X_r$.

Applying \eqref{eq:HarperDet} with $X=x^{3/4}$ and $z=x$, which is admissible because $X^{8/7}=x^{6/7}\le x=X^{4/3}$, we obtain
\begin{multline}\label{eq:vr-lower}
\frac{v_r(x)}{x}=\frac{1}{x}\sum_{a_r<p\le x}(M(x/p))^2 \geq \frac{1}{x}\sum_{x^{3/4}<p\le x}(M(x/p))^2 \geq\frac{c_H}{\log(x^{3/4})} \int_1^{x^{3/28}}\frac{(M(u))^2}{u^2}\,{\rm d}u\\
\geq\frac{2c_H}{3\log a_r}\int_1^{a_r^{1/7}}\frac{(M(u))^2}{u^2}\,{\rm d}u.
\end{multline}
This inequality holds simultaneously for all $x\in\mathcal X_r$, and its right-hand side does not depend on $x$.

Next we apply \eqref{eq:HarperProb} with $X=a_r$ and $W=(\log\log\log a_r)/44$. For all sufficiently large $r$, we have $1\le W\le (\log\log a_r)^{1/100}$ and thereupon $$\mmp\Big\{\frac{1}{\log a_r}\int_1^{a_r^{1/7}}
\frac{(M(u))^2}{u^2}\,{\rm d}u\geq (\log\log a_r) ^{-11/20}\Big\}\geq 1-C_H(\log\log a_r)^{-1/440}.$$ Consequently, choosing $c_0:=2c_H/3$, we obtain $$\mmp\left\{\inf_{x\in\mathcal X_r}\frac{v_r(x)}{x}\geq \frac{c_0}{(\log\log a_r)^{11/20}}\right\}\geq 1-C_H(\log\log a_r)^{-1/440}~\to~1,\quad r\to\infty,$$ thereby proving \eqref{eq:var v+w}.

Now we pass to the proof of Lemma \ref{lem:lowerclock}. Put $q:=7/6$ and then  $J:=\left\lceil \log 2/\log q\right\rceil=5$. For $0\le j<J$, define $z_j:=\eee^{Tq^j}$ and $X_j:=z_j^{7/8}$. If $x\in[z_j,z_j^{7/6}]$, then $X_j^{8/7}=z_j \le x\le z_j^{7/6}=X_j^{4/3}$ and $x^{1/2} \le z_j^{7/12}<z_j^{7/8}=X_j$. Hence, each prime satisfying $X_j<p\le x$ also satisfies $x^{1/2}<p\le x$. Using \eqref{eq:HarperDet}, with $X=X_j$ and $z=x$, we obtain \begin{equation*}
\frac{Q(x)}{x}=\frac{1}{x}\sum_{x^{1/2}<p\leq x}(M(x/p))^2 \geq \frac{1}{x}\sum_{X_j<p\le x}(M(x/p))^2\geq\frac{c_H}{\log X_j}\int_1^{X_j^{1/7}}\frac{(M(u))^2}{u^2}\dd u.
\end{equation*}

Next, we apply \eqref{eq:HarperProb} to each of the $J=5$ integrals, with $W_j=(\log\log X_j)^{1/100}/2$. For all sufficiently large $T$, these choices satisfy $1\leq W_j\leq(\log\log X_j)^{1/100}$ for each $0\leq j<J$. Put
$$\mathcal E_T:=\bigcap_{j=0}^{J-1}\Big\{\frac{1}{\log X_j}\int_1^{X_j^{1/7}}\frac{(M(u))^2}{u^2} \dd u \geq \frac{\eee^{-1.1\,(\log\log X_j)^{1/100}}}{(\log\log X_j)^{1/2}}\Big\}.$$ Since $\log\log X_j=\log T+\log(7/8)+j\log(7/6)$, we have, for all sufficiently large $T$ and all $0\leq j<J$, $(\log T)/2\leq\log\log X_j\leq 2\log T$. Consequently, the union bound and \eqref{eq:HarperProb} ensure that
$$\mmp(\mathcal E_T^c)\leq C_H\sum_{j=0}^{J-1}\eee^{-(\log\log X_j)^{1/100}/20}\leq 5C_H\eee^{-c_1(\log T)^{1/100}},$$ where $c_1:=2^{-1/100}/20$.
For dyadic $T=2^r$, we have, for $r_0\in\mn$ large enough, $$\sum_{r\geq r_0}\mmp(\mathcal E_{2^r}^c)\leq 5C_H\sum_{r\geq r_0}\eee^{-c_1(r\log2)^{1/100}}<\infty.$$
By the direct part of the Borel-Cantelli lemma, $\mathcal E_{2^r}$ occurs eventually a.s.

Observe that $\bigcup_{j=0}^{J-1}[z_j,z_j^{7/6}]=[\eee^T,\eee^{Tq^J}]\supseteq [\eee^T,\eee^{2T}]$ because $q^J>2$. On $\mathcal E_T$, this fact, together with the preceding deterministic comparison, entails
\begin{equation*}
\inf_{\eee^T\leq x\leq\eee^{2T}}\frac{Q(x)}{x}\geq c_H\min_{0\leq j<J}\frac{\eee^{-1.1(\log\log X_j)^{1/100}}}{(\log\log X_j)^{1/2}}\geq \frac{c_2\eee^{-C_2(\log T)^{1/100}}}{(\log T)^{1/2}},
\end{equation*}
where $c_2:=2^{-1/2}c_H$ and $C_2:=1.1\times 2^{1/100}$. Finally, $$\lim_{T\to\infty} \frac{C_2(\log T)^{1/100}+(\log\log T)/2-\log c_2}{\log T}=0.$$ Thus, for each fixed $b>0$ and all sufficiently large $T$,
$C_2(\log T)^{1/100}+(\log\log T)/2-\log c_2 \leq b\log T$, or, equivalently, $$\frac{c_2\eee^{-C_2(\log T)^{1/100}}}{(\log T)^{1/2}}\geq \eee^{-b\log T}=T^{-b}.$$ Combining these observations completes the proof of Lemma \ref{lem:lowerclock}.
\end{proof}

For each realization $\omega$ and each $x\in\mathcal X_r$, define $$B_r(x;\omega):=\{y\in\mathcal X_r\setminus\{x\}:~|C_r(x,y;\omega)|>(\log\log a_r)^{-3/5}(xy)^{1/2}\}.$$ For distinct $x,y\in\mathcal X_r$, we call $(x,y)$ a bad pair if $y\in B_r(x;\omega)$, and a good pair otherwise. We also call $y$ a bad or good neighbor of $x$, respectively.

Formula \eqref{eq:covinput} gives, for each fixed $\kappa\in(0,1)$,
\begin{equation}\label{eq:badneighbors}
\lim_{r\to\infty}\mmp\{\max_{x\in\mathcal X_r}|B_r(x)|\leq(\log a_r)^\kappa\}=1.
\end{equation}
By definition, each good pair $(x,y)$ satisfies $|C_r(x,y)|\leq(\log\log a_r)^{-3/5}(xy)^{1/2}$.

On the events appearing in \eqref{eq:var v+w}, we have $(v_r(x)v_r(y))^{1/2}\geq c_0(\log\log a_r)^{-11/20}(xy)^{1/2}$. Consequently, for each good pair $(x,y)$,
\begin{multline}\label{eq:goodneighbors}
\frac{|C_r(x,y)|}{(v_r(x)v_r(y))^{1/2}}=\frac{|C_r(x,y)|}{(xy)^{1/2}}\frac{(xy)^{1/2}}{(v_r(x)v_r(y))^{1/2}}\leq \frac{1}{(\log\log a_r)^{3/5}}\frac{(\log\log a_r)^{11/20}}{c_0}\\=\frac{1}{c_0(\log\log a_r)^{1/20}}=o(1),\quad r\to\infty.
\end{multline}

For later use, we note that by Mertens' theorem (see, for instance, Theorem 2.7 (d) in \cite{MV}),
\begin{equation}\label{eq:mertens}
\sum_{p\leq z}\frac{1}{p}=\log\log z+b+O\Big(\frac{1}{\log z}\Big),\quad z\to\infty,
\end{equation}
where $b=\gamma-\sum_p\sum_{k\geq 2}(1/(kp^k))$ and $\gamma$ is the Euler-Mascheroni constant.

With $c_0$ as in Proposition \ref{prop:novel}, put $$\mathcal G_r:=\left\{x\in\mathcal X_r:~(U_r(x))^2+u_r(x) \leq \frac{c_0x}{2(\log\log a_r)^{31/20}}\right\}.$$ At this point the set $\mathcal G_r$ may be empty. Recall the notation $\mathcal H_r=\mathcal F_{a_r}$.
\begin{lemma}\label{lem:prune}
The random set $\mathcal G_r$ is $\mathcal H_r$-measurable
and $\lim_{r\to\infty}\mmp\{|\mathcal G_r|\geq |\mathcal X_r|/2\}=1$.

On the events appearing in \eqref{eq:var v+w}, for each $x\in\mathcal G_r$,
\begin{equation}\label{eq:compare Q}
\frac{Q(x)}{v_r(x)}=1+O\big((\log\log a_r)^{-1}\big)\quad\text{and}\quad
|U_r(x)|\leq \Big(\frac{Q(x)}{\log\log a_r}\Big)^{1/2}
\end{equation}
where the $O$-term is uniform in $x\in\mathcal G_r$.
\end{lemma}
\begin{proof}
Measurability follows because, for each $x\in\mathcal X_r$, $U_r(x)$ and $u_r(x)$ are $\mathcal H_r$-measurable.

Uniformly in $x\in\mathcal X_r$,
\begin{multline*}
\me \big[(U_r(x))^2\big]=\me [u_r(x)]\leq x\sum_{x^{1/2}<p\leq a_r}p^{-1}\leq x\sum_{\exp(T_r/2)<p\leq a_r}p^{-1}=O(xT_r^{-1})+O(xT_r^{-\lambda})\\=O(xT_r^{-\lambda}),\quad r\to\infty.
\end{multline*}
Here, the first equality is a consequence of $\me_{x^{1/2}}\big[(U_r(x))^2\big]=u_r(x)$, the first inequality is secured by \eqref{eq:orthogon}, the second equality follows from \eqref{eq:mertens} and $$\log\log a_r-\log\log \exp(T_r/2)=\log (T_r+T_r^{1-\lambda})-\log T_r=O(T_r^{-\lambda}),\quad r\to\infty,$$ and the last equality holds because $\lambda\in (0,1)$.

Hence, by Markov's inequality, uniformly in $x\in\mathcal X_r$, $$\mmp\{x\notin\mathcal G_r\}\leq\frac{2(\log\log a_r)^{31/20}}{c_0 x}\me\big[(U_r(x))^2+u_r(x)\big]=O((\log\log a_r)^{31/20}T_r^{-\lambda})=o(1)$$ as $r\to\infty$. The last equality is ensured by $\log\log a_r\sim \log T_r$ as $r\to\infty$. Summing over the deterministic set $\mathcal X_r$ we obtain $\me |\mathcal X_r\setminus\mathcal G_r|=o(|\mathcal X_r|)$ as $r\to\infty$. With this at hand, another application of Markov's inequality gives $$\mmp\{|\mathcal G_r|\leq |\mathcal X_r|/2\}=\mmp\{|\mathcal X_r\setminus \mathcal G_r|\geq |\mathcal X_r|/2\}\leq \frac{2}{|\mathcal X_r|}\me\big[|\mathcal X_r\setminus\mathcal G_r|\big]~\to~0,\quad r\to\infty.$$

On the events appearing in \eqref{eq:var v+w}, for each $x\in\mathcal G_r$, $$0\leq\frac{u_r(x)}{v_r(x)}\leq\frac{1}{2\log\log a_r},\quad \frac{Q(x)}{v_r(x)}=1+\frac{u_r(x)}{v_r(x)}.$$
This proves the first assertion of \eqref{eq:compare Q}, uniformly over these points and realizations. Furthermore, $$|U_r(x)|\leq\Big(\frac{c_0x}{2(\log\log a_r)^{31/20}}\Big)^{1/2}
\leq\Big(\frac{v_r(x)}{2\log\log a_r}\Big)^{1/2}\leq\Big(\frac{Q(x)}{\log\log a_r}\Big)^{1/2},$$ which proves the second assertion.
\end{proof}

Fix $\kappa\in(0,1-\lambda)$ and let $c_0$ be as in Proposition~\ref{prop:novel}. Denote by $\mathcal E_r$ the intersection of the events
\begin{equation}\label{eq:uvv}
\begin{gathered}
\max_{u\leq a_r}|M(u)|\leq a_r^{2/3},\qquad
\min_{x\in\mathcal X_r}\frac{v_r(x)}x\geq\frac{c_0}{(\log\log a_r)^{11/20}},\\
\max_{x\in\mathcal X_r}|B_r(x)|\leq(\log a_r)^\kappa,\qquad
|\mathcal G_r|\geq|\mathcal X_r|/2.
\end{gathered}
\end{equation}
Each event is $\mathcal H_r$-measurable, hence so is $\mathcal E_r$. By \eqref{eq:maxM}, with $b_r=a_r$, \eqref{eq:var v+w}, \eqref{eq:badneighbors} and Lemma~\ref{lem:prune}, we have $\lim_{r\to\infty}\mmp(\mathcal E_r)=1$.

Fix a realization and consider the finite graph whose vertex set is $\mathcal G_r$, with two distinct vertices joined by an edge if they form a bad pair. Since $C_r(x,y)=C_r(y,x)$, this graph is undirected. An independent set in this graph, being a set containing no two vertices joined by an edge, is a subset of $\mathcal G_r$ containing no bad pair.

We construct such a set recursively. Initially, all points of $\mathcal G_r$ are available. If there is an available point, select the smallest one, include it in $\mathcal S_r$, and remove both
that point and all its available bad neighbors from the available set. Repeat until no available points remain. If $\mathcal G_r$ is empty, put $\mathcal S_r=\varnothing$. Every two distinct selected points form a good pair. Indeed, when a point is selected, all its bad neighbors that could otherwise be selected later are removed. Thus $\mathcal S_r$ is an independent set. The construction is $\mathcal H_r$-measurable. Indeed, the set $\mathcal X_r$ is finite and deterministic, membership in $\mathcal G_r$ and each bad-neighbor relation are $\mathcal H_r$-measurable, and at each step the choice is determined by the usual order on the real line.

On $\mathcal E_r$, we have $|\mathcal G_r|\geq |\mathcal X_r|/2$ by the definition of $\mathcal E_r$, and each selection removes at most $1+(\log a_r)^\kappa$ points, namely, the selected point $x$ and at most $|B_r(x)|\leq(\log a_r)^\kappa$ of its bad neighbors. Since  the number of selections is equal to $|\mathcal S_r|$, each point of $\mathcal G_r$ is removed at most once, and all points of $\mathcal G_r$ have been removed when the procedure terminates, we infer $|\mathcal G_r|\leq |\mathcal S_r|\big(1+(\log a_r)^\kappa\big)$. Hence,
\begin{equation*}
|\mathcal S_r| \geq\frac{|\mathcal X_r|}{2(1+(\log a_r)^\kappa)}.
\end{equation*}
In view of $\log |\mathcal X_r|=\log\big( (\lfloor T_r^{1-\lambda}\rfloor)+1\big)=(1-\lambda)\log T_r+o(1)$ and $\log\log a_r=\log T_r+O(1)$ as $r\to\infty$, we obtain, on $\mathcal E_r$,
\begin{equation}\label{eq:sizeS}
\log |\mathcal S_r|\geq (1-\lambda-\kappa)\log T_r-O(1),\quad r\to\infty.
\end{equation}
Here, the constant implicit in $O(1)$ is deterministic and independent of the realization in $\mathcal E_r$.

We proceed with a lemma regarding maxima of weighted Rademacher sums. This is the last ingredient needed for the proof of Proposition \ref{prop:lower bound}.
\begin{lemma}\label{lem:rademacher-max}
For each $n\in\mn$, let $I_n$ be a finite set and let $(\xi_{n,j})_{j\in I_n}$ be independent copies of a random variable $\xi$ with $\mmp\{\xi=\pm 1\}=1/2$. Put
$$Z_{n,i}:=\sum_{j\in I_n}c_{n,j,i}\xi_{n,j},\quad 1\leq i\leq n,$$ where the coefficients are real and deterministic, and
$\sum_{j\in I_n}c_{n,j,i}^2=1$ for $1\leq i\leq n$. Assume that $\beta_n:=\max_{j\in I_n,~1\leq i\leq n}|c_{n,j,i}|$ satisfies $\lim_{n\to\infty}\beta_n\log n=0$ and  that, for some deterministic sequence $(\rho_n)_{n\geq 1}$
satisfying $\lim_{n\to\infty}\rho_n=0$, $$\max_{1\leq i<k\leq n}\sum_{j\in I_n}c_{n,j,i}c_{n,j,k}\leq \rho_n$$ for all sufficiently large $n$. Then, for each fixed $\delta\in(0,1)$,
\begin{equation}\label{eq:rademacher max}
\lim_{n\to\infty} \mmp\big\{\max_{1\leq i\leq n} Z_{n,i}>(1-\delta)(2\log n)^{1/2}\big\}=1.
\end{equation}
With $n_0\geq 2$ fixed independently of the coefficient array, the convergence is uniform over coefficient arrays satisfying, for all $n\geq n_0$, $\beta_n\leq b_n$ and
$$\max_{1\leq i<k\leq n}\sum_{j\in I_n}c_{n,j,i}c_{n,j,k}\leq r_n,$$ where $(b_n)_{n\geq 1}$ and $(r_n)_{n\geq 1}$ are given deterministic nonnegative sequences satisfying $\lim_{n\to\infty}b_n\log n=0$ and $\lim_{n\to\infty}r_n=0$.
\end{lemma}
\begin{proof}
Replacing $\rho_n$ by $\max(\rho_n,0)$ if needed, we can and do assume that $\rho_n\in [0,1)$ for all sufficiently large $n$.

Fix $\delta\in(0,1)$ and choose $\alpha\in(1-\delta,1)$. Put $t_n:=\alpha(2\log n)^{1/2}$ and $h_n:=t_n+4$ for $n\in\mn$. We first obtain a uniform lower bound for the individual exceedance probabilities.
For $1\leq i\leq n$ and $h\in\mr$, put $$K_{n,i}(h):=\log \me\big[\eee^{hZ_{n,i}}\big]=\sum_{j\in I_n} \log\cosh(hc_{n,j,i}).$$ In view of $$\sum_{j\in I_n} c_{n,j,i}^4 \leq\beta_n^2\sum_{j\in I_n}c_{n,j,i}^2=\beta_n^2,$$
the expansions $\log\cosh z=z^2/2+O(z^4)$ and $\tanh z=z+O(z^3)$ as $z\to 0$ show that, uniformly in $1\leq i\leq n$,
\begin{multline}\label{eq:mgf}
K_{n,i}(h_n)=h_n^2/2+O(\beta_n^2h_n^4)=h_n^2/2+o(1)\quad\text{and}\quad K_{n,i}^\prime(h_n)=\sum_{j\in I_n}c_{n,j,i}\tanh (h_n c_{n,j,i})\\=h_n+O(\beta_n^2h_n^3)=h_n+o(1),\quad n\to\infty.
\end{multline}
Here, we have used $\lim_{n\to\infty} \beta_n\log n=0$. Moreover, in view of $\cosh z\geq 1$ for $z\in\mr$, we conclude that, for each $h\in\mr$,
\begin{equation}\label{eq:rad variance}
K_{n,i}^{\prime\prime}(h)=\sum_{j\in I_n}\frac{c_{n,j,i}^2}{\cosh^2(hc_{n,j,i})}\leq 1.
\end{equation}

Introduce the exponentially tilted probability measure $\mmp_{n,i}^{(h_n)}$ by $$\frac{{\rm d}\mmp_{n,i}^{(h_n)}}{{\rm d}\mmp}=\exp(h_nZ_{n,i}-K_{n,i}(h_n)),\quad 1\leq i\leq n,~n\geq 2.$$
Under $\mmp_{n,i}^{(h_n)}$, the variable $Z_{n,i}$ has mean $K_{n,i}'(h_n)$ and variance $K_{n,i}''(h_n)\leq1$. Hence, by Chebyshev's inequality, $\mmp_{n,i}^{(h_n)}\big\{|Z_{n,i}-K_{n,i}'(h_n)|\leq 2\big\}\geq 3/4$. By the second part of \eqref{eq:mgf}, for all sufficiently large $n$, uniformly in $1\leq i\leq n$, $$\{|Z_{n,i}-K_{n,i}'(h_n)|\leq 2\}\subseteq \{t_n<Z_{n,i}\leq t_n+7\}.$$
Consequently,
\begin{multline*}
\mmp\{Z_{n,i}>t_n\}=\me_{n,i}^{(h_n)}\left[\eee^{-h_nZ_{n,i}+K_{n,i}(h_n)}\1_{\{Z_{n,i}>t_n\}}\right]\geq \me_{n,i}^{(h_n)}\left[\eee^{-h_nZ_{n,i}+K_{n,i}(h_n)}\1_{\{t_n<Z_{n,i}\leq t_n+7\}}\right]\\
\geq (3/4)\exp(K_{n,i}(h_n)-h_n(t_n+7)).
\end{multline*}
Recalling the first part of \eqref{eq:mgf} and noting that $h_n^2/2-h_n(t_n+7)=-t_n^2/2-7t_n-20$, we conclude that there exists an absolute constant $C>0$ such that, for all
sufficiently large $n$,
\begin{equation}\label{eq:onedim}
\min_{1\leq i\leq n}\mmp\{Z_{n,i}>t_n\}\geq \exp(-t_n^2/2-7t_n-C)=:q_n.
\end{equation}

Now we estimate the joint exceedance probabilities. Put
$$r_{n,i,k}:=\sum_{j\in I_n}c_{n,j,i}c_{n,j,k},\quad 1\leq i,k\leq n,~n\geq 2.$$ For $1\leq i,k\leq n$, $i\neq k$ and $h>0$, the inequality $\log\cosh z\leq z^2/2$ for $z\in\mr$ entails
\begin{equation*}
\log\me\big[\eee^{h(Z_{n,i}+Z_{n,k})}\big]\leq (h^2/2)\sum_{j\in I_n}(c_{n,j,i}+c_{n,j,k})^2=h^2(1+r_{n,i,k})\leq h^2(1+\rho_n).
\end{equation*}
An application of Markov's inequality, with
$h=t_n/(1+\rho_n)$, then yields
\begin{equation}\label{eq:twodim}
\mmp\{Z_{n,i}>t_n,~ Z_{n,k}>t_n\}\leq \mmp\{Z_{n,i}+Z_{n,k}>2t_n\}\leq \exp(-t_n^2/(1+\rho_n)).
\end{equation}
For $n\in\mn$, put $N_n:=\sum_{i=1}^n\1_{\{Z_{n,i}>t_n\}}$ and $H_n:=\max_{1\leq i\leq n}Z_{n,i}$. By \eqref{eq:onedim},
$$\me[N_n]\geq nq_n=\exp\{(1-\alpha^2)\log n-O\big((\log n)^{1/2}\big)\}~\to~\infty,\quad n\to\infty.$$ Combining \eqref{eq:onedim} and \eqref{eq:twodim}, we obtain
\begin{multline*}
\frac{\me\big[N_n^2\big]}{(\me [N_n])^2}=\frac{1}{\me [N_n]}+\frac{\sum_{1\leq i\neq k\leq n}\mmp\{Z_{n,i}>t_n,~Z_{n,k}>t_n\}}{(\me [N_n])^2}\\\leq \frac{1}{nq_n}+\frac{\exp(-t_n^2/(1+\rho_n))}{q_n^2}=\frac{1}{nq_n}+\exp\left(\frac{\rho_nt_n^2}{1+\rho_n}+14t_n+2C\right)=:R_n.
\end{multline*}
Hence, by the second moment method,
\begin{equation}\label{eq:rad-weak-exceedance}
\mmp\{H_n>t_n\}=\mmp\{N_n>0\}\geq \frac{(\me [N_n])^2}{\me\big[N_n^2\big]}\geq R_n^{-1}.
\end{equation}
For later use, we note that
$R_n\geq\eee^{2C}\geq 1$. Since $t_n^2=2\alpha^2\log n$, $\lim_{n\to\infty}nq_n=\infty$ and $\lim_{n\to\infty}\rho_n=0$, we further have
$$\log R_n=\frac{2\alpha^2\rho_n}{1+\rho_n}\log n+O\big((\log n)^{1/2}\big)=o(\log n),\quad n\to\infty.$$

Now we work towards strengthening the lower bound in \eqref{eq:rad-weak-exceedance}.
The function $F_n$ defined by $F_n(z):=\max_{1\leq i\leq n}\sum_{j\in I_n}c_{n,j,i}(2z_j-1)$ for $z\in\mr^{|I_n|}$ is convex. It is also $2$-Lipschitz with respect to the
Euclidean norm, because, by the
Cauchy-Schwarz inequality,
\begin{multline*}
|F_n(z)-F_n(z')|\leq 2\max_{1\leq i\leq n}\left|\sum_{j\in I_n}c_{n,j,i}(z_j-z_j')\right|\leq 2\max_{1\leq i\leq n} \Big(\sum_{j\in I_n}c_{n,j,i}^2\Big)^{1/2}\Big(\sum_{j\in I_n}(z_j-z_j')^2\Big)^{1/2}\\=2\Big(\sum_{j\in I_n}(z_j-z_j')^2\Big)^{1/2}.
\end{multline*}
Let $\mu_n$ be a median of the distribution of $H_n$, that is, a not necessarily unique number satisfying $\mmp\{H_n\leq \mu_n\}\geq 1/2$ and $\mmp\{H_n\geq \mu_n\}\geq 1/2$. Since $H_n=F_n(\eta^{(n)})$, where $\eta^{(n)}$ is the vector with coordinates $(\xi_{n,j}+1)/2$, $j\in I_n$, we obtain with the help of Talagrand's convex concentration inequality as given in Corollary~4 of \cite{Johnson+Schechtman:1991} with $p=2$,
\begin{equation}\label{eq:concentr}
\mmp\{|H_n-\mu_n|> s\}\leq 4\exp(-s^2/16), \quad s>0.
\end{equation}

If $\mu_n<t_n$, then \eqref{eq:rad-weak-exceedance} and \eqref{eq:concentr} imply $R_n^{-1}\leq 4\exp(-(t_n-\mu_n)^2/16)$. As a consequence, $$\mu_n\geq t_n-4(\log(4R_n))^{1/2}=\alpha(2\log n)^{1/2}-o\big((\log n)^{1/2}{\big)},\quad n\to\infty$$ because $\log R_n=o(\log n)$ as $n\to\infty$. If $\mu_n\geq t_n$, then the latter estimate holds trivially. Since $\alpha>1-\delta$, we conclude that $s_n:=\mu_n-(1-\delta)(2\log n)^{1/2}\to \infty$ as $n\to\infty$. In particular, $s_n>0$ for all sufficiently large $n$. For such $n$, an application of \eqref{eq:concentr} with $s=s_n/2$ yields
\begin{multline*}
\mmp\{H_n\leq (1-\delta)(2\log n)^{1/2}\}=\mmp\{H_n-\mu_n \leq -s_n\}\\ \leq\mmp\{|H_n-\mu_n|>s_n/2\}\leq 4\eee^{-s_n^2/64}~\to 0
\end{multline*}
as $n\to\infty$. This proves \eqref{eq:rademacher max}.

For the claimed uniformity, repeat the preceding argument with $\rho_n$ replaced by $r_n$ and $\beta_n$ bounded by $b_n$. All remainder estimates and all thresholds for sufficiently large $n$ can then be chosen independently of the coefficient array.
\end{proof}

\subsection{Proof of Proposition \ref{prop:lower bound}}\label{sect:proofLB}

Fix $q\in (0,1)$. Choose $\lambda\in (0,1)$, $\kappa\in (0, 1-\lambda)$ and $d\in (0,1)$ satisfying
\begin{equation}\label{eq:ineq2}
(1-d)(1-\lambda-\kappa)^{1/2}>q.
\end{equation}
The constructions of Sections~\ref{sect:strat} and~\ref{sect:aux} are carried out with these $\lambda$ and $\kappa$. Recall that $\mathcal S_r\subseteq\mathcal G_r$ is the $\mathcal H_r$-measurable set obtained by the greedy selection in Subsection~\ref{sect:aux}, so that any two distinct points of $\mathcal S_r$ form a good pair. Put $t_r(x):=(1-d)(2v_r(x)\log(\max(|\mathcal S_r|,2)))^{1/2}$ for $x\in\mathcal X_r$, and define
$$A_r:=\{\exists x\in\mathcal S_r:~V_r(x)>t_r(x)\}.$$ By \eqref{eq:sizeS}, there is a deterministic sequence $(n_r)_{r\geq 1}$ with $\lim_{r\to\infty}n_r=\infty$ such that, on $\mathcal E_r$, $|\mathcal S_r|\geq n_r$ for all sufficiently large $r$. In particular, the maximum in the definition of $t_r(x)$ only matters outside $\mathcal E_r$. We first estimate $\mmp_{a_r}(A_r)$ from below. To this end, we apply Lemma~\ref{lem:rademacher-max} conditionally on $\mathcal H_r$. On $\mathcal E_r$, for $x\in\mathcal S_r$, we use a representation $$\frac{V_r(x)}{(v_r(x))^{1/2}}=\sum_{a_r<p\leq Y_r}c_{r,p,x}f(p),$$ where $c_{r,p,x}:=\1_{\{p\leq x\}}M(x/p)(v_r(x))^{-1/2}$. As explained in Subsection \ref{sect:strat}, all the coefficients $c_{r,p,x}$ are $\mathcal H_r$-measurable, and conditionally on $\mathcal H_r$, the variables $(f(p))_{a_r<p\leq Y_r}$ remain independent Rademacher variables. Also, $\sum_{a_r<p\leq Y_r}c_{r,p,x}^2=1$. On $\mathcal E_r$, the bounds $\max_{u\leq a_r}|M(u)|\leq a_r^{2/3}$ and $v_r(x)\geq c_0x/(\log\log a_r)^{11/20}$ for $x\in\mathcal X_r$ give
\begin{multline}\label{eq:rad-application-beta}
\beta_r:=\max_{x\in\mathcal S_r, a_r<p\leq Y_r}|c_{r,p,x}|\leq c_0^{-1/2}(\log\log a_r)^{11/40}a_r^{2/3}\eee^{-T_r/2}\\=
c_0^{-1/2}(\log\log a_r)^{11/40}\exp\left(-T_r/6+T_r^{1-\lambda}/3\right)\leq \eee^{-T_r/8}
\end{multline}
for all sufficiently large $r$, uniformly over the realizations in $\mathcal E_r$, having utilized the bound $x\geq \eee^{T_r}$ secured by $x\in\mathcal X_r$. Since $\log |\mathcal S_r|\leq \log |\mathcal X_r|\sim (1-\lambda)\log T_r$ as $r\to\infty$, inequality \eqref{eq:rad-application-beta} entails $\lim_{r\to\infty}\beta_r\log |\mathcal S_r|=0$ uniformly over the realizations in $\mathcal E_r$.

Every two distinct points of $\mathcal S_r$ form a good pair. Consequently, uniformly over the realizations in $\mathcal E_r$, for $x,y\in\mathcal S_r$, $x\neq y$,
\begin{equation*}
\Big|\sum_{a_r<p\leq Y_r}c_{r,p,x}c_{r,p,y}\Big|=\frac{|C_r(x,y)|}{(v_r(x)v_r(y))^{1/2}}\leq \frac{1}{c_0(\log\log a_r)^{1/20}}=:\rho_r~\to~0,\quad r\to\infty.
\end{equation*}

The estimates in the proof of Lemma~\ref{lem:rademacher-max} depend only on the number of sums, the maximal coefficient, and the upper bound on the pairwise covariances.
By \eqref{eq:sizeS}, \eqref{eq:rad-application-beta} and the bound on $\rho_r$, we have $|\mathcal S_r|\geq n_r$, $\lim_{r\to\infty}\beta_r\log |\mathcal S_r|=0$ and $\lim_{r\to\infty}\rho_r=0$, with deterministic bounds uniform over the realizations in $\mathcal E_r$. More precisely, on $\mathcal E_r$, $n_r\leq|\mathcal S_r|\leq|\mathcal X_r|\leq T_r+1$ for all sufficiently large $r$. Hence, on $\mathcal E_r$, $\beta_r\leq\eee^{-T_r/8}\leq b_{|\mathcal S_r|}$ and $\rho_r\leq r_{|\mathcal S_r|}$, where $b_n:=\eee^{-(n-1)/8}$ and $r_n:=c_0^{-1}(\log((n-1)/2))^{-1/20}$ for $n\geq 4$. These deterministic sequences satisfy $\lim_{n\to\infty}b_n\log n=0$ and $\lim_{n\to\infty}r_n=0$, so that the uniform version of Lemma~\ref{lem:rademacher-max} applies, with $n=|\mathcal S_r|$ and $\delta=d$. If $\eta_n(d)$ denotes the resulting deterministic upper bound for the probability of the complement of the event in \eqref{eq:rademacher max}, then one can take $\theta_r(d):=\sup_{n\geq n_r}\eta_n(d)$ below. Applying the lemma conditionally on $\mathcal H_r$, we conclude that, for the chosen $d\in(0,1)$, there exists a {\it deterministic} sequence $(\theta_r(d))_{r\geq 1}$ with $\lim_{r\to\infty}\theta_r(d)=0$ such that, for all sufficiently large $r$, a.s.\ on $\mathcal E_r$,
\begin{equation}\label{eq:Vmax-rademacher}
\mmp_{a_r}(A_r)\geq1-\theta_r(d).
\end{equation}

Lemma \ref{lem:deterministic-Q-bounds} implies that, uniformly in $x\in\mathcal X_r$,
\begin{equation}\label{eq:loglog trivial}
\log\log Q(x)=\log T_r+O(1)
\end{equation}
with a deterministic $O(1)$-constant.

We claim that, on $\mathcal E_r$, a point $x\in\mathcal S_r$ satisfying $V_r(x)>(1-d)(2v_r(x)\log |\mathcal S_r|)^{1/2}$ also satisfies $M^{>}(x)=U_r(x)+V_r(x)>q(2Q(x)\log\log Q(x))^{1/2}$ for all sufficiently large $r$. Indeed, since $\mathcal S_r\subseteq\mathcal G_r$, such a point satisfies
\begin{multline*}
\frac{U_r(x)+V_r(x)}{(2Q(x)\log\log Q(x))^{1/2}}>(1-d)\left(\frac{v_r(x)}{Q(x)}\frac{\log|\mathcal S_r|}{\log\log Q(x)}\right)^{1/2}-\frac{1}{(2\log\log a_r\log\log Q(x))^{1/2}}\\\geq(1-d)(1-\lambda-\kappa)^{1/2}-o(1)>q
\end{multline*}
for all sufficiently large $r$. Here we have used the second part of \eqref{eq:compare Q} for the first inequality, the first part of \eqref{eq:compare Q}, \eqref{eq:sizeS} and \eqref{eq:loglog trivial} for the second inequality, with all error estimates being uniform over $x\in\mathcal S_r$ and the realizations in $\mathcal E_r$. The last inequality follows from the fact that \eqref{eq:ineq2} is a strict inequality.

Define the events $$B_r^+:=\big\{\exists x\in\mathcal X_r:~M^{>}(x)>q(2Q(x)\log\log Q(x))^{1/2}\big\},\quad r\in\mn.$$ By the preceding implication, for all sufficiently large $r$, a.s.\ on $\mathcal E_r$,
\begin{equation*}
\mmp_{a_r}(B_r^+)\geq\mmp_{a_r}(A_r)\geq1-\theta_r(d).
\end{equation*}
Consequently, for all sufficiently large $r$, a.s., $\mmp_{a_r}\big(B_r^+\big)\geq \1_{\mathcal E_r}/2$. Since $\lim_{r\to\infty}\mmp(\mathcal E_r)=1$ entails $\mmp(\mathcal E_r\ \text{i.o.})=1$, where 'i.o.' stands for \emph{infinitely often}, we conclude that
\begin{equation}\label{eq:diver}
 \sum_{r\geq 1}\mmp_{a_r}(B_r^+)=\infty \quad\text{a.s.}
\end{equation}

Now we verify the filtration requirements for the conditional Borel-Cantelli lemma, see, for instance, Theorem 5.3.2 on p.~240 in \cite{Durrett:2010}. The sigma-fields $(\mathcal H_r)_{{r\geq 0}}$ form a filtration.
Moreover, $B_r^+$ is $\mathcal F_{Y_r}$-measurable. Since $\log Y_r=T_r+T_r^{1-\lambda}$ and $\log a_{r+1}=2T_r+(4T_r)^{1-\lambda}/2$, we have $Y_r<a_{r+1}$ for all {sufficiently large $r$}. Hence
$B_r^+\in\mathcal H_{r+1}$. With this at hand, the conditional Borel-Cantelli lemma applied to
\eqref{eq:diver} yields $\mmp(B_r^+~~\text{i.o.})=\mmp\{\sum_r\mmp_{a_r}(B_r^+)=\infty\}=1$. Taking a countable sequence $q\to 1-$, with the corresponding choices $\lambda=\lambda_q$, $\kappa=\kappa_q$ and $d=d_q$
satisfying \eqref{eq:ineq2}, and intersecting the resulting events of probability one, we arrive at \eqref{eq:lower sup}.

The second assertion of Proposition~\ref{prop:lower bound} follows by symmetry. Indeed, conditionally on $\mathcal H_r$, the vector $(f(p))_{a_r<p\leq Y_r}$ has the same distribution as $(-f(p))_{a_r<p\leq Y_r}$, and this sign reversal changes $V_r(x)$ into $-V_r(x)$, whereas the $\mathcal H_r$-measurable quantities $U_r(x)$, $v_r(x)$, $Q(x)$, $\mathcal S_r$ and $\mathcal E_r$ remain unchanged. Hence, the preceding argument applies to $-M^{>}(x)=-U_r(x)-V_r(x)$.

\subsection{Proof of Corollary~\ref{cor:scale}}\label{sect:proofscale}

\begin{proof}[Proof of Corollary~\ref{cor:scale}]
We shall show that there exist absolute constants $c,C>0$ such that, for each $W\geq 1$,
\begin{equation}\label{eq:scale-prob}
\mmp\Big\{\limsup_{x\to\infty}\frac{M^{>}(x)}{x^{1/2}(\log\log x)^{1/4}}\geq c\,\eee^{-1.1W}\Big\}\geq 1-C\eee^{-0.1W}.
\end{equation}
Letting $W\to\infty$ then proves the corollary.

We use the notation of Sections~\ref{sect:strat}--\ref{sect:proofLB}. Recall that, for $r\in\mn$, $\mathcal X_r\subset[\eee^{T_r},Y_r]$ is the set of points of the $r$th block, $\mathcal H_r=\F_{a_r}$, $\mathcal E_r$ is the $\mathcal H_r$-measurable event defined by intersection of events \eqref{eq:uvv}, which satisfies $\lim_{r\to\infty}\mmp(\mathcal E_r)=1$ and $A_r=\{\exists x\in\mathcal S_r:~V_r(x)>t_r(x)\}$.
Fix $q=1/2$, $W\geq 1$ and choose $\lambda$, $\kappa$ and $d$ satisfying \eqref{eq:ineq2}.

Put
$$J_r:=\frac{1}{\log a_r}\int_1^{a_r^{1/7}}\frac{(M(u))^2}{u^2}\,\dd u.$$
This variable is $\mathcal H_r$-measurable. Inequality \eqref{eq:vr-lower}, established in the proof of \eqref{eq:var v+w}, states that $v_r(x)\geq(2c_H/3)\,x J_r$ for all $x\in\mathcal X_r$ and all sufficiently large $r$. Hence, on the event
$$\mathcal E_r^{(W)}:=\mathcal E_r\cap\Big\{J_r\geq\frac{\eee^{-2.2W}}{(\log\log a_r)^{1/2}}\Big\},$$
we have, for all $x\in\mathcal X_r$,
\begin{equation}\label{eq:Q-lower-W}
Q(x)\geq v_r(x)\geq\frac{2c_H}{3}\,\eee^{-2.2W}\,\frac{x}{(\log\log a_r)^{1/2}}
\end{equation}
because $Q(x)=u_r(x)+v_r(x)$ (recall \eqref{eq:UV}) and $u_r(x)\geq 0$. The event $\mathcal E_r^{(W)}$ is $\mathcal H_r$-measurable. Since $W$ is fixed, for all sufficiently large $r$, we have $W\leq(\log\log a_r)^{1/100}$, so that \eqref{eq:HarperProb} with $X=a_r$ applies.
Since $\mmp(\mathcal E_r^{(W)})\geq\mmp\{J_r\geq\eee^{-2.2W}(\log\log a_r)^{-1/2}\}-\mmp(\mathcal E_r^c)$ and $\lim_{r\to\infty}\mmp(\mathcal E_r)=1$, this gives
\begin{equation}\label{eq:EW-prob}
\liminf_{r\to\infty}\mmp(\mathcal E_r^{(W)})\geq 1-C_H\eee^{-0.1W}.
\end{equation}

In the proof of Proposition~\ref{prop:lower bound}, see the paragraph following \eqref{eq:loglog trivial}, we have shown that, for all sufficiently large $r$, on $\mathcal E_r$, every $x\in\mathcal S_r$ with $V_r(x)>t_r(x)$ satisfies $M^{>}(x)>q(2Q(x)\log\log Q(x))^{1/2}$. Let $r$ be large, assume that $\mathcal E_r^{(W)}\cap A_r$ occurs and let $x\in\mathcal S_r$ be such a point. Recalling that $q=1/2$ and using \eqref{eq:Q-lower-W}, we obtain
$$M^{>}(x)>\frac12\big(2Q(x)\log\log Q(x)\big)^{1/2}\geq\Big(\frac{c_H}{3}\Big)^{1/2}\eee^{-1.1W}\,x^{1/2}\,\frac{(\log\log Q(x))^{1/2}}{(\log\log a_r)^{1/4}}.$$
By \eqref{eq:loglog trivial}, $\log\log Q(x)=\log T_r+O(1)$, and also $\log\log a_r=\log T_r+O(1)$ and $\log\log x=\log T_r+o(1)$, uniformly in $x\in\mathcal X_r$. Hence, the last fraction is at least $(\log\log x)^{1/4}/2$ for all sufficiently large $r$, and
$$M^{>}(x)>c\,\eee^{-1.1W}x^{1/2}(\log\log x)^{1/4},$$ where $c:=(c_H/12)^{1/2}$.

Denote by $A_r^{(W)}$ the event that the last inequality holds for some $x\in\mathcal X_r$. We already know that
$\mathcal E_r^{(W)}\cap A_r\subseteq A_r^{(W)}$ for all sufficiently large $r$. Since $\mathcal E_r^{(W)}$ is $\mathcal H_r$-measurable, \eqref{eq:Vmax-rademacher} gives
$$\mmp(A_r^{(W)})\geq\me\big[\1_{\mathcal E_r^{(W)}}\,\mmp_{a_r}(A_r)\big]\geq(1-\theta_r(d))\,\mmp(\mathcal E_r^{(W)})$$
and thereupon,
by \eqref{eq:EW-prob}, $\liminf_{r\to\infty}\mmp(A_r^{(W)})\geq 1-C_H\eee^{-0.1W}$. Since $\min\mathcal X_r=\eee^{T_r}\to\infty$, the event $\{A_r^{(W)}~\text{i.o.}\}$ is contained in the event on the left-hand side of \eqref{eq:scale-prob}. Finally,
$$\mmp\{A_r^{(W)}~\text{i.o.}\}=\lim_{n\to\infty}\mmp\Big(\bigcup_{r\geq n}A_r^{(W)}\Big)\geq\limsup_{r\to\infty}\mmp(A_r^{(W)})\geq 1-C_H\eee^{-0.1W}.$$
This proves \eqref{eq:scale-prob} with $C=C_H$. The assertion for $-M^{>}$ follows in the same way from the symmetry argument.
\end{proof}

\section{Proof of Corollary \ref{cor:cluster}}\label{sect:cor}

Since $M^{>}(x)=M^{>}(\lfloor x\rfloor)$ and $Q(x)=Q(\lfloor x\rfloor)$ for $x\geq 1$ (all changes occur at integers $pm$ or prime squares $p^2$), the function whose cluster set appears in \eqref{eq:cluster} is constant on each interval $[n,n+1)$, $n\in\mn$. Hence, it suffices to prove \eqref{eq:cluster} with $x$ replaced by $n\in\mn$.

For each prime $p$ and $n\in\mn_0$, put $v_p(n):=M(n/p)\1_{\{p>n^{1/2}\}}$, so that $Q(n)=\sum_{p\leq n}(v_p(n))^2$. Also, put $$h_n:=\Big(\sum_{p\leq n}(v_p(n)-v_p(n-1))^2\Big)^{1/2},\quad n\in\mn.$$
In the notation of Section~\ref{sect:upper}, $v_p(n)=a_p(n)$ and $h_n=d(n-1,n)$ for $n\geq 3$. If $n$ is not the square of a prime, then, for $p\leq n^{1/2}$, $v_p(n)=v_p(n-1)=0$, whereas, for $p>n^{1/2}$, $$v_p(n)-v_p(n-1)=M(n/p)-M((n-1)/p)=
\begin{cases}
f(n/p),& \text{if}~p|n,\\
0,& \text{if}~p\nmid n.
\end{cases}
$$
There is at most one prime divisor $p>n^{1/2}$, since two such divisors would have product greater than $n$. Consequently,
$$M^{>}(n)-M^{>}(n-1)=\begin{cases}
f(p)f(n/p),& \text{if there exists a prime }~p>n^{1/2}~\text{dividing}~n,\\
0,& \text{otherwise}.
\end{cases}
$$
If $n=p^2$ for a prime $p$, then $v_p(n)-v_p(n-1)=-M(p-1)$, whereas $v_q(n)-v_q(n-1)=0$ for each prime $q\neq p$. Hence, $$M^{>}(n)-M^{>}(n-1)=-f(p)M(p-1).$$ Summarizing, $$|M^{>}(n)-M^{>}(n-1)|\leq h_n\leq
 \begin{cases}
  1,& n\text{ is not the square of a prime},\\
  |M(p-1)|,& n=p^2,\ p\text{ prime}.
 \end{cases}
$$
Using \eqref{eq:LTW} yields $$h_n=O\big(n^{1/4}(\log\log n)^3\big),\quad n\to\infty\quad\text{a.s.}$$ Together with the deterministic lower bound $Q(n)\geq n/(3\log n)$ for all sufficiently large $n$ (see \eqref{eq:deterministic-Q-bounds}), this implies
\begin{equation}\label{eq:hn}
\lim_{n\to\infty}\frac{h_n}{(Q(n))^{1/2}}=0\quad\text{a.s.}
\end{equation}

The triangle inequality in $\ell^2$ (the space of square summable sequences) yields $\big|(Q(n))^{1/2}-(Q(n-1))^{1/2}\big|\leq h_n$. Thus \eqref{eq:hn} and $\lim_{n\to\infty} Q(n)=\infty$ a.s.\
give $\lim_{n\to\infty}(Q(n-1)/Q(n))=1$ a.s. and thereupon, with $D(x):=(2Q(x)\log\log Q(x))^{1/2}$ for all sufficiently large $x$, $\lim_{n\to\infty}(D(n-1)/D(n))=1$ a.s. because $\log\log Q(x)\sim \log\log x$ as $x\to\infty$ (see Lemma \ref{lem:deterministic-Q-bounds}). Writing $Z_n=M^{>}(n)/D(n)$, we have
$$Z_n-Z_{n-1}=\frac{M^{>}(n)-M^{>}(n-1)}{D(n)}+Z_{n-1}\Big(\frac{D(n-1)}{D(n)}-1\Big).$$ The sequence $(Z_n)$ is eventually bounded by Theorem \ref{thm:main}. Therefore,
\begin{equation}\label{eq:zn}
\lim_{n\to\infty} (Z_n-Z_{n-1})=0\quad\text{a.s.}
\end{equation}

Fix a sample point at which both limit relations of Theorem \ref{thm:main} and \eqref{eq:zn} hold. For any $z\in(-1,1)$, the sequence $(Z_n)$ crosses $z$ upwards i.o. Hence, there exists an increasing sequence
$(n_k)$, depending on the chosen sample point, such that $Z_{n_k-1}\leq z<Z_{n_k}$. Therefore, $$0<Z_{n_k}-z\leq |Z_{n_k}-Z_{n_k-1}|~\to~0,\quad k\to\infty,$$ so that $z$ is a subsequential limit.
The endpoints $-1$ and $1$ are subsequential limits by Theorem \ref{thm:main}, and the same relation excludes
subsequential limits outside $[-1,1]$. We conclude that \eqref{eq:cluster} holds.

\section{Preparation for the proof of \eqref{eq:covinput}}\label{sect:proofnovel}

We start with an outline of the proof of \eqref{eq:covinput}. Lemma~\ref{lem:perron} represents $C_r(x,y)(xy)^{-1/2}$, up to a negligible error, as a double integral involving the random Euler product $F_{a_r}$. Lemma~\ref{lem:truncation} shows that, with probability tending to one, this double integral can be replaced by the integral $I_r(x,y)$ over pairs of nearby arguments at which the partial Euler products stay below the so-called strong barriers. The integrals $I_r(x,y)$ are controlled simultaneously for all $x,y\in\mathcal X_r$ in Lemma~\ref{lem:highmoment} with the help of high moments, using the moment bound of Lemma~\ref{lem:cell} for the Euler products subject to the strong barriers. The barrier estimates are obtained in Section~\ref{sec:barriers} and in the three subsequent sections
by means of a change of measure and a comparison with a Gaussian random walk. Finally, \eqref{eq:covinput} is deduced from Lemmas~\ref{lem:truncation} and~\ref{lem:highmoment} in Section~\ref{sect:proofcov}.

\subsection{Standard analytic estimates} We start with an estimate for a sum over primes which is secured by the prime number theorem with a remainder.
\begin{lemma}\label{lem:osc-prime-sum}
There exist absolute constants $c',C'>0$ such that, for all
$u>s\ge3$ and $v\in\mr$,
\begin{equation}\label{eq:osc}
\left|\sum_{s<p\le u}p^{-1-\ii v}-\int_{\log s}^{\log u}\frac{\eee^{-\ii vy}}{y}\,\dd y\right| \le C'(1+|v|)\eee^{-c'(\log s)^{1/2}}.
\end{equation}
\end{lemma}
\begin{proof}
Recall the notation $\pi(x)=\#\{p:~p\leq x\}$ for $x\geq 0$ for the prime counting function.
Representing the sum as a Lebesgue-Stieltjes integral with respect to $\pi$ and using \eqref{eq:pix}, we obtain
$$\sum_{s<p\le u}p^{-1-{\rm i}v}=\int_s^u\frac{z^{-1-{\rm i}v}}{\log z}\,{\rm d}z+\int_{(s,u]}z^{-1-{\rm i}v}\,{\rm d}R(z)=\int_{\log s}^{\log u}\frac{\eee^{-\ii vy}}{y}\,\dd y+\int_{(s,u]}z^{-1-{\rm i}v}\,{\rm d}R(z).$$
Integration by parts, in combination with \eqref{eq:Restim}, gives
\begin{multline*}
\left|\int_{(s,u]}z^{-1-{\rm i}v}\,{\rm d}R(z)\right|=\left|u^{-1-{\rm i}v}R(u)-s^{-1-{\rm i}v}R(s)+(1+{\rm i}v)\int_s^u R(z)z^{-2-{\rm i}v}\,{\rm d}z\right|\\\leq 2c_1 \eee^{-c(\log s)^{1/2}}+c_1(1+|v|)\int_s^\infty z^{-1}\eee^{-c(\log z)^{1/2}}{\rm d}z
\end{multline*}
for all $s,u$ satisfying $u>s\geq 3$ and all $v\in\mr$. Noting that
$$\int_s^\infty z^{-1}\eee^{-c(\log z)^{1/2}}{\rm d}z=2c^{-2}(1+c(\log s)^{1/2})\eee^{-c(\log s)^{1/2}}\leq C_1^\prime\eee^{-c'(\log s)^{1/2}}$$ for any fixed $c'\in(0,c)$ and all $s\geq 3$, with an appropriate
constant $C_1'>0$ independent of $s$, we arrive at \eqref{eq:osc}.
\end{proof}

For $v\ne0$, integration by parts yields
\begin{equation}\label{eq:oscBound}
\left|\int_{\log s}^{\log u}\frac{\eee^{-\ii vy}}y\,\dd y\right|\leq \frac{3}{|v|\log s}.
\end{equation}

The asymptotic relations below are standard consequences of Mertens’ estimates.
\begin{lemma}\label{cor:Mert-consequences}
As $q\to\infty$,
\begin{equation}\label{eq:Mert conseq}
\prod_{p\le \eee^q}(1+p^{-1})=6\pi^{-2}\eee^\gamma q +O(1),
\quad
\sum_{p\le \eee^q}\frac{\log p}{p}=q+O(1),
\quad
\sum_{p\le \eee^q}\frac{(\log p)^2}{p}=O(q^2),
\end{equation}
where $\gamma$ is the Euler-Mascheroni constant.
\end{lemma}
\begin{proof}
By Theorem~2.7(e) in~\cite{MV}, $\prod_{p\le z}(1-p^{-1})^{-1}=\eee^\gamma\log z+O(1)$. It is a standard fact that $$\prod_{p\le z}(1-p^{-2})=6 \pi^{-2}+O\left(\sum_{p>z}p^{-2}\right)=6 \pi^{-2}+O(z^{-1}).$$
The first asymptotic relation in \eqref{eq:Mert conseq} follows from these and a representation $$\prod_{p\le z}(1+p^{-1})=\prod_{p\le z}(1-p^{-2})\prod_{p\le z}(1-p^{-1})^{-1}.$$ The second asymptotic relation is just Theorem~2.7 (b) in~\cite{MV}. The third follows from the second with the help of $(\log p)^2\leq q\log p$ for $p\leq \eee^q$.
\end{proof}

For each $z>\eee^\eee$ and $j\in\mn_0$, put $J_*(z):=\floor{\log\log z-1}$ and $q_{j,z}=\eee^{-j}\log z$. Define now the finite products
\begin{equation}\label{eq:Euler}
F_{j,z}(t)=\prod_{p\le\exp(q_{j,z})}(1+f(p)p^{-1/2-\ii t}), \quad j\in\mathbb{Z},~ 0\le j\le J_*(z),~ t\in\mr.
\end{equation}
Each product is nonzero because $p^{-1/2}<1$. Recalling that $P(n)$ denotes the largest prime factor of an integer $n\geq 2$ and that $P(1)=1$, we expand the product to obtain
\begin{equation}\label{eq:expan}
F_{j,z}(t)=\sum_{\substack{n\geq 1:~n~\text{squarefree}\\ P(n)\le \eee^{q_{j,z}}}}\frac{f(n)}{n^{1/2+{\rm i}t}}.
\end{equation}
Using orthogonality of $f(2)$, $f(3),\ldots$ for the first equality and the first part of~\eqref{eq:Mert conseq} for the asymptotic relation we obtain
\begin{equation}\label{eq:EulerSecond}
\me\big[|F_{j,z}(t)|^2\big]=\sum_{\substack{n\geq 1:~n~\text{squarefree}\\ P(n)\le \eee^{q_{j,z}}}} n^{-1}=\prod_{p\le \eee^{q_{j,z}}}(1+p^{-1})= 6\pi^{-2}\eee^\gamma q_{j,z}+O(1)
\end{equation}
with an absolute $O(1)$-constant.

We shall write for simplicity $F_z$ for $F_{0,z}$ and state explicitly that $$F_z(t)=\prod_{p\le z}\left(1+f(p)p^{-1/2-\ii t}\right).$$
\begin{lemma}\label{lem:twofrequency}
There exists an absolute constant $K>0$ such that, for all $z>\eee^\eee$ and $t,w\in\mr$,
\begin{multline}\label{eq:twofrequency}
\me\big[|F_z(t)F_z(w)|\big]\leq K(\log z)^{1/2}\Big(\min\,(\log z,|t-w|^{-1})^{1/2}\\+\min\,(\log z,|t+w|^{-1})^{1/2}+\log(2+|t|+|w|)\Big),
\end{multline}
where $1/0:=\infty$.
\end{lemma}
\begin{proof}
Put $R:=2+|t|+|w|$. For $z\ge 200$, apply the second part of Euler Product Result~2 on pp.~2107--2108 in~ \cite{Harper:2023} with $x=200$, $y=z$, $\sigma=0$, $t_1=t$ and $t_2=w$. Since
$$\prod_{p\le200}|1+f(p)p^{-1/2-\ii t}||1+f(p)p^{-1/2-\ii w}|\leq \prod_{p\le200}(1+p^{-1/2})^2,$$ that result implies the existence of an absolute constant $K_1>0$ such that $$\me\big[|F_z(t)F_z(w)|\big] \leq K_1(\log z)^{1/2}\left(1+\min\,\left(\log z,~|t-w|^{-1}+|t+w|^{-1}+(\log R)^2\right)\right)^{1/2}.$$

Using the inequality $\min(L,a+b+c)\leq\min(L,a)+\min(L,b)+c$ which holds for $L>0$ and $a,b,c\ge0$, together with subadditivity of $x\mapsto x^{1/2}$ on $[0,\infty)$ we infer
\begin{multline*}
\left(1+\min\left(\log z,~|t-w|^{-1}+|t+w|^{-1}+(\log R)^2\right)\right)^{1/2}\\\leq 1+\min(\log z,~|t-w|^{-1})^{1/2}+\min(\log z,|t+w|^{-1})^{1/2}+\log R\\ \leq (1+1/\log2)
\left(\min(\log z,|t-w|^{-1})^{1/2}+\min(\log z,|t+w|^{-1})^{1/2}+\log R\right).
\end{multline*}
We have used $R\ge2$ for the last inequality. These inequalities remain valid when $a$ or $b$ is infinite, with the convention in the statement. This proves \eqref{eq:twofrequency} for $z\ge 200$ with
$K=K_1(1+1/\log2)$.

For $\eee^\eee<z<200$, $\me\big[|F_z(t)F_z(w)|\big] \leq B:=\prod_{p\le200}(1+p^{-1/2})^2$, whereas $(\log z)^{1/2}\log R\geq \eee^{1/2}\log2$. Thus choosing  $$K:=\max\left(K_1\left(1+\frac1{\log2}\right),~\frac{B}{\eee^{1/2}\log2}\right)$$ establishes \eqref{eq:twofrequency} throughout the range $z>\eee^\eee$ and $t,w\in\mr$.
\end{proof}

For $r\in\mn$, put $H_r:=T_r^8$, $\omega_r:=(\log\log a_r)^4/T_r$ and $
K_r(v):=\sum_{a_r<p\leq Y_r}p^{-1-\ii v}$ for $v\in\mr$.
Recall that $Y_r=a_r^2$ and
$\log a_r=(T_r+T_r^{1-\lambda})/2$. We claim that
\begin{equation}\label{eq:kernelBounds}
|K_r(v)|\leq C\min\left(1,\frac{1}{T_r|v|}\right),
\quad |v|\leq 2H_r
\end{equation}
and
\begin{equation}\label{eq:kernelBounds2}
\int_{-\omega_r}^{\omega_r}|K_r(v)|\,\dd v \leq 10C\frac{\log\log T_r}{T_r}
\end{equation}
for an absolute constant $C>0$ and all sufficiently large $r$, with the minimum interpreted as $1$ at $v=0$. Indeed, the integral appearing in \eqref{eq:osc}, with $s=a_r$ and $u=Y_r$, is bounded
in absolute value by $\log\log Y_r-\log\log a_r=\log 2$ and, for $v\neq0$, by $3/(|v|\log a_r)$, see \eqref{eq:oscBound}. The right-hand side of \eqref{eq:osc}, uniformly in $|v|\leq 2H_r$, is
$$O\left(T_r^8\exp(-c'(\log a_r)^{1/2})\right)=o(T_r^{-9}), \quad r\to\infty.$$ Therefore, it can be absorbed into the bound in \eqref{eq:kernelBounds}. Integrating this bound separately over the regions $|v|\leq T_r^{-1}$ and $T_r^{-1}<|v|\leq\omega_r$ proves \eqref{eq:kernelBounds2}. For instance, $$\int_{T_r^{-1}<|v|\leq \omega_r}|K_r(v)|\,\dd v\leq \frac{2C}{T_r}\int_{T_r^{-1}}^{\omega_r}\frac{{\rm d}v}{v}=\frac{8C}{T_r}\log\log ((T_r+T_r^{1-\lambda})/2)\leq 8C \frac{\log\log T_r}{T_r}.$$

Lemma \ref{lem:perron} represents the conditional covariances by double integrals involving random Euler products, with an error uniform over $\mathcal X_r$. Related representations are derived in Section~3.2 of \cite{Harper:2023} and Lemma~5.2 of \cite{Gou+Xu:2026}. Since those results do not directly cover our setting, we provide a proof for $\mathcal X_r$, using orthogonality to control the quadratic truncation error directly on the line $\Re s=1/2$.

Recall the notation  $$C_r(x,y)=\sum_{a_r<p\leq\min(x,y)}M(x/p)M(y/p), \quad x,y\in\mathcal X_r.$$
\begin{lemma}\label{lem:perron}
There exist $\mathcal H_r$-measurable events $\mathcal P_r$ with $\lim_{r\to\infty}\mmp(\mathcal P_r)=1$ such that, on $\mathcal P_r$, simultaneously for all $x,y\in\mathcal X_r$,
\begin{equation}\label{eq:perron}
\frac{C_r(x,y)}{(xy)^{1/2}}=\frac{1}{4\pi^2}\int_{-H_r}^{H_r}\int_{-H_r}^{H_r}\frac{F_{a_r}(v)\overline{F_{a_r}(t)}x^{\ii v}y^{-\ii t}}{(1/2+\ii v)(1/2-\ii t)}K_r(v-t)\,\dd v\,\dd t+O(T_r^{-1}),
\end{equation}
where $\overline{x}$ denotes the complex conjugate of $x\in\mathbb{C}$. The implicit constant is deterministic and independent of $x,y$ and the realization in $\mathcal P_r$.
\end{lemma}
\begin{proof}
For $h\geq 1$ and $z>0$, put $$P_h(z):=\frac{1}{2\pi}\int_{-h}^{h}\frac{z^{1/2+\ii t}}{1/2+\ii t}\,\dd t\quad \text{and}\quad \Delta_h(z):=\1_{[1,\infty)}(z)-P_h(z).$$ We claim that there exists an absolute constant $C>0$ such that
\begin{equation}\label{eq:Delta-bound}
|\Delta_h(z)|\leq C z^{1/2}\min\left(1,\frac1{h|\log z|}\right),
\quad h\geq 1,\quad z>0,
\end{equation}
where the minimum is interpreted as $1$ at $z=1$.

\noindent {\sc Proof of \eqref{eq:Delta-bound}}. Assume first that $h|\log z|>1$, so that $z\neq 1$. An application of Theorem~5.1 on p.~138 in~ \cite{MV}, with $a_1=1$,
$a_n=0$ for $n\geq 2$, $\sigma_0=1/2$ and $x=z$ yields $$\1_{[1,\infty)}(z)=\lim_{R\to\infty}\frac{1}{2\pi}\int_{-R}^{R}\frac{z^{1/2+\ii t}}{1/2+\ii t}\,\dd t.$$ Here, the half-weight convention present in the theorem has no effect because $z\neq 1$ and all coefficients other than $a_1$ vanish. The preceding equality entails $$\Delta_h(z)=\frac{z^{1/2}}{2\pi}\lim_{R\to\infty}\int_{h\leq |t|\leq R}\frac{\eee^{\ii t\log z}}{1/2+\ii t}\,\dd t.$$
For $R>h$, integration by parts gives $$\left|\int_h^R\frac{\eee^{\ii t\log z}}{1/2+\ii t}\,\dd t\right|\leq\frac1{|\log z|}\left(\frac1R+\frac1h+\int_h^R\frac{\dd t}{t^2}\right)=\frac2{h|\log z|}.$$ The same bound holds for the integral over $[-R,-h]$. Hence, $$|\Delta_h(z)|\leq \frac{2z^{1/2}}{\pi h|\log z|},$$
which proves \eqref{eq:Delta-bound} in this case.

Now assume that $h|\log z|\leq 1$. After simple manipulations, we obtain
$$P_h(z)=\frac{z^{1/2}}{\pi}\int_0^h\frac{\frac12\cos(t\log z)+t\sin(t\log z)}{1/4+t^2}\,\dd t.$$
Using $|\sin(t\log z)|\leq t|\log z|$ for $t\geq 0$, we infer
\begin{multline*}
|P_h(z)|\leq\frac{z^{1/2}}{\pi}\left(\frac{1}{2}\int_0^h \frac{\dd t}{1/4+t^2}+|\log z|\int_0^h\frac{t^2}{1/4+t^2}\,\dd t\right)\\\leq \frac{z^{1/2}}{\pi}\left(\frac{\pi}{2}+h|\log z|\right)\leq \left(\frac{1}{2}+\frac{1}{\pi}\right)z^{1/2}.
\end{multline*}
Since $\1_{[1,\infty)}(z)\leq z^{1/2}$ for all $z>0$, this entails $|\Delta_h(z)|\leq (3/2+1/\pi) z^{1/2}$ and proves \eqref{eq:Delta-bound} also in the second case, including $z=1$.

For $x\in\mathcal X_r$ and $a_r<p\leq Y_r$, put $$A_{r,p}(x):=x^{-1/2}M(x/p)\quad\text{and}\quad B_{r,p}(x):=\frac{1}{2\pi p^{1/2}}
\int_{-H_r}^{H_r}\frac{F_{a_r}(t)(x/p)^{\ii t}}{1/2+\ii t}\,\dd t$$ and denote the corresponding vectors by $\boldsymbol A_r(x):=(A_{r,p}(x))_{a_r<p\leq Y_r}$ and $\boldsymbol B_r(x):=(B_{r,p}(x))_{a_r<p\leq Y_r}$. Using \eqref{eq:expan} with $j=0$ and $z=a_r$, we conclude that $$B_{r,p}(x)
=x^{-1/2}
\sum_{\substack{n\geq 1:~ n~ \mathrm{squarefree}\\
                 P(n)\leq a_r}}
f(n)P_{H_r}\left(x/(pn)\right).$$
Since $x\leq Y_r=a_r^2$ and $p>a_r$, we have $x/p<a_r$. Hence, $$A_{r,p}(x)=x^{-1/2}\sum_{\substack{n\geq1:\ n\ \mathrm{squarefree}\\ P(n)\leq a_r}} f(n)\1_{[1,\infty)}(x/(pn))$$ and thereupon
$$A_{r,p}(x)-B_{r,p}(x)=x^{-1/2}\sum_{\substack{n\geq1:~ n~ \mathrm{squarefree}\\ P(n)\leq a_r}} f(n)\Delta_{H_r}\left(x/(pn)\right).$$
All sums in these representations are finite.

By orthogonality of $f(2)$, $f(3),\ldots$ and \eqref{eq:Delta-bound},
\begin{multline*}
\me\left[|A_{r,p}(x)-B_{r,p}(x)|^2\right]
=\frac{1}{x}\sum_{\substack{n\geq1:\ n\ \mathrm{squarefree}\\ P(n)\leq a_r}}\left|\Delta_{H_r}\left(\frac{x}{pn}\right)\right|^2\\
\leq C^2 \sum_{\substack{n\geq1:\ n\ \mathrm{squarefree}\\ P(n)\leq a_r}} \frac{1}{pn}\min\left(1,\frac1{H_r^2|\log(x/(pn))|^2}\right).
\end{multline*}
Summing over $p$ and writing $\|\cdot\|_2$ for the Euclidean norm, we obtain
\begin{equation}\label{eq:quadratic-Perron-error}
\me\left[\|\boldsymbol A_r(x)-\boldsymbol B_r(x)\|_2^2\right]
\leq C^2 \mathcal R_r(x),
\end{equation}
where
$$\mathcal R_r(x):=\sum_{a_r<p\leq Y_r}\sum_{\substack{n\geq1:\ n\ \mathrm{squarefree}\\ P(n)\leq a_r}} \frac{1}{pn}\min\left(1,\frac1{H_r^2|\log(x/(pn))|^2}\right).$$ Here and hereafter, minimum involving an expression with a vanishing denominator is interpreted as $1$.

We next prove that, uniformly in $x\in\mathcal X_r$,
\begin{equation}\label{eq:R-x-bound}
\mathcal R_r(x)=O(H_r^{-1}),\quad r\to\infty.
\end{equation}
To this end, we treat separately the ranges $x/\eee \leq pn\leq\eee x$ and $pn<x/\eee$ or $pn>\eee x$. In the first range, $p$ is the unique prime factor of
the integer $pn$ exceeding $a_r$. Thus,
\begin{multline*}
\sum_{a_r<p\leq Y_r}\sum_{\substack{n\geq1:\ n\ \mathrm{squarefree}\\ P(n)\leq a_r,\ x/\eee\leq pn\leq\eee x}}\frac{1}{pn}\min\left(1,\frac1{H_r^2|\log(x/(pn))|^2}\right)\\ \leq
\sum_{\substack{m\in\mn\\x/\eee\leq m\leq\eee x}}\frac{1}{m} \min\left(1,\frac1{H_r^2|\log(x/m)|^2}\right) \leq\frac{\eee}{x}\sum_{\substack{m\in\mn\\x/\eee\leq m\leq\eee x}}
\min\left(1,\frac{\eee^2x^2}{H_r^2|x-m|^2}\right)\\ \leq\frac{2\eee}{x} \left(1+\sum_{j\geq1} \min\left(1,\frac{\eee^2x^2}{H_r^2j^2}\right)\right)\\ \leq\frac{2\eee^3}{x}
\left(1+\sum_{j\geq1}\min\left(1,\frac{x^2}{H_r^2j^2}\right)\right).
\end{multline*}
Here, the second inequality follows from $m^{-1}\leq\eee/x$ and $|\log(x/m)|\geq|x-m|/(\eee x)$, which is justified by the mean value theorem for differentiable functions.
For the third inequality, we have used
$$
\#\{m\in\mn:|m-x|<1\}\leq 2,\qquad
\#\{m\in\mn:j\leq|m-x|<j+1\}\leq2,\quad j\in\mn.
$$
The inequality $x\geq\eee^{T_r}$, secured by $x\in\mathcal X_r$, implies that $\lim_{r\to\infty}(x/H_r)=\infty$ uniformly in $x\in\mathcal X_r$. Splitting the last sum at $j=\lfloor x/H_r\rfloor$ gives
$$\sum_{j\geq1}\min\left(1,\frac{x^2}{H_r^2j^2}\right)\leq\frac{x}{H_r}+\frac{x^2}{H_r^2}\sum_{j>x/H_r}j^{-2}=O(x/H_r).$$ Thus the contribution from the first range is $O(H_r^{-1})$,
uniformly in $x\in\mathcal X_r$.

In the second range, $|\log(x/(pn))|>1$. This ensures that the contribution of the second range to $\mathcal R_r(x)$ does not exceed
\begin{multline*}
H_r^{-2}\sum_{a_r<p\leq Y_r}p^{-1}\sum_{\substack{n\geq1:~ n~\mathrm{squarefree}\\ P(n)\leq a_r}}n^{-1}\\=H_r^{-2}\sum_{a_r<p\leq a_r^2}p^{-1}\prod_{q\leq a_r}(1+q^{-1})=O(T_rH_r^{-2})=o(H_r^{-1}),\quad r\to\infty.
\end{multline*}
Here, we have used \eqref{eq:mertens}, \eqref{eq:Mert conseq} and $\log a_r=O(T_r)$. This proves \eqref{eq:R-x-bound}. Now \eqref{eq:R-x-bound}, in combination with \eqref{eq:quadratic-Perron-error},
secures
\begin{equation}\label{eq:ab-L2-error}
\me\left[\|\boldsymbol A_r(x)-\boldsymbol B_r(x)\|_2^2\right]
=O(H_r^{-1})
\end{equation}
uniformly in $x\in\mathcal X_r$.

Since $|\mathcal X_r|=\lfloor T_r^{1-\lambda}\rfloor+1\leq 2T_r$, Markov's inequality and the union bound entail
\begin{multline}\label{eq:ab-uniform}
\mmp\left\{\max_{x\in\mathcal X_r}\|\boldsymbol A_r(x)-\boldsymbol B_r(x)\|_2>T_r^{-2}\right\}\\\leq T_r^4
\sum_{x\in\mathcal X_r}
\me\left[\|\boldsymbol A_r(x)-\boldsymbol B_r(x)\|_2^2\right]
=O(T_r^5/H_r)=O(T_r^{-3}).
\end{multline}
Furthermore, by \eqref{eq:orthogon} and \eqref{eq:mertens}, $$\me\left[\|\boldsymbol A_r(x)\|_2^2\right]=x^{-1}\sum_{a_r<p\leq Y_r}\me[(M(x/p))^2]\leq\sum_{a_r<p\leq Y_r}p^{-1}=O(1)$$ uniformly in $x\in\mathcal X_r$. Another application of Markov's inequality and the union bound gives
\begin{equation}\label{eq:A-uniform}
\mmp\left\{\max_{x\in\mathcal X_r}\|\boldsymbol A_r(x)\|_2>T_r\right\}=O(T_r^{-1}).
\end{equation}

For $r\in\mn$, define the events
$$\mathcal P_r:=
\left\{
\max_{x\in\mathcal X_r}
\|\boldsymbol A_r(x)-\boldsymbol B_r(x)\|_2\leq T_r^{-2},
\quad
\max_{x\in\mathcal X_r}\|\boldsymbol A_r(x)\|_2\leq T_r
\right\}.$$
All coordinates of the vectors involved are $\mathcal H_r$-measurable, hence so is $\mathcal P_r$. It follows from \eqref{eq:ab-uniform} and \eqref{eq:A-uniform} that $\mmp(\mathcal P_r^c)=O(T_r^{-1})$ as $r\to\infty$.

For complex vectors, we shall use the notation $\langle a,b\rangle=\sum_p a_p\overline{b_p}$. On $\mathcal P_r$, the Cauchy-Schwarz inequality ensures that,
simultaneously for all $x,y\in\mathcal X_r$,
\begin{multline*}
\left|\langle\boldsymbol B_r(x),\boldsymbol B_r(y)\rangle-\langle\boldsymbol A_r(x),\boldsymbol A_r(y)\rangle\right|\\\leq \left|\langle\boldsymbol B_r(x)-\boldsymbol A_r(x),\boldsymbol A_r(y)\rangle\right|+\left|\langle \boldsymbol A_r(x),\boldsymbol B_r(y)-\boldsymbol A_r(y)\rangle\right|+\left|\langle\boldsymbol B_r(x)-\boldsymbol A_r(x),\boldsymbol B_r(y)-\boldsymbol A_r(y)\rangle\right|\\ \leq
\|\boldsymbol B_r(x)-\boldsymbol A_r(x)\|_2 \|\boldsymbol A_r(y)\|_2+\|\boldsymbol A_r(x)\|_2\|\boldsymbol B_r(y)-\boldsymbol A_r(y)\|_2+\|\boldsymbol B_r(x)-\boldsymbol A_r(x)\|_2 \|\boldsymbol B_r(y)-\boldsymbol A_r(y)\|_2\\\leq 2T_r^{-1}+T_r^{-4}.
\end{multline*}
Since $M(u)=0$ for $u<1$,
$$\langle\boldsymbol A_r(x),\boldsymbol A_r(y)\rangle =\frac{C_r(x,y)}{(xy)^{1/2}}.$$
On the other hand, substituting the definitions of the coordinates of $\boldsymbol B_r(x)$ and $\boldsymbol B_r(y)$, interchanging the finite prime sum with the integrals and using $\overline{\frac{y^{{\rm i}t}}{1/2+\ii t}}=\frac{y^{-{\rm i}t}}{1/2-\ii t}$ we finally obtain
\begin{equation*}
\langle\boldsymbol B_r(x),\boldsymbol B_r(y)\rangle=\frac1{4\pi^2}\int_{-H_r}^{H_r}\int_{-H_r}^{H_r}\frac{F_{a_r}(v)\overline{F_{a_r}(t)}x^{\ii v}y^{-\ii t}}{(1/2+\ii v)(1/2-\ii t)}
\left(\sum_{a_r<p\leq Y_r}p^{-1-\ii(v-t)}\right)\,\dd v\,\dd t.
\end{equation*}
The sum in parentheses is $K_r(v-t)$. Combining the last three displays proves \eqref{eq:perron}.
\end{proof}

\subsection{Barriers for the partial Euler products}\label{sec:barriers}

Recall the notation $J_*(a_r)=\lfloor \log\log a_r-1\rfloor$, $q_{j,a_r}=\eee^{-j}\log a_r$ for $j\in\mn_0$ and $$F_{j,a_r}(t)=\prod_{p\leq\exp(q_{j,a_r})}\left(1+f(p)p^{-1/2-\ii t}\right),\quad 0\leq j\leq J_*(a_r),~ t\in\mr.$$ In particular, $F_{0,a_r}(t)=F_{a_r}(t)$.

In this section we impose upper bounds, also called {\it barriers}, on the moduli of the partial Euler products $F_{j,a_r}$ associated with the Euler product $F_{a_r}$ appearing in \eqref{eq:perron}. We shall show that the part of the integral in \eqref{eq:perron} where the modulus of at least one partial Euler product exceeds its barrier makes a negligible contribution to $C_r(x,y)(xy)^{-1/2}$. Related barrier estimates are given in Multiplicative Chaos Results~1 and~2 in Section~2.4 of \cite{Harper:2023} and used in Section~3.3 of that paper. We provide proofs for the particular barriers needed here.

Fix $B\geq1$ and $b\in(0,1)$. For each $t\in\mr$, define the $\mathcal H_r$-measurable events
\begin{equation}\label{eq:weakbarrier}
\mathcal D_r(t):=\bigcap_{0\leq j\leq J_*(a_r)}\left\{|F_{j,a_r}(t)|\leq q_{j,a_r}(\log\log a_r)^2\right\}
\end{equation}
and
\begin{equation}\label{eq:strongbarrier}
\mathcal A_r(t):=\mathcal D_r(t)\cap
\bigcap_{0\leq j\leq\lfloor(1-b)\log\log a_r\rfloor}\left\{|F_{j,a_r}(t)|\leq q_{j,a_r}(\log\log a_r)^{-B}\right\}.
\end{equation}
The second event is defined for all sufficiently large $r$, satisfying $\lfloor(1-b)\log\log a_r \rfloor\leq J_*(a_r)$. We refer to the quantities $q_{j,a_r}(\log\log a_r)^2$ and $q_{j,a_r}(\log\log a_r)^{-B}$ as the {\it weak} and {\it strong barriers}, respectively. The constants $B$ and $b$ will be chosen later.

Lemma~\ref{lem:weak}, proved using a derivative estimate, shows that, with probability tending to one, the modulus of each partial Euler product remains below its weak barrier for all arguments belonging to the prescribed interval. Lemma~\ref{lem:weightedBarrier} uses a change of measure and a comparison of barrier probabilities with those of the Gaussian random walk, together with Lemma~\ref{lem:strip}, to bound the expectation of $|F_{a_r}(t)|^2$ over realizations satisfying the weak barriers but violating at least one strong barrier.
\begin{lemma}\label{lem:weak}
There exists an absolute constant $C>0$ such that, for all sufficiently large $r$,
\begin{multline*}
\mmp\big\{|F_{j,a_r}(t)|\leq q_{j,a_r}(\log\log a_r)^2\text{ for all }0\leq j\leq J_*(a_r)\text{ and }|t|\leq (\log\log a_r)^2+2
\big\}\\\geq 1-C/\log\log a_r.
\end{multline*}
\end{lemma}
\begin{proof}
Fix $0\leq j\leq J_*(a_r)$ and put $q:=q_{j,a_r}$.
Differentiating in $t$ the finite expansion \eqref{eq:expan}, we obtain $$F_{j,a_r}'(t)=-\ii\sum_{\substack{n\geq1:~n\ \mathrm{squarefree}\\ P(n)\leq\eee^q}}\frac{f(n)\log n}{n^{1/2+\ii t}}.$$
By orthogonality of $f(2)$, $f(3),\ldots$,
$$\me\big[|F_{j,a_r}'(t)|^2\big]=\sum_{\substack{n\geq1:\ n\ \mathrm{squarefree}\\ P(n)\leq\eee^q}}\frac{(\log n)^2}{n}.$$ Since $n$ is squarefree,
$$(\log n)^2=\sum_{p\mid n}(\log p)^2
+2\sum_{\substack{p<q'\\p,q'\mid n}}\log p\log q'.
$$
Here and hereafter, $p$, $q'$ and $\ell$ denote primes. Hence, $$\sum_{\substack{n\geq1:~ n~ \mathrm{squarefree}\\P(n)\leq \eee^q}}\frac{(\log n)^2}{n}=\sum_{p\leq \eee^q}(\log p)^2 \sum_{\substack{n~\mathrm{squarefree}\\P(n)\leq \eee^q,~ p\mid n}}\frac{1}{n}+2\sum_{p<q'\leq \eee^q}\log p\log q' \sum_{\substack{n~\mathrm{squarefree}\\P(n)\leq \eee^q,~ p,q'\mid n}}\frac{1}{n}.$$ For a fixed prime $p\leq \eee^q$,
$$\sum_{\substack{n~\mathrm{squarefree}\\ P(n)\leq \eee^q,~ p\mid n}}\frac{1}{n}=\frac{1}{p}\prod_{\substack{\ell\leq \eee^q\\\ell\neq p}}(1+\ell^{-1})=\frac{1}{p+1}\prod_{\ell\leq \eee^q}(1+\ell^{-1}).$$  Similarly, requiring both $p\mid n$ and $q'\mid n$ gives the factor $1/((p+1)(q'+1))$. Summarizing,
\begin{multline*}
\me\big[|F_{j,a_r}'(t)|^2\big]=\prod_{p\leq\eee^q}(1+p^{-1})\Big(\sum_{\ell \leq\eee^q}\frac{(\log \ell)^2}{\ell+1}+2\sum_{\ell<q'\leq\eee^q}
\frac{\log \ell\log q'}{(\ell+1)(q'+1)}\Big)\\\leq \prod_{p\leq\eee^q}(1+p^{-1})
\Big(\sum_{\ell \leq\eee^q}\frac{(\log \ell)^2}{\ell}+\Big(\sum_{\ell\leq\eee^q}\frac{\log \ell}{\ell}\Big)^2\Big).
\end{multline*}
Invoking \eqref{eq:Mert conseq}, we conclude that
\begin{equation}\label{eq:Fj-derivative-bound}
\me\big[|F_{j,a_r}'(t)|^2\big]\leq C_1q^3,
\end{equation}
where $C_1>0$ is a constant independent of $r,j$ and $t$.

Let $I$ be a closed interval of length one. For each realization,
there exists $s\in I$ such that $$|F_{j,a_r}(s)|^2\leq\int_I|F_{j,a_r}(u)|^2\,\dd u.$$ According to the fundamental theorem of calculus, for each $t\in I$,
$$|F_{j,a_r}(t)|^2= |F_{j,a_r}(s)|^2+ \int_s^t \frac{\dd}{\dd u}|F_{j,a_r}(u)|^2\,\dd u.$$ This, together with $$\left|\frac{\dd}{\dd u}|F_{j,a_r}(u)|^2\right|=2\left|{\rm Re}\big(F_{j,a_r}'(u)\overline{F_{j,a_r}(u)}\big)\right|\le 2|F_{j,a_r}(u)||F_{j,a_r}'(u)|,$$ entails $$\sup_{t\in I}|F_{j,a_r}(t)|^2 \leq\int_I|F_{j,a_r}(u)|^2\,\dd u+2\int_I|F_{j,a_r}(u)||F_{j,a_r}'(u)|\,\dd u.$$
Taking expectations and using the Cauchy--Schwarz inequality, \eqref{eq:EulerSecond} and \eqref{eq:Fj-derivative-bound}, we obtain
\begin{multline*}
\me\big[\sup_{t\in I}|F_{j,a_r}(t)|^2\big]\leq \int_I\me\big[|F_{j,a_r}(u)|^2\big]\,\dd u + 2\int_I \big(\me\big[|F_{j,a_r}(u)|^2\big]\big)^{1/2}\big(\me\big[|F_{j,a_r}'(u)|^2\big]\big)^{1/2}
\,\dd u \\ \leq C_2(q+q^2)\leq 2C_2 q^2
\end{multline*}
having utilized $q\geq \eee$, where $C_2>0$ is a constant independent of $r$, $j$ and $t$.

The interval $[-(\log\log a_r)^2-2,(\log\log a_r)^2+2]$ can be covered by $\lfloor 2(\log\log a_r)^2\rfloor+5$ closed intervals of length one. For each such interval $I$, Markov's inequality yields $$\mmp\big\{\sup_{t\in I}|F_{j,a_r}(t)|>q_{j,a_r}(\log\log a_r)^2\big\}\leq 2C_2(\log\log a_r)^{-4}.$$ Since there are $J_*(a_r)+1$ possible values of $j$ and $J_*(a_r)+1\leq \log\log a_r$, an application of the union bound yields, for all sufficiently large $r$, with $C=16C_2$, $$\mmp\big\{\exists~|t|\le (\log\log a_r)^2+2,~ \exists~ 0\le j\le J_*(a_r):~ |F_{j,a_r}(t)|>q_{j,a_r}(\log\log a_r)^2\big\}\leq C(\log\log a_r)^{-1},$$ thereby completing the proof.
\end{proof}

\subsection{Two estimates for a Gaussian random walk}

We shall use the following estimates to control the probability that a standard random walk with Gaussian increments stays below a barrier and, additionally, approaches that barrier at a sufficiently late time. Estimates of this type are classical. In fact, \eqref{eq:walkSurvival} holds in the stronger form $\mmp\{\max_{0\leq k\leq n}S_k\leq h\}\leq C(h+1)n^{-1/2}$, see, for instance, Lemma~3.3 in \cite{Pemantle+Peres:1995} and Probability Result~1 in \cite{Harper:2020}, and related ballot-type estimates with a prescribed endpoint can be found, for instance, in Probability Result~2 of \cite{Harper:2020}. Since Lemma~\ref{lem:strip} in the present form suffices for our purposes and we have not found \eqref{eq:walkStrip} in the literature in the form needed, we give a short self-contained proof based on the reflection principle. The terms involving $d_L$ and $L^{-19}$ are harmless in our applications.
\begin{lemma}\label{lem:strip}
Let $G_1,G_2,\ldots$ be independent random variables with the normal distribution of mean zero and variance $1/2$ and put
$S_0:=0$ and $S_k:=G_1+\cdots+G_k$ for $k\in\mn$. For integer $L\geq3$, put $d_L:=(20\log L)^{1/2}$. There exist absolute positive constants $C$ and $C'$ such that, for
integers $1\leq n\leq L$ and $h\geq0$,
\begin{equation}\label{eq:walkSurvival}
\mmp\{\max_{0\leq k\leq n}S_k\leq h\}\leq C\Big(\frac{h+d_L}{n^{1/2}}+L^{-19}\Big).
\end{equation}
Furthermore, for integers $1\leq m<n\leq L$ and $h,a\geq 0$,
\begin{equation}\label{eq:walkStrip}
\mmp\{\max_{0\leq k\leq n}S_k\leq h,~\max_{m\leq k\leq n}S_k\geq-a\}\leq C'\Big(\frac{(h+d_L)(h+a+d_L)}{(m(n-m))^{1/2}}+L^{-19}\Big).
\end{equation}
\end{lemma}
\begin{proof}
Since $(S_k)_{k\in\mn_0}$ has the same distribution as $(B(k))_{k\in\mn_0}$, where $B=(B(s))_{s\geq0}$ is a Brownian motion with $\me[B(s)^2]=s/2$, we can and do work with $(B(k))_{k\in\mn_0}$. By the reflection principle (see, for instance, Example 8.4.1 on p.~372 in~\cite{Durrett:2010}),
\begin{equation}\label{eq:BrownianMaximum}
\mmp\{\sup_{0\leq s\leq u}B(s)\leq z\}=2\Phi\big(z(2/u)^{1/2}\big)-1, \quad u>0,~ z\geq0,
\end{equation}
where $\Phi$ denotes the standard normal distribution function.

For each $k\in\mn_0$, another application of the reflection principle gives
\begin{multline*}
\mmp\{\sup_{0\leq s\leq1}|B(k+s)-B(k)|>d_L\}=\mmp\{\sup_{0\leq s\leq1}|B(s)|>d_L\}\leq 4\big(1-\Phi\big((2)^{1/2}d_L\big)\big)\\=4\pi^{-1/2}\int_{d_L}^\infty \eee^{-y^2}{\rm d}y\leq 2(\pi^{1/2}d_L)^{-1} \eee^{-d_L^2} \leq (5\pi)^{-1/2}\eee^{-d_L^2}
\end{multline*}
having utilized $d_L\geq 20^{1/2}$. This ensures that the event $$\mathcal Q_{n,L}:=\{\max_{0\leq k<n}\sup_{0\leq s\leq 1}|B(k+s)-B(k)|\leq d_L\}$$ satisfies $\mmp(\mathcal Q_{n,L}^c)\leq (5\pi)^{-1/2}n\eee^{-d_L^2}\leq (5\pi)^{-1/2}L^{-19}$. For each $s\in[0,n]$, choosing $k\in\{0,\ldots,n-1\}$ such that $s\in[k,k+1]$ we obtain $$B(s)\leq B(k)+|B(s)-B(k)|\leq \max_{0\leq j\leq n}B(j)+\max_{0\leq j<n}\sup_{0\leq u\leq1}|B(j+u)-B(j)|.$$ In particular, the last inequality holds with $\sup_{0\leq s\leq n}B(s)$ replacing $B(s)$ on the left-hand side. Hence, on $\mathcal Q_{n,L}$, the inequality $\max_{0\leq k\leq n}B(k)\leq h$ implies $\sup_{0\leq s\leq n}B(s)\leq h+d_L$. Thus, using \eqref{eq:BrownianMaximum} and the trivial estimate $$2\Phi(u)-1=\frac{2}{(2\pi)^{1/2}}\int_0^u\eee^{-v^2/2}\,\dd v \leq (2/\pi)^{1/2}\,u,\quad u\geq 0,$$ we obtain $$\mmp\{\max_{0\leq k\leq n}B(k)\leq h\}\leq \mmp\{\sup_{0\leq s\leq n}B(s)\leq h+d_L\}+\mmp(\mathcal Q_{n,L}^{\,c})\leq \frac{2(h+d_L)}{(\pi n)^{1/2}}+(5\pi)^{-1/2}L^{-19}.$$ This proves \eqref{eq:walkSurvival} with $C=2\pi^{-1/2}$.

To prove \eqref{eq:walkStrip}, observe that, on $\mathcal Q_{n,L}$, the event on the left-hand side of \eqref{eq:walkStrip}, with $B(k)$ replacing $S_k$, entails $$\sup_{0\leq s\leq m}B(s)\leq h+d_L,~-a\leq\sup_{m\leq s\leq n}B(s)\leq h+d_L.$$ Conditionally on the Brownian path up to time $m$, the latter maximum has the same distribution as $B(m)+\sup_{0\leq s\leq n-m}B^\ast(s)$, where $B^\ast$ is an independent copy of $B$. It follows that
\begin{multline*}
\mmp\{\sup_{0\leq s\leq m}B(s)\leq h+d_L,~ -a\leq\sup_{m\leq s\leq n}B(s)\leq h+d_L\}\\=\me\left[\begin{aligned}
&\mmp\{-a-B(m)\leq \sup_{0\leq s\leq n-m} B^\ast(s)\leq h+d_L-B(m)\mid(B(s))_{s\in [0,m]}\}\\
&\qquad\times\1_{\{\sup_{0\leq s\leq m}B(s)\leq h+d_L,~B(m)\leq h+d_L\}}
\end{aligned}\right]\\\leq \frac{2}{\pi^{1/2}}\frac{h+a+d_L}{(n-m)^{1/2}}\mmp\{\sup_{0\leq s\leq m}B(s)\leq h+d_L\}\leq \frac{4}{\pi}\frac{(h+d_L)(h+a+d_L)}{(m(n-m))^{1/2}}.
\end{multline*}
Here, we have used that, according to \eqref{eq:BrownianMaximum}, the random variable $\sup_{0\leq s\leq n-m}B^\ast(s)$ has density $$z\mapsto \frac{2}{(\pi(n-m))^{1/2}}\exp\left(-\frac{z^2}{n-m}\right)\1_{(0,\infty)}(z), \quad z\in\mr,$$ which is bounded by $2(\pi(n-m))^{-1/2}$. Consequently, conditionally on the Brownian path up to time $m$, the probability that this random variable belongs to the interval $[-a-B(m),h+d_L-B(m)]$, whose length is $h+a+d_L$, does not exceed $2(h+a+d_L)(\pi(n-m))^{-1/2}$. This inequality remains valid if the interval intersects, or is entirely contained in, the negative half-line, since the density vanishes there. Now \eqref{eq:walkStrip}, with $C'=4/\pi$, follows from
\begin{multline*}
\mmp\{\max_{0\leq k\leq n}B(k)\leq h,~\max_{m\leq k\leq n}B(k)\geq-a\}\\\leq \mmp\{\sup_{0\leq s\leq m}B(s)\leq h+d_L,~ -a\leq\sup_{m\leq s\leq n}B(s)\leq h+d_L\}+\mmp(\mathcal Q_{n,L}^{\,c}).
\end{multline*}
\end{proof}

\subsection{A change of measure and comparison of barrier probabilities}\label{subsec:tilt}

In this section we compare, under a tilted measure, the probabilities that partial sums of logarithmic Euler-product blocks satisfy prescribed barriers with the corresponding probabilities for the Gaussian random walk in Lemma~\ref{lem:strip}. The change of measure \eqref{eq:tilt} below is the one used by Harper in Sections~3.2 and~3.3 of \cite{Harper:2020}, where it is combined, as here, with the Berry--Esseen theorem and Gaussian barrier estimates.
Under $\mmp_{r,t}$, the increments $X_{r,j}(t)$ defined below have mean close to $1$, which compensates the growth of the barriers, so that the walk $Z_{r,k}(t)$ is centred.

For fixed $r\in\mn$ and $t\in\mr$, introduce the probability measure
$\mmp_{r,t}$ by
\begin{equation}\label{eq:tilt}
\frac{{\rm d}\mmp_{r,t}}{{\rm d}\mmp}
=\frac{|F_{a_r}(t)|^2}{\me[|F_{a_r}(t)|^2]}
\end{equation}
and denote the corresponding expectation and variance by $\me_{r,t}$ and $\operatorname{Var}_{r,t}$, respectively. The density in \eqref{eq:tilt} is a product of factors
depending on separate prime signs. Hence, the variables $(f(p))_{p\leq a_r}$ remain independent under $\mmp_{r,t}$. Furthermore, noting that
\begin{equation*}
|1+f(p)p^{-1/2-\ii t}|^2=1+p^{-1}+2f(p)p^{-1/2}\cos (t\log p),
\end{equation*}
we infer $$\me_{r,t}[f(p)]=\frac{\me[f(p)|1+f(p)p^{-1/2-\ii t}|^2]}{\me[|1+f(p)p^{-1/2-\ii t}|^2]}=\frac{2p^{-1/2}\cos(t\log p)}{1+p^{-1}},\quad p\leq a_r.$$

For $p\in \mathcal P$ and $t\in\mr$, put $\xi_p(t):=\log|1+f(p)p^{-1/2-\ii t}|$. Since $|f(p)p^{-1/2-\ii t}|=p^{-1/2}<1$, we use an expansion
\begin{multline*}
\log|1+f(p)p^{-1/2-\ii t}|={\rm Re}\log(1+f(p)p^{-1/2-\ii t})={\rm Re} \sum_{k\geq1}\frac{(-1)^{k+1}(f(p)p^{-1/2-\ii t})^k}{k}\\=f(p)p^{-1/2}\cos(t\log p)-(2p)^{-1}\cos(2t\log p)+g_{f,t}(p),
\end{multline*}
where in the first equality $\log$ denotes the principal branch of the logarithm. Here, $g_{f,t}(p):={\rm Re}\sum_{k\geq 3}\frac{(-1)^{k+1}(f(p)p^{-1/2-\ii t})^k}{k}$, and using $1/k\leq 1/3$ and $p\geq 2$, we infer
$$|g_{f,t}(p)| \leq\sum_{k\geq3}\frac{p^{-k/2}}{k}\leq\frac{1}{3}\sum_{k\geq3} p^{-k/2}=\frac{p^{-3/2}}{3(1-p^{-1/2})}\leq \frac{p^{-3/2}}{3(1-2^{-1/2})}=:Cp^{-3/2}.$$ For $p\leq a_r$, using the preceding formula for $\me_{r,t}[f(p)]$ and the equality $2\cos^2\theta-\cos(2\theta)/2=1+\cos(2\theta)/2$, we obtain
\begin{equation}\label{eq:tiltedMoments}
\me_{r,t}[\xi_p(t)]=\frac{1}{p}+\frac{1}{2p}\cos(2t\log p)+O(p^{-3/2}).
\end{equation}
Observe that $(\log|1+f(p)p^{-1/2-\ii t}|)^2=(2p)^{-1}(1+\cos (2t\log p))+O(p^{-3/2})$. This, together with \eqref{eq:tiltedMoments}, shows that
\begin{equation}\label{eq:tiltedMoments2}
{\rm Var}_{r,t}[\xi_p(t)]=(2p)^{-1}\big(1+\cos(2t\log p)\big)+O(p^{-3/2}).
\end{equation}
In the three preceding formulas, all implicit constants are absolute. Also, in view of $$|\xi_p(t)|\leq-\log(1-p^{-1/2})\leq\frac{p^{-1/2}}{1-p^{-1/2}}\leq\frac{p^{-1/2}}{1-2^{-1/2}},$$ we infer
\begin{equation}\label{eq:tiltedMoments3}
\me_{r,t}\big[|\xi_p(t)-\me_{r,t}[\xi_p(t)]|^3
\big]\leq C_1p^{-3/2},
\end{equation}
where $C_1:=8(1-2^{-1/2})^{-3}$.

For all sufficiently large $r$, put $N_r:=\left\lfloor\log\log a_r-(\log\log\log a_r)^2\right\rfloor$ and
\begin{equation}\label{eq:anchor}
X_{r,j}(t):=\log\left|\frac{F_{j,a_r}(t)}{F_{j+1,a_r}(t)}\right|=\sum_{\exp(q_{j+1,a_r})<p\leq \exp(q_{j,a_r})}\xi_p(t),\quad 0\leq j<N_r,~t\in\mr.
\end{equation}
For each fixed $r\in\mn$ and $0\leq j<N_r$, the variable $X_{r,j}(t)$ is defined by primes inside $(\eee^{q_{j+1,a_r}}, \eee^{q_{j,a_r}}]$. For different $j$ these intervals are disjoint, which ensures that the variables $X_{r,j}(t)$ are independent under $\mmp_{r,t}$.

We claim that, uniformly in $(\log\log a_r)^{-2}\leq |t| \leq (\log\log a_r)^2+2$ and $0\leq j<N_r$, $j\in\mathbb{Z}$, for each fixed $A>0$,
\begin{equation}\label{eq:gaussianBlocks}
\me_{r,t}[X_{r,j}(t)]=1+O\big((\log\log a_r)^{-A}\big),\quad {\rm Var}_{r,t}[X_{r,j}(t)]=\frac{1}{2}+O\big((\log\log a_r)^{-A}\big),
\end{equation}
where the implicit constants depend on $A$ and do not depend on $j$, nor $t$, and that, for all sufficiently large $r$,
\begin{equation}\label{eq:gaussianBlocks2}
\sum_{\exp(q_{j+1,a_r})<p\leq\exp(q_{j,a_r})}\me_{r,t}\big[|\xi_p(t)-\me_{r,t}[\xi_p(t)]|^3\big] \leq 3C_1\exp(-q_{N_r,a_r}/2).
\end{equation}

Indeed, by \eqref{eq:mertens}, $$\sum_{\exp(q_{j+1,a_r})<p\leq\exp(q_{j,a_r})}p^{-1}=1+O(q_{N_r,a_r}^{-1})=1+O\big((\log\log a_r)^{-A}\big).$$
The last equality is a consequence of $q_{N_r,a_r}\geq\exp((\log\log\log a_r)^2)$. Furthermore, by Lemma~\ref{lem:osc-prime-sum} and \eqref{eq:oscBound}, for $t\neq 0$,
\begin{multline*}
\left|\sum_{\exp(q_{j+1,a_r})<p\leq\exp(q_{j,a_r})}\frac{\cos(2t\log p)}{p}\right|\leq\frac{3}{2|t|q_{N_r,a_r}}+C'(1+2|t|)\exp\big(-c'(q_{N_r,a_r})^{1/2}\big)\\=O((\log\log a_r)^{-A}).
\end{multline*}
Finally, for all sufficiently large $r$, $$\sum_{p>\exp(q_{N_r,a_r})}p^{-3/2}\leq\sum_{n>\exp(q_{N_r,a_r})}n^{-3/2}\leq 3\exp(-q_{N_r,a_r}/2).$$
These estimates, \eqref{eq:tiltedMoments}, \eqref{eq:tiltedMoments2} and \eqref{eq:tiltedMoments3} prove \eqref{eq:gaussianBlocks} and \eqref{eq:gaussianBlocks2}.

The Berry-Esseen theorem for sums of independent, not necessarily identically distributed random variables (see, for instance,  Theorem 3 on p.~111 in \cite{Petrov}) ensures
that
\begin{multline}\label{eq:block-normal}
\sup_{u\in\mr}\left|\mmp_{r,t}\{X_{r,j}(t)-\me_{r,t}[X_{r,j}(t)]\leq u\}-\Phi\big(({\rm Var}_{r,t}[X_{r,j}(t)])^{-1/2}\,u\big)\right|\\=O\Big(\frac{\sum_{\exp(q_{j+1,a_r})<p\leq\exp(q_{j,a_r})}\me_{r,t}\big[|\xi_p(t)-\me_{r,t}[\xi_p(t)]|^3\big]}{({\rm Var}_{r,t}[X_{r,j}(t)])^{3/2}}\Big)=O((\log\log a_r)^{-A}),\quad r\to\infty
\end{multline}
for each fixed $A>0$, uniformly in $(\log\log a_r)^{-2}\leq |t| \leq (\log\log a_r)^2+2$ and $0\leq j<N_r$, $j\in\mathbb{Z}$. Here, the last equality follows from the second part of \eqref{eq:gaussianBlocks} and \eqref{eq:gaussianBlocks2} and, as before, $\Phi$ denotes the standard normal distribution function. Now we prove that, as $r\to\infty$,
\begin{equation}\label{eq:reduction}
\sup_{u\in\mr}\left|\Phi\big(({\rm Var}_{r,t}[X_{r,j}(t)])^{-1/2}(u+1-\me_{r,t}[X_{r,j}(t)])\big)-\Phi(2^{1/2}\, u)\right|=O\big((\log\log a_r)^{-A}\big).
\end{equation}
Indeed, for all sufficiently large $r$, both $({\rm Var}_{r,t}[X_{r,j}(t)])^{-1/2}$ and $2^{1/2}$ are contained in the interval $[1,2]$. Hence, for each $u\in\mr$, by the mean value theorem for differentiable functions,  \begin{multline*}
\left|\Phi\big(({\rm Var}_{r,t}[X_{r,j}(t)])^{-1/2}\,u\big)-\Phi\big(2^{1/2}\,u\big)\right|\leq \big|({\rm Var}_{r,t}[X_{r,j}(t)])^{-1/2}-2^{1/2}\big|\sup_{v\in[1,2]}|u|\Phi'(vu)\\\leq |({\rm Var}_{r,t}[X_{r,j}(t)])^{-1/2}-2^{1/2}| \sup_{z\in\mr}|z|\Phi'(z)=\frac{|({\rm Var}_{r,t}[X_{r,j}(t)])^{-1/2}-2^{1/2}|}{(2\pi \eee)^{1/2}}\\=O\big((\log\log a_r)^{-A}\big)
\end{multline*}
as $r\to\infty$. Another application of the mean value theorem yields
\begin{multline*}
\sup_{u\in\mr}\left|\Phi\bigl(({\rm Var}_{r,t}[X_{r,j}(t)])^{-1/2}(u+1-\me_{r,t}[X_{r,j}(t)])\bigr)-\Phi\big(({\rm Var}_{r,t}[X_{r,j}(t)])^{-1/2}\,u\big)\right|\\\leq \frac{|\me_{r,t}[X_{r,j}(t)]-1|}{({\rm Var}_{r,t}[X_{r,j}(t)])^{1/2}(2\pi)^{1/2}}=O\big((\log\log a_r)^{-A}\big),\quad r\to\infty.
\end{multline*}
The $O$-estimates in both formulas are secured by \eqref{eq:gaussianBlocks}. Combining the last two displayed formulas proves \eqref{eq:reduction}.
For all sufficiently large $r$ and each $t$ satisfying $(\log\log a_r)^{-2}\leq |t|\leq(\log\log a_r)^2+2$, let $H_j$ denote the distribution function of $X_{r,j}(t)-1$ under $\mmp_{r,t}$ and put $\Psi(u):=\Phi\big(2^{1/2}\,u\big)$ for $u\in\mr$. Then \eqref{eq:block-normal} with $A=100$, together with \eqref{eq:reduction}, entails
\begin{equation}\label{eq:coupl}
\max_{0\leq j<N_r}\sup_{u\in\mr}|H_j(u)-\Psi(u)|\leq C_2(\log\log a_r)^{-100}=:\varepsilon_r
\end{equation}
for a universal constant $C_2>0$ and all sufficiently large $r$.

We now derive a direct comparison of barrier probabilities from \eqref{eq:coupl}. For fixed $t$ in the stated range, put $$Y_{r,i}(t):=X_{r,N_r-i}(t)-1,\quad 1\leq i\leq N_r,\quad Z_{r,0}(t):=0,\quad Z_{r,k}(t):=\sum_{i=1}^kY_{r,i}(t),\quad 1\leq k\leq N_r.$$ Let $G_1,\ldots,G_{N_r}$ be independent normal random variables of mean zero and variance $1/2$. We write $\mmp_G$ for their probability measure and put $S_0:=0$ and $S_k:=\sum_{i=1}^kG_i$ for $1\leq k\leq N_r$.

For integers $1\leq m<N_r$ and $h,a\geq0$, define
\begin{align*}
D_h&:=\{z\in\mr^{N_r}:~\textstyle\sum_{i=1}^kz_i\leq h,~1\leq k\leq N_r\}~~\text{and}\\
D_{h,a,m}&:=D_h\cap\{z\in\mr^{N_r}:~\textstyle\sum_{i=1}^kz_i\leq-a,~m\leq k\leq N_r\}.
\end{align*}
For fixed values of all coordinates except one, the corresponding
section of either $D_h$ or $D_{h,a,m}$ is empty or a half-line
of the form $(-\infty,u]$. We shall write $D$ for either $D_h$ or $D_{h,a,m}$.
We claim that
\begin{equation}\label{eq:barrier-comparison}
\left|\mmp_{r,t}\{(Y_{r,i}(t))_{1\leq i\leq N_r}\in D\}
-\mmp_G\{(G_i)_{1\leq i\leq N_r}\in D\}\right|\leq N_r\varepsilon_r.
\end{equation}
Indeed, extend the probability space by adjoining $G_1,\ldots,G_{N_r}$, independently of
$(Y_{r,i}(t))_{1\leq i\leq N_r}$, and denote the resulting probability measure by $\mmp_{r,t}^\ast$. Put $\boldsymbol Z^{(0)}:=(Y_{r,1}(t),\ldots,Y_{r,N_r}(t))$, $\boldsymbol Z^{(N_r)}:=(G_1,\ldots, G_{N_r})$ and
$
\boldsymbol Z^{(i)}
:=(G_1,\ldots,G_i,Y_{r,i+1}(t),\ldots,Y_{r,N_r}(t))$ for $1\leq i<N_r$. By the triangle inequality,
\begin{multline*}
\left|\mmp_{r,t}\{(Y_{r,i}(t))_{1\leq i\leq N_r}\in D\}
-\mmp_G\{(G_i)_{1\leq i\leq N_r}\in D\}\right|\\
\leq\sum_{i=1}^{N_r}
\left|\mmp_{r,t}^\ast\{\boldsymbol Z^{(i-1)}\in D\}
-\mmp_{r,t}^\ast\{\boldsymbol Z^{(i)}\in D\}\right|
\leq N_r\varepsilon_r.
\end{multline*}
To obtain the last inequality, it suffices to show that each modulus does not exceed $\varepsilon_r$. To check this for the $i$th summand, we condition on the common coordinates
$G_1,\ldots,G_{i-1},Y_{r,i+1}(t),\ldots,Y_{r,N_r}(t)$. If the corresponding section is empty, the two
conditional probabilities are zero. Otherwise, their difference
is $H_{N_r-i}(u)-\Psi(u)$ for the endpoint $u$ of that section.
By independence and \eqref{eq:coupl}, the absolute value of that difference  does not exceed $\varepsilon_r$. Taking expectations proves the
claimed bound for the $i$th summand, hence also for the whole sum.

In particular, \eqref{eq:barrier-comparison} ensures that, for each $h\geq 0$,
\begin{equation}\label{eq:survival-comparison}
\left|\mmp_{r,t}\{\max_{0\leq k\leq N_r}Z_{r,k}(t)\leq h\}
-\mmp_G\{\max_{0\leq k\leq N_r}S_k\leq h\}\right|\leq N_r\varepsilon_r.
\end{equation}
Since
$$D_h\setminus D_{h,a,m}=\{z:~\max_{1\leq k\leq N_r}\textstyle\sum_{i=1}^kz_i\leq h,~\max_{m\leq k\leq N_r}\textstyle\sum_{i=1}^kz_i>-a\},$$
applying \eqref{eq:barrier-comparison} to $D_h$ and then to $D_{h,a,m}$ gives
\begin{multline}\label{eq:strip-comparison}
\left|\mmp_{r,t}\{\max_{0\leq k\leq N_r}Z_{r,k}(t)\leq h,~\max_{m\leq k\leq N_r}Z_{r,k}(t)>-a\}\right.\\
\left.-\mmp_G\{\max_{0\leq k\leq N_r}S_k\leq h,~\max_{m\leq k\leq N_r}S_k>-a\}\right|\leq2N_r\varepsilon_r.
\end{multline}
Formulas \eqref{eq:barrier-comparison}, \eqref{eq:survival-comparison} and \eqref{eq:strip-comparison} hold uniformly in $(\log\log a_r)^{-2}\leq|t|\leq(\log\log a_r)^2+2$ as well as in the indicated barriers and indices. Since $N_r\leq\log\log a_r$, we have
$$N_r\varepsilon_r\leq C_2(\log\log a_r)^{-99}=o((\log\log a_r)^{-19}),\quad r\to\infty.$$

\subsection{From weak to strong barriers}

Lemma \ref{lem:weightedBarrier} bounds the second moment of the Euler product under the weak barriers and the contribution from realizations satisfying the weak barriers but violating at least one strong barrier. Formula \eqref{eq:badWeighted} is a variant of Multiplicative Chaos Result~2 in Section~2.4 of \cite{Harper:2023}, which, as stated, does not cover the barriers and the bound considered here. In view of this, we provide a full discussion.
\begin{lemma}\label{lem:weightedBarrier}
For fixed $B\geq 1$ and $b\in(0,1)$, uniformly in $(\log\log a_r)^{-2}\leq |t| \leq (\log\log a_r)^2+2$,
\begin{equation}\label{eq:weakWeighted}
\me\left[|F_{a_r}(t)|^2\1_{\mathcal D_r(t)}\right]=O\left(\log a_r\,\frac{(\log\log\log a_r)^2}{(\log\log a_r)^{1/2}}\right),\quad r\to\infty
\end{equation}
and
\begin{equation}\label{eq:badWeighted}
\me\left[|F_{a_r}(t)|^2\1_{\mathcal D_r(t)\setminus\mathcal A_r(t)}\right]=O\left(\log a_r\,\frac{(\log\log\log a_r)^4}{\log\log a_r}\right), \quad r\to\infty.
\end{equation}
The implicit constant in \eqref{eq:weakWeighted} is absolute, whereas it only depends on $B$ and $b$ in \eqref{eq:badWeighted}.
\end{lemma}
\begin{proof}
Fix $t$ in the stated range. By \eqref{eq:tilt}, for each event $E$,
\begin{equation}\label{eq:tilt-identity}
\me\big[|F_{a_r}(t)|^2\1_E\big]=\me\big[|F_{a_r}(t)|^2\big]\mmp_{r,t}(E).
\end{equation}
According to \eqref{eq:EulerSecond}, $\me[|F_{a_r}(t)|^2]\sim 6\pi^{-2}\eee^\gamma \log a_r$ as $r\to\infty$ uniformly in $t\in\mr$. Hence, it suffices to estimate $\mmp_{r,t}(\mathcal D_r(t))$ and $\mmp_{r,t}(\mathcal D_r(t)\setminus\mathcal A_r(t))$.

First we discard the event $\mathcal L_r(t):=\big\{|F_{N_r,a_r}(t)|<(\log\log a_r)^{-1}\big\}$. By the definition of this event, independence of the prime signs and $\me\big[\big|1+f(p)p^{-1/2-\ii t}\big|^2\big]=1+p^{-1}$,
\begin{multline}\label{eq:small-anchor-final}
\me\big[|F_{a_r}(t)|^2\1_{\mathcal L_r(t)}\big] \leq (\log\log a_r)^{-2}\me\left[\left|\frac{F_{a_r}(t)}{F_{N_r,a_r}(t)}\right|^2\right]=(\log\log a_r)^{-2}\prod_{\exp(q_{N_r,a_r})<p\leq a_r}(1+p^{-1})\\~\sim~
\frac{\log a_r}{(\log\log a_r)^2q_{N_r,a_r}}=\frac{\log a_r}{(\log\log a_r)^2\eee^{(\log\log\log a_r)^2+O(1)}}\\=o\left(\log a_r\,(\log\log a_r)^{-A}\right),\quad r\to\infty
\end{multline}
for each fixed $A>0$. Here, the asymptotic equivalence follows from \eqref{eq:Mert conseq}, the penultimate equality is justified by $\log q_{N_r,a_r}=\log\log a_r-\lfloor \log\log a_r-(\log\log\log a_r)^2\rfloor=(\log\log\log a_r)^2+O(1)$ whereas the last equality follows from the fact that $\eee^{(\log\log\log a_r)^2}$ grows faster than any fixed positive power of $\log\log a_r$.

The partial sums introduced in Section~\ref{subsec:tilt} satisfy $$Z_{r,k}(t)=\sum_{i=1}^k(X_{r,N_r-i}(t)-1)=\log\left|\frac{F_{N_r-k,a_r}(t)}{F_{N_r,a_r}(t)}\right|-k,\quad 0\leq k\leq N_r.$$ Moreover, $q_{N_r-k,a_r}=\eee^kq_{N_r,a_r}$. On $\mathcal D_r(t)\cap (\mathcal L_r(t))^c$, we therefore have, for each $0\leq k\leq N_r$, $$Z_{r,k}(t)\leq\log\big(\eee^kq_{N_r,a_r}(\log\log a_r)^3\big)-k=\log q_{N_r,a_r}+3\log\log\log a_r=:h_r.$$ Consequently, $$\mathcal D_r(t)\cap (\mathcal L_r(t))^c \subseteq \big\{\max_{0\leq k\leq N_r} Z_{r,k}(t)\leq h_r\big\}.$$

By \eqref{eq:survival-comparison}, with $h=h_r$, and \eqref{eq:walkSurvival}, with $L=\lfloor\log\log a_r\rfloor$, $n=N_r$ and $h=h_r$, we obtain
\begin{multline*}
\mmp_{r,t}(\mathcal D_r(t)\cap(\mathcal L_r(t))^c)
\leq\mmp_{r,t}\{\max_{0\leq k\leq N_r}Z_{r,k}(t)\leq h_r\}\\
\leq\mmp_G\{\max_{0\leq k\leq N_r}S_k\leq h_r\}+N_r\varepsilon_r
\leq C\Big(\frac{h_r+d_r}{N_r^{1/2}}+\lfloor\log\log a_r\rfloor^{-19}\Big)+N_r\varepsilon_r\\
=O\Big(\frac{(\log\log\log a_r)^2}{(\log\log a_r)^{1/2}}\Big),\quad r\to\infty,
\end{multline*}
where $d_r:=(20\log\lfloor\log\log a_r\rfloor)^{1/2}$. Here, we have used $N_r\sim\log\log a_r$, $h_r\sim(\log\log\log a_r)^2$ and $N_r\varepsilon_r=O((\log\log a_r)^{-99})$. Together with \eqref{eq:tilt-identity} and \eqref{eq:small-anchor-final}, this proves \eqref{eq:weakWeighted}.

Now we treat \eqref{eq:badWeighted}. On $\mathcal D_r(t)\setminus\mathcal A_r(t)$, there exists $0\leq j\leq\lfloor(1-b)\log\log a_r \rfloor$ such that $$|F_{j,a_r}(t)|>q_{j,a_r}(\log\log a_r)^{-B}.$$
At the same time, $|F_{N_r,a_r}(t)|\leq q_{N_r,a_r}(\log\log a_r)^2$ because $N_r\leq J_\ast (a_r)$. For all sufficiently large $r$, the chosen index $j$ satisfies $j\le \lfloor(1-b)\log\log a_r \rfloor<N_r$, whence $$Z_{r,N_r-j}(t)>\log\left(\frac{q_{j,a_r}(\log\log a_r)^{-B}}{q_{N_r,a_r}(\log\log a_r)^2}\right)-(N_r-j)=-(B+2)\log\log\log a_r.$$ Also, $N_r-j\geq b \log\log a_r-(\log\log\log a_r)^2-1 \geq \lfloor b\log\log a_r/2\rfloor=:m_r$ for all sufficiently large $r$. On $(\mathcal D_r(t)\setminus\mathcal A_r(t))\cap (\mathcal L_r(t))^c$, we thus have $$\max_{0\leq k\leq N_r}Z_{r,k}(t)\leq h_r\quad\text{and}\quad \max_{m_r\leq k\leq N_r}Z_{r,k}(t)>-(B+2)\log\log\log a_r.$$ The validity of the first inequality was justified in the proof of \eqref{eq:weakWeighted}.

Put $a'_r:=(B+2)\log\log\log a_r$. By \eqref{eq:strip-comparison}, with $h=h_r$ and $a=a'_r$,
\begin{multline*}
\mmp_{r,t}((\mathcal D_r(t)\setminus\mathcal A_r(t))\cap (\mathcal L_r(t))^c)
\leq\mmp_{r,t}\{\max_{0\leq k\leq N_r}Z_{r,k}(t)\leq h_r,~\max_{m_r\leq k\leq N_r}Z_{r,k}(t)>-a'_r\}\\
\leq\mmp_G\{\max_{0\leq k\leq N_r}S_k\leq h_r,~\max_{m_r\leq k\leq N_r}S_k>-a'_r\}+2N_r\varepsilon_r\\
\leq\mmp_G\{\max_{0\leq k\leq N_r}S_k\leq h_r,~\max_{m_r\leq k\leq N_r}S_k\geq-a'_r\}+2N_r\varepsilon_r.
\end{multline*}
Applying \eqref{eq:walkStrip} with $L=\lfloor\log\log a_r\rfloor$, $n=N_r$, $m=m_r$, $h=h_r$ and $a=a'_r$, we obtain
\begin{multline*}
\mmp_{r,t}((\mathcal D_r(t)\setminus\mathcal A_r(t))\cap (\mathcal L_r(t))^c)\\
\leq C'\Big(\frac{(h_r+d_r)(h_r+(B+2)\log\log\log a_r+d_r)}{(m_r(N_r-m_r))^{1/2}}+\lfloor\log\log a_r\rfloor^{-19}\Big)+2N_r\varepsilon_r\\
=O\Big(\frac{(\log\log\log a_r)^4}{\log\log a_r}\Big),\quad r\to\infty.
\end{multline*}
Here, we have used $m_r\sim b\log\log a_r/2$, $N_r-m_r\sim(1-b/2)\log\log a_r$, $h_r=O((\log\log\log a_r)^2)$ and $N_r\varepsilon_r=O((\log\log a_r)^{-99})$. Finally, by \eqref{eq:tilt-identity} and \eqref{eq:small-anchor-final},
\begin{multline*}
\me[|F_{a_r}(t)|^2\1_{\mathcal D_r(t)\setminus\mathcal A_r(t)}]
\leq\me[|F_{a_r}(t)|^2\1_{\mathcal L_r(t)}]\\
+\me[|F_{a_r}(t)|^2]\mmp_{r,t}((\mathcal D_r(t)\setminus\mathcal A_r(t))\cap (\mathcal L_r(t))^c)
=O\Big(\log a_r\,\frac{(\log\log\log a_r)^4}{\log\log a_r}\Big),
\end{multline*}
as $r\to\infty$, uniformly in the stated range of $t$. This proves \eqref{eq:badWeighted}.
\end{proof}

\subsection{Restricting the covariance integral}\label{sec:truncate}

In this subsection we show that, uniformly over $x,y\in\mathcal X_r$, the conditional covariance can be approximated by an integral over pairs of nearby arguments at which the strong barriers are satisfied. This follows the argument in Sections~3.2--3.3 of \cite{Harper:2023}, see, in particular, formula~(3.4) and the subsequent error estimate. Since the cutoffs, barriers and error bound in~\cite{Harper:2023} differ from those used here, we provide full details.

Recalling that $\omega_r=(\log\log a_r)^4/T_r$, put $$\Delta_r:=\left\{(v,t)\in\mr^2:~(\log\log a_r)^{-2}\leq |v|,|t|\leq (\log\log a_r)^2,~|v-t|\leq \omega_r\right\}$$ and, for $x,y\in\mathcal X_r$, define
\begin{equation}\label{eq:I}
I_r(x,y):=\frac{1}{4\pi^2}\iint_{\Delta_r}\frac{\1_{\mathcal A_r(v)}F_{a_r}(v)x^{\ii v}\1_{\mathcal A_r(t)}\overline{F_{a_r}(t)}y^{-\ii t}}{(1/2+\ii v)(1/2-\ii t)}K_r(v-t)\,\dd v\,\dd t.
\end{equation}
The dependence on the fixed parameters $B$ and $b$ appearing in the definition of $\mathcal A_r(t)$ is suppressed in this notation.
\begin{lemma}\label{lem:truncation}
For each fixed $B\geq1$ and $b\in(0,1)$,
\begin{equation}\label{eq:truncation}
\lim_{r\to\infty}\mmp\left\{\max_{x,y\in\mathcal X_r}\left|\frac{C_r(x,y)}{(xy)^{1/2}}-I_r(x,y)\right|\leq (\log\log a_r)^{-5/8}\right\}=1.
\end{equation}
\end{lemma}
\begin{proof}
Put $g(t):=(1+|t|)^{-1}$ for $t\in\mr$. Since $$(1/4+t^2)^{-1/2}\leq 5^{1/2}g(t),\quad t\in\mr$$ and $|x^{\ii v}y^{-\ii t}|=1$, all estimates that follow hold simultaneously for $x,y\in\mathcal X_r$.

We first remove the part of the integral in \eqref{eq:perron} in which $\max(|v|,|t|)>(\log\log a_r)^2$. Recalling that $H_r=T_r^8$, denote its absolute majorant by $$\mathcal R_r^{(1)}:=\iint_{\substack{|v|,|t|\leq H_r\\ \max(|v|,|t|)>(\log\log a_r)^2}}|F_{a_r}(v)F_{a_r}(t)|g(v)g(t)|K_r(v-t)|\,\dd v\,\dd t.$$
By the Cauchy-Schwarz inequality and \eqref{eq:EulerSecond},
\begin{equation}\label{eq:two-products-trivial}
\me\big[|F_{a_r}(v)F_{a_r}(t)|\big]\leq \left(\me[|F_{a_r}(v)|^2]\me[|F_{a_r}(t)|^2]\right)^{1/2}=O(\log a_r),\quad r\to\infty
\end{equation}
for all $v,t\in\mr$. By \eqref{eq:two-products-trivial} and \eqref{eq:kernelBounds}, $$\me\big[\mathcal R_r^{(1)}\big]=O(\log a_r)\iint_{\substack{|v|,|t|\leq H_r\\ \max(|v|,|t|)>(\log\log a_r)^2}}g(v)g(t)\min\left(1,\frac1{T_r|v-t|}\right)\,\dd v\,\dd t.$$ We further split the region of integration in $\mathcal R_r^{(1)}$ into subregions.

\noindent {\sc Subregion where $|v|/2\leq|t|\leq 2|v|$}. By symmetry in $|v|$ and $|t|$, it suffices to consider the region $(\log\log a_r)^2<|v|\leq H_r$ and $|v|/2\leq|t|\leq|v|$. On it, $$g(v)g(t)\leq \frac{2}{(1+|v|)^2}.$$ Therefore, enlarging the range of integration in $t$, the expected contribution from the subregion that we treat now is at most $$O(\log a_r)\int_{(\log\log a_r)^2<|v|\leq H_r}\frac{1}{(1+|v|)^2}\left[\int_{-H_r}^{H_r}\min\left(1,\frac1{T_r|v-t|}\right)\,\dd t\right]\dd v.$$ For each $|v|\leq H_r$, the inner integral is bounded by $$2\int_0^{2H_r}\min\left(1,\frac1{T_ru}\right)\,\dd u=\frac2{T_r}\big(1+\log(2H_rT_r)\big).$$ Therefore, using $\log a_r\sim T_r/2$ we conclude that the preceding expression does not exceed
\begin{multline*}
O(1)\big(1+\log(2H_rT_r)\big)\int_{|v|>(\log\log a_r)^2}\frac{\dd v}{(1+|v|)^2}=O\big(\big(1+\log(2H_rT_r)\big)(\log\log a_r)^{-2}\big)\\=O((\log\log a_r)^{-1}),\quad r\to\infty.
\end{multline*}

\noindent {\sc Subregion where $|v|\geq 2|t|$ and $|v|>(\log\log a_r)^2$.} Here, $|v-t|\geq |v|-|t|\geq |v|/2$. Hence the contribution of this subregion to $\me[\mathcal R_r^{(1)}]$ does not exceed
\begin{multline*}
O(1)\int_{(\log\log a_r)^2<|v|\leq H_r}\frac{1}{|v|(1+|v|)}\left(\int_{|t|\leq|v|/2}\frac{\dd t}{1+|t|}\right)\dd v\\=4O(1)\int_{(\log\log a_r)^2}^{H_r}\frac{\log(1+v/2)}{v(1+v)}\,\dd v\\\leq O(1)\int_{(\log\log a_r)^2}^\infty\frac{\log v}{v^2}\,\dd v=O(1)\frac{2\log\log\log a_r+1}{(\log\log a_r)^2}= O\left(\frac{\log\log\log a_r}{(\log\log a_r)^2}\right).
\end{multline*}
Here, the inequality holds whenever $(\log\log a_r)^2\geq 2$, that is, for all sufficiently large $r$. The same estimate holds when $|t|\geq2|v|$ and $|t|>(\log\log a_r)^2$. Consequently, $$\me[\mathcal R_r^{(1)}]=O((\log\log a_r)^{-1})\quad\text{and}\quad \mmp\{\mathcal R_r^{(1)}>(\log\log a_r)^{-3/4}\}=O((\log\log a_r)^{-1/4}),
\quad r\to\infty$$ having utilized Markov's inequality for the second equality.

Next we remove the part for which $\max (|v|,|t|)\leq (\log\log a_r)^2$ and $|v-t|>\omega_r$. Denote its absolute majorant by
$$\mathcal R_r^{(2)}:=\iint_{\substack{|v|,|t|\leq (\log\log a_r)^2\\ |v-t|>\omega_r}}|F_{a_r}(v)F_{a_r}(t)|g(v)g(t)|K_r(v-t)|\,\dd v\,\dd t.$$ Lemma~\ref{lem:twofrequency}, together with
\eqref{eq:kernelBounds}, \eqref{eq:two-products-trivial} and $\log a_r\sim T_r/2$, shows that $\me[\mathcal R_r^{(2)}]$ is bounded by the sum of
\begin{align*}
J_{r,1}&:=O((\log a_r)^{-1/2})\iint_{\substack{|v|,|t|\leq (\log\log a_r)^2\\ |v-t|>\omega_r}}\frac{\min(T_r,|v-t|^{-1})^{1/2}}{|v-t|}\,g(v)g(t)\,\dd v\,\dd t,\\
J_{r,2} &:=O((\log a_r)^{-1/2})\iint_{\substack{|v|,|t|\leq (\log\log a_r)^2\\ |v-t|>\omega_r}} \frac{\min(T_r,|v+t|^{-1})^{1/2}}{|v-t|}\,g(v)g(t)\,\dd v\,\dd t,\\
J_{r,3}&:= O((\log a_r)^{-1/2})\iint_{\substack{|v|,|t|\leq (\log\log a_r)^2\\ |v-t|>\omega_r}}\frac{\log(2+|v|+|t|)}{|v-t|}\,g(v)g(t)\,\dd v\,\dd t.
\end{align*}

For the first term, putting $u=v-t$, we write $$\int_{\mr}g(t+u)g(t)\,\dd t \leq \big(\int_\mr (g(t+u))^2\,\dd t\big)^{1/2}\left(\int_{\mr}(g(t))^2\,\dd t\right)^{1/2}=\int_\mr (g(t))^2\,\dd t=2.$$
Since $T_r\omega_r=(\log\log a_r)^4>1$ for large $r$, we obtain $$J_{r,1}\leq 2O((\log a_r)^{-1/2})\int_{\omega_r<|u|\leq 2(\log\log a_r)^2}|u|^{-3/2}\,\dd u =O((T_r\omega_r)^{-1/2})=O((\log\log a_r)^{-2}).$$
For the second term, putting $u=v-t$, $s=v+t$ and using $g(v)g(t)\leq1$ gives
\begin{multline*}
J_{r,2}\leq O((\log a_r)^{-1/2})\int_{\omega_r<|u|\leq 2(\log\log a_r)^2}\frac{\dd u}{|u|}\int_{|s|\leq 2 (\log\log a_r)^2} \min(T_r,|s|^{-1})^{1/2}\,\dd s\\=O\Big(\frac{(\log\log a_r)^2}{(\log a_r)^{1/2}}\Big).
\end{multline*}
Here, each integral is $O(\log\log a_r)$. For the first, it follows from \begin{multline*}
\int_{\omega_r<|u|\leq 2(\log\log a_r)^2}\frac{\dd u}{|u|}
=2\log \Big(\frac{2(\log\log a_r)^2}{\omega_r}\Big)\\
=2\log T_r-4\log\log\log a_r+2\log 2=O(\log\log a_r),
\end{multline*} where we have used $\log T_r=\log\log a_r+O(1)$. Finally,
\begin{multline*}
J_{r,3}\leq O((\log a_r)^{-1/2})\log\log\log a_r \int_{\omega_r<|u|\leq 2(\log\log a_r)^2}\frac{\dd u}{|u|}\int_{|t|\leq (\log\log a_r)^2}\dd t\\=O\Big(\frac{(\log\log a_r)^3\log\log\log a_r}{(\log a_r)^{1/2}}\Big).
\end{multline*}
Thus, $$\me[\mathcal R_r^{(2)}]=O((\log\log a_r)^{-2})\quad\text{and}\quad \mmp\{\mathcal R_r^{(2)}>(\log\log a_r)^{-1}\}=O((\log\log a_r)^{-1}), \quad r\to\infty.$$ Here, the second equality is justified by another application of Markov's inequality.

We now remove the part where $|v-t|\leq\omega_r$ and $\min(|v|,|t|)<(\log\log a_r)^{-2}$. In this region, $\max(|v|,|t|)\leq (\log\log a_r)^{-2}+\omega_r$. Denoting its absolute majorant by $\mathcal R_r^{(3)}$, invoking \eqref{eq:kernelBounds2} and \eqref{eq:two-products-trivial} and noting that $\omega_r=o\big((\log\log a_r)^{-2}\big)$ we arrive at
\begin{equation*}
\me[\mathcal R_r^{(3)}]\leq O(\log a_r)((\log\log a_r)^{-2}+\omega_r)\int_{-\omega_r}^{\omega_r}|K_r(u)|\,\dd u=O\Big(\frac{\log\log\log a_r}{(\log\log a_r)^2}\Big).
\end{equation*}
Consequently, $\mmp\{\mathcal R_r^{(3)}>(\log\log a_r)^{-1}\}=O((\log\log a_r)^{-1}\log\log\log a_r)$ as $r\to\infty$. For later use, observe that, on the event $\mathcal P_r$ appearing in Lemma \ref{lem:perron},
\begin{multline}\label{eq:inter}
\max_{x,y\in\mathcal X_r}\left|\frac{C_r(x,y)}{(xy)^{1/2}}-\frac{1}{4\pi^2}\iint_{\Delta_r}\frac{F_{a_r}(v)x^{\ii v}\overline{F_{a_r}(t)}y^{-\ii t}}{(1/2+\ii v)(1/2-\ii t)}K_r(v-t)\,\dd v\,\dd t\right|\\ \leq \mathcal R_r^{(1)}+\mathcal R_r^{(2)}+\mathcal R_r^{(3)}+O(T_r^{-1}).
\end{multline}

It remains to impose the strong barriers on $\Delta_r$. Put $E_r:=\{t\in\mr:~(\log\log a_r)^{-2}\leq|t|\leq (\log\log a_r)^2\}$,
\begin{equation*}
\mathcal U_r:=\int_{E_r}\frac{|F_{a_r}(t)|^2\1_{\mathcal D_r(t)}}{1/4+t^2}\,\dd t\quad\text{and}\quad \mathcal V_r:=\int_{E_r}\frac{|F_{a_r}(t)|^2\1_{\mathcal D_r(t)\setminus\mathcal A_r(t)}}{1/4+t^2}\,\dd t.
\end{equation*}
By Fubini's theorem and Lemma~\ref{lem:weightedBarrier}, $$\me[\mathcal U_r]=O\Big(\frac{\log a_r\,(\log\log\log a_r)^2}{(\log\log a_r)^{1/2}}\Big)\quad\text{and}\quad \me[\mathcal V_r]=O\Big(\frac{\log a_r\,(\log\log\log a_r)^4}{\log\log a_r}\Big),\quad r\to\infty.$$ Now Markov's inequality ensures that
\begin{equation}\label{eq:energies}
\lim_{r\to\infty}\mmp\big\{\mathcal U_r\leq\log a_r\,(\log\log a_r)^{-2/5},~\mathcal V_r\leq\log a_r\,(\log\log a_r)^{-9/10}\big\}=1.
\end{equation}

On the event appearing in Lemma~\ref{lem:weak}, the event $\mathcal D_r(t)$ holds for each $t\in E_r$. Thus, on the intersection of $\mathcal P_r$ and the event from Lemma~\ref{lem:weak}, multiplying the integrands in \eqref{eq:I} and \eqref{eq:inter} by $\1_{\mathcal D_r(v)}\1_{\mathcal D_r(t)}$ does not change the values of the integrals. In particular, \eqref{eq:inter} remains valid. For $t,u\in\mr$, put $k_r(u):=|K_r(u)|\1_{[-\omega_r, \omega_r]}(u)$,
\begin{equation*}
A_r(t):=\frac{|F_{a_r}(t)|}{(1/4+t^2)^{1/2}}\1_{E_r}(t)\1_{\mathcal D_r(t)}\quad\text{and}\quad B_r(t):=\frac{|F_{a_r}(t)|}{(1/4+t^2)^{1/2}}\1_{E_r}(t)\1_{\mathcal D_r(t)\setminus\mathcal A_r(t)}.
\end{equation*}
With this notation laid down, on the aforementioned intersection of the events, we obtain
\begin{multline}\label{eq:maximum}
\max_{x,y\in\mathcal X_r}\left|\frac{1}{4\pi^2}\iint_{\Delta_r}\frac{F_{a_r}(v)\1_{\mathcal D_r(v)}x^{\ii v}\overline{F_{a_r}(t)}\1_{\mathcal D_r(t)}y^{-\ii t}}{(1/2+\ii v)(1/2-\ii t)}K_r(v-t)\,\dd v\,\dd t-I_r(x,y)\right|\\ \leq \frac{1}{4\pi^2} \iint_{\mr^2}k_r(v-t)\big(B_r(v)A_r(t)+A_r(v)B_r(t)\big)\,\dd v\,\dd t
\end{multline}
having utilized the inequality $$\1_{\mathcal D_r(v)}\1_{\mathcal D_r(t)}(1-\1_{\mathcal A_r(v)}\1_{\mathcal A_r(t)})\leq \1_{\mathcal D_r(v)\setminus\mathcal A_r(v)}\1_{\mathcal D_r(t)}+\1_{\mathcal D_r(v)}\1_{\mathcal D_r(t)\setminus\mathcal A_r(t)}.$$ Now we show that the latter double integral is bounded from above by $$2\Big(\int_{\mr}k_r(u)\,\dd u\Big) \Big(\int_{\mr}(A_r(t))^2\,\dd t\Big)^{1/2}\Big(\int_{\mr}B_r(t)^2\,\dd t\Big)^{1/2}=2\Big(\int_{-\omega_r}^{\omega_r}|K_r(u)|\,\dd u\Big) (\mathcal U_r\mathcal V_r)^{1/2}.$$ Indeed, changing the variable $u=v-t$ and applying Fubini's theorem gives $$ \iint_{\mr^2}k_r(v-t)B_r(v)A_r(t)\,\dd v\,\dd t= \int_{\mr}k_r(u)\left(\int_{\mr}B_r(t+u)A_r(t)\,\dd t\right)\dd u.$$ For each fixed $u$, the Cauchy–Schwarz inequality and translation invariance of Lebesgue measure yield
\begin{multline*}
\int_{\mr}B_r(t+u)A_r(t)\,\dd t \leq \Big(\int_{\mr}(B_r(t+u))^2\,\dd t\Big)^{1/2}\Big(\int_{\mr}(A_r(t))^2\,\dd t\Big)^{1/2}\\=\Big(\int_{\mr}(B_r(t))^2\,\dd t\Big)^{1/2}\Big(\int_{\mr}(A_r(t))^2\,\dd t\Big)^{1/2}.
\end{multline*}
Hence, $$\iint_{\mr^2}k_r(v-t)B_r(v)A_r(t)\,\dd v\,\dd t\leq \int_{\mr}k_r(u)\,\dd u \Big(\int_{\mr}(A_r(t))^2\,\dd t\Big)^{1/2}\Big(\int_{\mr}B_r(t)^2\,\dd t\Big)^{1/2}.$$ The other term admits the same upper bound.

Using \eqref{eq:kernelBounds2} and \eqref{eq:energies},
we conclude that, with probability tending to one, the maximum in \eqref{eq:maximum} is at most
\begin{multline*}
10C\, T_r^{-1}\log\log T_r\log a_r\,(\log\log a_r)^{-13/20}=O(\log\log\log a_r\,(\log\log a_r)^{-13/20})\\=o\big((\log\log a_r)^{-5/8}\big).
\end{multline*}
Finally, with probability tending to one, the maximum in \eqref{eq:inter} is at most $$(\log\log a_r)^{-3/4}+2(\log\log a_r)^{-1}+O\big(T_r^{-1}\big)=o\big((\log\log a_r)^{-5/8}{\big)}.$$
This proves \eqref{eq:truncation}.
\end{proof}

\subsection{How often is \texorpdfstring{$|C_r(x,y)|$}{the conditional covariance} large?}

Next we estimate the high-order moments of the integrals $I_r(x,y)$ and use these estimates to bound, simultaneously for $x\in\mathcal X_r$, the number of points $y$ for which $|C_r(x,y)|$ is large.

The unit spacing of the logarithms of the points in $\mathcal X_r$ gives rise to a periodic function. We first estimate its integral and then prove an auxiliary inequality for the gaps between ordered arguments. For $u\in\mr$ and $m\in\mn$, put $d_{2\pi}(u):=\inf_{j\in\mathbb Z}|u-2\pi j|$ and $D_m(u):=\min(m, \pi/d_{2\pi}(u))$ with $D_m(u):=m$ when $u\in2\pi\mathbb Z$. We claim that
\begin{equation}\label{eq:geometric}
\Big|\sum_{j=0}^{m-1}\eee^{\ii ju}\Big|\leq D_m(u),\quad m\in\mn,~u\in\mr.
\end{equation}
Indeed, if $u\in2\pi\mathbb Z$, then \eqref{eq:geometric} is trivial. Note that $|\sin(u/2)|\geq d_{2\pi}(u)/\pi$ for $u\in\mr$. With this at hand, if $u\notin 2\pi\mathbb Z$, it follows from $$\Big|\sum_{j=0}^{m-1}\eee^{\ii ju}\Big|=\Big|\frac{\eee^{\ii mu}-1}{\eee^{\ii u}-1}\Big|=\Big|\frac{\sin (mu/2)}{\sin(u/2)}\Big|\leq \frac{\pi}{d_{2\pi}(u)},$$ and also from $\big|\sum_{j=0}^{m-1}\eee^{\ii ju}\big|\leq m$.

For integer $1\leq n\leq2\log\log a_r$, put $$\mathcal K_{r,n}(u):=\sup_{|v-u|\leq n\omega_r}D_{|\mathcal X_r|}(v),
\quad u\in\mr,$$ where $\omega_r=(\log\log a_r)^4/T_r$. This function is $2\pi$-periodic and does not exceed $|\mathcal X_r|=\lfloor T_r^{1-\lambda}\rfloor+1$. Since $\omega_r=o((\log\log a_r)^{-1})$, we conclude that, for all sufficiently large $r$, $n\omega_r<1$. Further, for such $r$,
$$\mathcal K_{r,n}(u)=
\begin{cases}
|\mathcal X_r|,&d_{2\pi}(u)\leq n\omega_r,\\[2mm]
\displaystyle\min(|\mathcal X_r|, \pi/(d_{2\pi}(u)-n\omega_r)),& d_{2\pi}(u)>n\omega_r.
\end{cases}
$$
Using symmetry about $\pi$, we obtain
\begin{multline*}
\int_0^{2\pi}\mathcal K_{r,n}(u)\,\dd u =2\int_0^{n\omega_r} |\mathcal X_r|\,\dd u+2\int_{n\omega_r}^\pi\min (|\mathcal X_r|,\pi/(u-n\omega_r))\,\dd u\\
=2n\omega_r |\mathcal X_r|+2\int_0^{\pi-n\omega_r}
\min(|\mathcal X_r|,\pi/s)\,\dd s\leq 2(n\omega_r|\mathcal X_r|+\pi(1+\log |\mathcal X_r|))
\end{multline*}
having utilized $$\int_0^\pi \min(m,\pi/s)\,\dd s =\int_0^{\pi/m}m\,\dd s
+\pi\int_{\pi/m}^{\pi}s^{-1}\dd s=\pi(1+\log m),$$ which holds for any $m\in\mn$. For later use, we note that
\begin{multline}\label{eq:thickened}
0\leq\mathcal K_{r,n}(u)\leq |\mathcal X_r|\quad\text{and}\\ \sup_{\substack{J\subset\mr\ {\rm interval}\\ |J|\leq2}}
\int_J\mathcal K_{r,n}(u)\,\dd u\leq \int_0^{2\pi}\mathcal K_{r,n}(u)\,{\rm d}u\leq 2(n\omega_r|\mathcal X_r|+\pi(1+\log |\mathcal X_r|))
\end{multline}
for all sufficiently large $r$, uniformly in $1\leq n\leq 2\log\log a_r$.

The following inequality will control the total interaction between different gaps.
\begin{lemma}\label{lem:gap}
For each $n\in\mn$ and all positive numbers $a_1,\ldots,a_n$,
\begin{equation}\label{eq:gap}
\sum_{1\leq j<l\leq n} \frac{\min(a_j,a_l)}{a_j+\ldots+a_{l-1}}
\leq 2n\log n.
\end{equation}
\end{lemma}
\begin{proof}
Fix $1\leq j<n$ and put $b_l:=\min(1,a_l/a_j)$, $A_l:=a_j+\ldots+a_{l-1}$ for $j<l\leq n$ and $B_l:=1+b_{j+1}+\ldots+b_{l-1}$ for $j<l\leq n+1$. Since
$A_l\geq a_jB_l$, $B_{l+1}=B_l+b_l$ and $0<b_l/B_l\leq 1$ (just because $b_l\leq 1$ and $B_l\geq 1$), we conclude that
$$\sum_{l=j+1}^{n}\frac{\min(a_j,a_l)}{A_l} \leq\sum_{l=j+1}^{n}\frac{b_l}{B_l} \leq2\sum_{l=j+1}^{n}\log\frac{B_{l+1}}{B_l}=2\log B_{n+1}
\leq 2\log(n-j+1)$$ having utilized $x\leq 2\log(1+x)$ for $x\in [0,1]$ for the second inequality. Summing over $1\leq j\leq n-1$ we obtain  $$\sum_{1\leq j<l\leq n}
\frac{\min(a_j,a_l)}{a_j+\ldots+a_{l-1}}
\leq 2\log(n!)\leq 2n\log n,$$ thereby completing the proof.
\end{proof}

Next we bound joint moments of the Euler products subject to the strong barriers by a product of deterministic functions of the gaps between their arguments. For $0<x\leq(\log\log a_r)^2$, put $$\psi_r(x):=\min\big(\log a_r,~\max\big(1/x,(\log\log a_r)^2\big)\big)$$ and $$\chi_r(x):=
\begin{cases}
(\log\log a_r)^{-2B}\psi_r(x),   &0<x<(\log a_r)^{-b},\\
(\log\log a_r)^4\psi_r(x),  &(\log a_r)^{-b}\leq x\leq(\log\log a_r)^2,
\end{cases}
$$
where $B\geq 1$ and $b\in (0,1)$ are fixed numbers as appear in \eqref{eq:strongbarrier}. Also, we recall the definition of events $\mathcal A_r(t)$ from \eqref{eq:strongbarrier}.
\begin{lemma}\label{lem:cell}
There exists an absolute constant $C>0$ such that, for all
sufficiently large $r$, integers
$1\leq k\leq\lfloor\log\log a_r\rfloor$ and arguments
$$(\log\log a_r)^{-2}\leq t_1<\cdots<t_{2k}\leq(\log\log a_r)^2,$$ we have
\begin{equation}\label{eq:cellMoment}
\me\Big[\prod_{j=1}^{2k}\big(\1_{\mathcal A_r(t_j)}|F_{a_r}(t_j)|^2\big)\Big]\leq (Ck)^{Ck}(\log a_r)^{2k}\prod_{j=1}^{2k-1}\chi_r(t_{j+1}-t_j).
\end{equation}
The constant $C$ does not depend on $B$ or $b$.
\end{lemma}
\begin{proof}
Put $n:=2k$ and $x_j:=t_{j+1}-t_j$ for $1\leq j<n$.
For these indices, define $$h_j:=\max\big(0,~\lceil \log(x_j\log a_r)\rceil-1\big)\quad\text{and}\quad \widetilde h_j:=\min\big(h_j,~\lfloor\log\log a_r-2\log\log\log a_r\rfloor\big).$$
For all sufficiently large $r$, each $\widetilde h_j$ belongs to $\{0,\ldots,J_\ast(a_r)\}$, where $J_\ast(a_r)=\lfloor\log\log a_r-1\rfloor$. We claim that
\begin{equation}\label{eq:gap-cutoffs}
(\log\log a_r)^2\leq q_{\widetilde h_j,a_r}\leq \eee\psi_r(x_j).
\end{equation}

\noindent {\sc Proof of \eqref{eq:gap-cutoffs}.} We first show that $q_{h_j,a_r}\leq \eee/x_j$. Indeed, if $h_j\geq 1$, then $h_j<\log(x_j\log a_r)\leq h_j+1$, whence $1/x_j<q_{h_j,a_r}\leq \eee/x_j$. If $h_j=0$, then $x_j\log a_r\leq \eee$, so that $q_{0,a_r}=\log a_r \leq \eee/x_j$.

Further, using $$\log\log a_r-2\log\log\log a_r-1<H_r^\ast:=\lfloor \log\log a_r-2\log\log\log a_r\rfloor\leq\log\log a_r-2\log\log\log a_r,$$ we infer $$(\log\log a_r)^2\leq q_{H_r^\ast, a_r}=\exp(-H_r^\ast)\log a_r <\eee (\log\log a_r)^2.$$ Since the function $h\mapsto q_{h,a_r}$ is decreasing, we have $q_{\widetilde h_j,a_r}=\max(q_{h_j,a_r},q_{H_r^\ast,a_r})$.
Consequently, $$(\log\log a_r)^2\leq q_{\widetilde h_j,a_r}\leq \eee\max(1/x_j,(\log\log a_r)^2).$$ This proves the left-hand inequality in \eqref{eq:gap-cutoffs}. To obtain the right-hand inequality in \eqref{eq:gap-cutoffs}, we combine the right-hand inequality in the latter displayed formula with $q_{\widetilde h_j,a_r}\leq \log a_r$ which follows from $\widetilde h_j\geq 0$:
\begin{multline*}
q_{\widetilde h_j,a_r}\leq \min\big(\log a_r,~\eee\max\big(1/x_j,~(\log\log a_r)^2\big)\big)\\\leq \eee\min\big(\log a_r,~\max\big(1/x_j,~(\log\log a_r)^2\big)\big)=\eee\psi_r(x_j).
\end{multline*}
The proof of \eqref{eq:gap-cutoffs} is complete.

Put $$\alpha_j:=
\begin{cases}
(\log\log a_r)^{-B},
   &x_j<(\log a_r)^{-b},\\
(\log\log a_r)^2,
   &x_j\geq(\log a_r)^{-b}.
\end{cases}
$$
If $x_j<(\log a_r)^{-b}$, then $\widetilde h_j\leq h_j\leq \lfloor(1-b)\log\log a_r\rfloor$ for all sufficiently large $r$. With this at hand, irrespective of whether $x_j<(\log a_r)^{-b}$ or not, the definitions of the weak and strong barriers imply that, on $\mathcal A_r(t_j)$, $|F_{\widetilde h_j,a_r}(t_j)|\leq q_{\widetilde h_j,a_r}\alpha_j$. Extracting these factors for $j<n$ and then discarding the
remaining indicators, we obtain
\begin{equation}\label{eq:moment-extraction}
\me\Big[\prod_{j=1}^{n}\big(\1_{\mathcal A_r(t_j)}|F_{a_r}(t_j)|^2\big)\Big]\leq \prod_{j=1}^{n-1}\big(q_{\widetilde h_j, a_r}\alpha_j\big)^2\me\Big[|F_{a_r}(t_n)|^2\prod_{j=1}^{n-1}\Big|\frac{F_{a_r}(t_j)}
{F_{\widetilde h_j,a_r}(t_j)}\Big|^2\Big].
\end{equation}
We shall prove that
\begin{equation}\label{eq:residualMoment}
\me\Big[|F_{a_r}(t_n)|^2\prod_{j=1}^{n-1}\Big|\frac{F_{a_r}(t_j)}{F_{\widetilde h_j,a_r}(t_j)}\Big|^2\Big]\leq (C_\ast n)^{C_\ast n}\log a_r\prod_{j=1}^{n-1}\eee^{\widetilde h_j}
\end{equation}
for an absolute constant $C_\ast>0$.

The ratio with $\widetilde h_j=0$ is equal to one. Each other ratio contains only primes exceeding $\exp(q_{\widetilde h_j,a_r})\geq\exp((\log\log a_r)^2)$, see \eqref{eq:gap-cutoffs}. For $p\leq a_r$, put $$\mathcal I_p:=
\{n\}\cup \big\{j\in\{1,\ldots,n-1\}:~p>\exp(q_{\widetilde h_j,a_r})\big\}.$$ By independence of the prime signs,
\begin{multline*}
\me\Big[|F_{a_r}(t_n)|^2\prod_{j=1}^{n-1}\Big|\frac{F_{a_r}(t_j)}{F_{\widetilde h_j,a_r}(t_j)}\Big|^2
\Big]=\prod_{p\leq a_r}\me\Big[\prod_{j\in\mathcal I_p}\big(1+f(p)p^{-1/2+\ii t_j}\big)\big(1+f(p)p^{-1/2-\ii t_j}\big)\Big]\\=\prod_{p\leq a_r}\me\Big[\prod_{k=1}^{2|\mathcal I_p|}\big(1+f(p)p^{-1/2+\ii s_{p,k}}\big)\Big],
\end{multline*}
where, for each prime $p\leq a_r$, $(s_{p,k})_{1\leq k\leq|\mathcal I_p|}$ is an enumeration of $(t_j)_{j\in\mathcal I_p}$ and $s_{p,|\mathcal I_p|+k}:=-s_{p,k}$ for
$1\leq k\leq|\mathcal I_p|$. Expanding the latter product gives $$\prod_{k=1}^{2|\mathcal I_p|}\big(1+f(p)p^{-1/2+\ii s_{p,k}}\big)=\sum_{\ell=0}^{2|\mathcal I_p|}(f(p)^\ell)p^{-\ell/2}e_\ell(p^{\ii s_{p,1}},\ldots, p^{\ii s_{p,2|\mathcal I_p|}}),$$ where $e_\ell$ is the elementary symmetric polynomial of degree $\ell$ given by $$e_\ell (p^{\ii s_{p,1}},\ldots, p^{\ii s_{p,2|\mathcal I_p|}})=\sum_{1\leq k_1<\cdots<k_\ell \leq 2|\mathcal I_p|}p^{\ii (s_{p, k_1}+\ldots+ s_{p, k_\ell})}, \quad 1\leq \ell\leq 2|\mathcal I_p|$$ and $e_0\equiv 1$. Put $$L_p:=\prod_{j\in\mathcal I_p}\big(1+f(p)p^{-1/2+\ii t_j}\big)\big(1+f(p)p^{-1/2-\ii t_j}\big).$$ Noting that
\begin{multline*}
e_2\big(p^{\ii s_{p,1}},\ldots, p^{\ii s_{p,2|\mathcal I_p|}}{\big)}=|\mathcal I_p|+4\sum_{\substack{j<l\\j,l\in\mathcal I_p}}\cos (t_j\log p)\cos (t_l\log p)\\=|\mathcal I_p|+
2\sum_{\substack{j<l\\j,l\in\mathcal I_p}}(\cos((t_j-t_l)\log p)+\cos((t_j+t_l)\log p))
\end{multline*}
and $|e_\ell\big(p^{\ii s_{p,1}},\ldots, p^{\ii s_{p,2|\mathcal I_p|}}\big)|\leq \binom{2|\mathcal I_p|}{\ell}$ and passing to expectations, we obtain
\begin{equation}\label{eq:inter1}
\me [L_p]=1+|\mathcal I_p|p^{-1}+2p^{-1}\sum_{\substack{j<l\\j,l\in\mathcal I_p}}(\cos((t_j-t_l)\log p)+\cos((t_j+t_l)\log p))+R_p,
\end{equation}
where $|R_p| \leq\sum_{\ell=4}^{2|\mathcal I_p|}\binom{2|\mathcal I_p|}{\ell}p^{-\ell/2}$. Here, as usual, an empty sum is understood to be zero. Since $\binom{2|\mathcal I_p|}{\ell}\leq(2|\mathcal I_p|)^\ell/\ell!$, we infer
\begin{equation}\label{eq:inter2}
|R_p|\leq\sum_{\ell\geq 4} \frac{(2|\mathcal I_p|p^{-1/2})^\ell}{\ell!}\leq \Big(16\sum_{\ell\geq 4} \frac{1}{\ell!}\Big)\big(|\mathcal I_p|p^{-1/2}\big)^4
\end{equation}
whenever $2|\mathcal I_p|p^{-1/2}\leq 1$.
Similarly, $$|\me [L_p]-1|\leq \sum_{\ell\geq 2} \frac{(2|\mathcal I_p|p^{-1/2})^\ell}{\ell!} \leq \Big(4\sum_{\ell\geq 1}\frac{1}{\ell!}\Big)(|\mathcal I_p|p^{-1/2})^2$$ whenever $2|\mathcal I_p|p^{-1/2}\leq 1$. We are now working with primes satisfying $\exp((\log\log a_r)^2)<p\leq a_r$. Hence, $$|\mathcal I_p|p^{-1/2}\leq np^{-1/2}\leq 2\log\log a_r\,\exp(-(\log\log a_r)^2/2)~\to~0,\quad r\to\infty,$$ and we conclude that, uniformly over the primes and arguments under consideration, $|\me[L_p]-1|\leq 1/2$ for all sufficiently large $r$. Using \eqref{eq:inter1} and \eqref{eq:inter2}, in combination with $|\log(1+v)-v|\leq v^2$ for $|v|\leq 1/2$, we arrive at
\begin{multline}\label{eq:localMoment}
\log\me [L_p]=|\mathcal I_p|p^{-1}+2p^{-1}\sum_{\substack{j<l\\j,l\in\mathcal I_p}}(\cos((t_j-t_l)\log p)+\cos((t_j+t_l)\log p))+O(n^4p^{-2}),
\end{multline}
where the implicit constant is absolute.

Now we are ready to prove \eqref{eq:residualMoment}. We first note that $$n^4\sum_{p>\exp((\log\log a_r)^2)}p^{-2}\leq n^4\sum_{k>\exp((\log\log a_r)^2)}k^{-2}=O\big(n^4\exp(-(\log\log a_r)^2)\big)=o(1)$$ because $n\leq 2\lfloor\log\log a_r\rfloor$. Next, we estimate the contribution to the expectation in \eqref{eq:residualMoment} of the primes $p\leq \exp\big((\log\log a_r)^2\big)$ and the terms $|\mathcal I_p|p^{-1}$ for $\exp\big((\log\log a_r)^2\big)<p\leq a_r$. For $p\leq \exp\big((\log\log a_r)^2\big)$, we have $\mathcal I_p=\{n\}$, and the contribution of one prime from the present range is $\me\big[\big|1+f(p)p^{-1/2-\ii t_n}\big|^2\big]=1+1/p$. For $\exp((\log\log a_r)^2)<p\leq a_r$, we have $|\mathcal I_p|=1+\sum_{j=1}^{n-1}\1_{\{p>\exp(q_{\widetilde h_j,a_r})\}}$, whence
$$\sum_{\exp((\log\log a_r)^2)<p\leq a_r}|\mathcal I_p|p^{-1}=\sum_{\exp((\log\log a_r)^2)<p\leq a_r}p^{-1}+\sum_{j=1}^{n-1}\sum_{\exp(q_{\widetilde h_j,a_r})<p\leq a_r}p^{-1}.$$ Therefore, the contribution under discussion is
\begin{multline*}
\prod_{p\leq \exp((\log\log a_r)^2)}(1+p^{-1})\exp\Big(\sum_{\exp((\log\log a_r)^2)<p\leq a_r}|\mathcal I_p|p^{-1}\Big)\\ \leq \exp\Big(\sum_{p\leq a_r}p^{-1}\Big)\prod_{j=1}^{n-1}
\exp\Big(\sum_{\exp(q_{\widetilde h_j,a_r})<p\leq a_r}p^{-1}\Big).
\end{multline*}
We have used $1+1/p\leq \eee^{1/p}$ for the inequality. By \eqref{eq:mertens}, for all sufficiently large $r$, $\exp\big(\sum_{p\leq a_r}p^{-1}\big)\leq C_0\log a_r$ and, uniformly in $1\leq j<n$ and the arguments $t_1,\ldots, t_n$ in the stated regions, $$
\sum_{\exp(q_{\widetilde h_j,a_r})<p\leq a_r}p^{-1}=\log\log a_r-\log q_{\widetilde h_j,a_r}+O(1)=\widetilde h_j+O(1).$$ Hence,
$\exp\big(\sum_{\exp(q_{\widetilde h_j,a_r})<p\leq a_r}p^{-1}\big)\leq C_1\eee^{\widetilde h_j}$. When $\widetilde h_j=0$, the sum is empty and this factor is equal to $1$. Summarizing, the contribution under discussion is bounded from above by $C_2^n \log a_r \prod_{j=1}^{n-1}\eee^{\widetilde h_j}$, where $C_2:=\max(C_0, C_1)$.

It remains to estimate the contribution of the pairs in \eqref{eq:localMoment}. Fix $j<l<n$ such that both $\widetilde h_j$ and $\widetilde h_l$ are positive. If $j,l\in\mathcal I_p$,
then necessarily $\exp(\max(q_{\widetilde h_j,a_r}, q_{\widetilde h_l,a_r}))<p\leq a_r$, and an application of Lemma~\ref{lem:osc-prime-sum} and \eqref{eq:oscBound} yields
\begin{multline*}
\sum_{\exp(\max(q_{\widetilde h_j,a_r}, q_{\widetilde h_l,a_r}))<p\leq a_r}p^{-1}(\cos((t_j-t_l)\log p)+\cos((t_j+t_l)\log p))\\\leq 6\frac{\min(\eee^{\widetilde h_j},\eee^{\widetilde h_l})}{(t_l-t_j)\log a_r}+4C^\prime (\log\log a_r)^2 \eee^{-c'\log\log a_r}.
\end{multline*}
To obtain the last exponent, we have used \eqref{eq:gap-cutoffs} ensuring that $\max(q_{\widetilde h_j,a_r}, q_{\widetilde h_l,a_r})\geq (\log\log a_r)^2$. Since
$\eee^{\widetilde h_j}/\log a_r\leq \eee^{h_j}/\log a_r \leq x_j$, see the paragraph following \eqref{eq:gap-cutoffs}, the first term is at most
$6\frac{\min(x_j,x_l)}{x_j+\cdots+x_{l-1}}$. Finally, by Lemma~\ref{lem:gap} applied to $x_1,\ldots, x_{n-1}$, $\sum_{1\leq j<l< n} \frac{\min(x_j,x_l)}{x_j+\ldots+x_{l-1}}\leq 2(n-1)\log (n-1)\leq 2n\log n$.

For a pair $(j,n)$ with $j<n$ and $\widetilde h_j>0$, we have $j,n\in\mathcal I_p$ if, and only if, $\exp(q_{\widetilde h_j,a_r})<p\leq a_r$ because the index $n$ belongs to $\mathcal I_p$ for each $p\leq a_r$. Applying the same estimates as above, the sum of the corresponding first terms is at most
$$6\sum_{\substack{1\leq j<n\\\widetilde h_j>0}}\frac{\eee^{\widetilde h_j}}{\log a_r\,(t_n-t_j)}\leq 6(n-1).$$ Here, we have used $\eee^{\widetilde h_j}/\log a_r\leq x_j\leq t_n-t_j$. Indices with $\widetilde h_j=0$ contribute no such pairs. Since $n\leq 2\lfloor \log\log a_r\rfloor$, $\sum_{1\leq j<l\leq n}(\log\log a_r)^2 \eee^{-c'\log\log a_r}=(\log\log a_r)^2 \eee^{-c'\log\log a_r}n(n-1)/2=o(1)$ as $r\to\infty$. Exponentiating the resulting $C_3 n\log n$ bound, with $C_3\geq 1$, proves \eqref{eq:residualMoment}, with $C_\ast:=\max(C_2,C_3)$.

Finally, observe that $$\log a_r\prod_{j=1}^{n-1}\big(q_{\widetilde h_j,a_r}^2 \eee^{\widetilde h_j}\alpha_j^2\big)=(\log a_r)^n \prod_{j=1}^{n-1}\big(q_{\widetilde h_j,a_r}\alpha_j^2\big).$$
Combining this equality with \eqref{eq:gap-cutoffs}, \eqref{eq:moment-extraction} and \eqref{eq:residualMoment} and noting that $\alpha_j^2=\chi_r(x_j)/\psi_r(x_j)$ proves \eqref{eq:cellMoment}, with $C:={\eee C_\ast}$.
\end{proof}

To proceed, we record three consequences of the definitions of $\psi_r$ and $\chi_r$. We shall show that there exist absolute constants $c_1,c_2\geq1$ such that, for all sufficiently large $r$,
\begin{equation}\label{eq:gap-integrals}
\int_0^{(\log a_r)^{-b}}\chi_r(x)\,\dd x\leq c_1(\log\log a_r)^{1-2B},\qquad
\int_0^{(\log\log a_r)^2}\chi_r(x)\,\dd x\leq c_2(\log\log a_r)^8,
\end{equation}
and
\begin{equation}\label{eq:large-gap-bound}
\max_{(\log a_r)^{-b}\leq x\leq(\log\log a_r)^2}\chi_r(x)
=(\log a_r)^b(\log\log a_r)^4.
\end{equation}
Indeed, $(\log a_r)^{-1}<(\log a_r)^{-b}<(\log\log a_r)^{-2}$ for all sufficiently large $r$. Hence, $\psi_r(x)=\min(\log a_r, 1/x)$ for $x\in (0,(\log a_r)^{-b})$ and
\begin{multline*}
\int_0^{(\log a_r)^{-b}}\chi_r(x)\,\dd x
=(\log\log a_r)^{-2B}\Big(\int_0^{1/\log a_r}\log a_r\,\dd x+
\int_{1/\log a_r}^{(\log a_r)^{-b}}\frac{\dd x}{x}\Big)\\
=(\log\log a_r)^{-2B}\big(1+(1-b)\log\log a_r\big).
\end{multline*}
On the remaining range, $\psi_r(x)=\max(1/x,(\log\log a_r)^2)$, which proves \eqref{eq:large-gap-bound} and gives
\begin{multline*}
\int_{(\log a_r)^{-b}}^{(\log\log a_r)^2}\chi_r(x)\,\dd x
=(\log\log a_r)^4\Big(\int_{(\log a_r)^{-b}}^{(\log\log a_r)^{-2}}\frac{\dd x}{x}
+\int_{(\log\log a_r)^{-2}}^{(\log\log a_r)^2}(\log\log a_r)^2\,\dd x\Big)\\
=(\log\log a_r)^4\big(b\log\log a_r-2\log\log\log a_r+(\log\log a_r)^4-1\big)~\sim~(\log\log a_r)^8.
\end{multline*}
These formulas prove \eqref{eq:gap-integrals}, with constants independent of $B$ and $b$.

Next, integrating the preceding moment bound, we obtain an estimate which controls all rows of the covariance integral $I_r(x,y)$ simultaneously.
\begin{lemma}\label{lem:highmoment}
There exists an absolute constant $C_0\geq1$, independent of $B$ and $b$, such that, for all sufficiently large $r$ and integers $1\leq k\leq\lfloor\log\log a_r\rfloor$,
\begin{multline}\label{eq:highmoment}
\me\left[\max_{x\in\mathcal X_r}\sum_{y\in\mathcal X_r}|I_r(x,y)|^{2k}\right]\\\leq (C_0k)^{C_0k}(\log\log a_r)^{C_0k}\big((\log a_r)^{2b}+(\log a_r)(\log\log a_r)^{-2Bk}\big).
\end{multline}
\end{lemma}
\begin{proof}
Put $n:=2k$. Throughout the proof, $r$ is sufficiently large and $2\leq n\leq2\lfloor\log\log a_r\rfloor$. We first extend the periodic-kernel estimate to intervals of length at most $(\log\log a_r)^2$. Since, as $r\to\infty$, $$n\omega_r |\mathcal X_r|\leq 2(\log\log a_r)^5T_r^{-1}\big(T_r^{1-\lambda}+1\big)\leq 4(\log\log a_r)^5T_r^{-\lambda}=o(1)$$ and $$\log |\mathcal X_r|=O(\log T_r)=O(\log\log a_r),$$ formula \eqref{eq:thickened} and periodicity yield
\begin{equation}\label{eq:kernel-for-moments}
\sup_{\substack{J\subset\mr\ {\rm interval}\\|J|\leq(\log\log a_r)^2}}
\int_J\mathcal K_{r,n}(u)\,\dd u
\leq\Big(1+\frac{(\log\log a_r)^2}{2\pi}\Big)\int_0^{2\pi}\mathcal K_{r,n}(u)\,\dd u \leq C_1(\log\log a_r)^3.
\end{equation}
Also,
\begin{equation}\label{eq:kernel-pointwise-moments}
0\leq\mathcal K_{r,n}(u)\leq|\mathcal X_r|\leq T_r+1\leq3\log a_r,\quad u\in\mr.
\end{equation}
Here, we have used $\log a_r=(T_r+T_r^{1-\lambda})/2$.

Define $$\Delta_r^+:=\left\{(v,t):~(\log\log a_r)^{-2}\leq v,t\leq(\log\log a_r)^2,\quad |v-t|\leq\omega_r\right\}$$ and
\begin{equation}\label{eq:J-definition}
J_r(x,y):=\frac1{4\pi^2}\iint_{\Delta_r^+}
\frac{\1_{\mathcal A_r(v)}F_{a_r}(v)x^{\ii v}\1_{\mathcal A_r(t)}\overline{F_{a_r}(t)}y^{-\ii t}}
{(1/2+\ii v)(1/2-\ii t)}K_r(v-t)\,\dd v\,\dd t.
\end{equation}
Since $\omega_r<(\log\log a_r)^{-2}$, the integration domain $\Delta_r$ in the definition of $I_r(x,y)$ is the disjoint union of $\Delta_r^+$ and $-\Delta_r^+$. The identities
$$F_{a_r}(-t)=\overline{F_{a_r}(t)},\quad\mathcal A_r(-t)=\mathcal A_r(t),\quad K_r(-t)=\overline{K_r(t)},\quad t\in\mr$$ ensure that $I_r(x,y)=2\,\operatorname{Re}J_r(x,y)$ and thereupon
\begin{equation}\label{eq:positive-part-moment}
\max_{x\in\mathcal X_r}\sum_{y\in\mathcal X_r}|I_r(x,y)|^n
\leq2^n\max_{x\in\mathcal X_r}\sum_{y\in\mathcal X_r}|J_r(x,y)|^n.
\end{equation}
We shall use the bound
\begin{equation}\label{eq:denominator-moment}
\frac1{(1/4+v^2)^{1/2}(1/4+t^2)^{1/2}}\leq 4,
\quad (v,t)\in\Delta_r^+.
\end{equation}
Put
$$\widehat F_r(t):=\frac{\1_{\mathcal A_r(t)}F_{a_r}(t)}{1/2+\ii t},\quad
\mathscr H_r(v,t):=\widehat F_r(v)\overline{\widehat F_r(t)}K_r(v-t),$$
$$\sigma_j:=\begin{cases}1,&1\leq j\leq k,\\-1,&k<j\leq n,\end{cases}
\quad\text{and}\quad s(\boldsymbol t):=\sum_{j=1}^n\sigma_jt_j.$$
Writing $\dd\boldsymbol v:=\dd v_1\cdots\dd v_n$ and $\dd\boldsymbol t:=\dd t_1\cdots\dd t_n$, an expansion of $J_r(x,y)^k\overline{J_r(x,y)}^k$ gives
\begin{multline}\label{eq:expanded-row}
\sum_{y\in\mathcal X_r}|J_r(x,y)|^n
=\frac1{(4\pi^2)^n}\int_{(\Delta_r^+)^n}\prod_{j=1}^k\mathscr H_r(v_j,t_j)
\prod_{j=k+1}^n\overline{\mathscr H_r(v_j,t_j)}\\
{}\times x^{\ii s(\boldsymbol v)}\Big(\sum_{y\in\mathcal X_r}y^{-\ii s(\boldsymbol t)}\Big)
\,\dd\boldsymbol v\,\dd\boldsymbol t.
\end{multline}
By \eqref{eq:geometric},
$$\Big|\sum_{y\in\mathcal X_r}y^{-\ii s(\boldsymbol t)}\Big|
=\Big|\sum_{l=0}^{|\mathcal X_r|-1}\eee^{-\ii l s(\boldsymbol t)}\Big|
\leq D_{|\mathcal X_r|}(s(\boldsymbol t)).$$
Consequently, with $P_r(\boldsymbol t):=\prod_{j=1}^n\1_{\mathcal A_r(t_j)}|F_{a_r}(t_j)|$,
\begin{equation}\label{eq:row-before-square}
\sum_{y\in\mathcal X_r}|J_r(x,y)|^n
\leq\pi^{-2n}\int_{(\Delta_r^+)^n}P_r(\boldsymbol v)P_r(\boldsymbol t)
D_{|\mathcal X_r|}(s(\boldsymbol t))\prod_{j=1}^n|K_r(v_j-t_j)|\,\dd\boldsymbol v\,\dd\boldsymbol t.
\end{equation}
Using $P_r(\boldsymbol v)P_r(\boldsymbol t)\leq \big((P_r(\boldsymbol v))^2+(P_r(\boldsymbol t))^2\big)/2$, we bound the integral in \eqref{eq:row-before-square} by the sum of two terms. In the term containing $(P_r(\boldsymbol v))^2$, we interchange $\boldsymbol v$ and $\boldsymbol t$. The domain is unchanged, and $|K_r(t_j-v_j)|=|K_r(v_j-t_j)|$. Furthermore,
$$|s(\boldsymbol v)-s(\boldsymbol t)|\leq\sum_{j=1}^n|v_j-t_j|\leq n\omega_r.$$ With this at hand, the definition of $\mathcal K_{r,n}$ secures $$D_{|\mathcal X_r|}(s(\boldsymbol v))\leq\mathcal K_{r,n}(s(\boldsymbol t))\quad\text{and}\quad D_{|\mathcal X_r|}(s(\boldsymbol t))\leq\mathcal K_{r,n}(s(\boldsymbol t)).$$ As a result,
\begin{equation*}
\sum_{y\in\mathcal X_r}\big|J_r(x,y)\big|^n \leq \pi^{-2n}\int_{(\Delta_r^+)^n}(P_r(\boldsymbol t))^2\mathcal K_{r,n}(s(\boldsymbol t))\prod_{j=1}^n|K_r(v_j-t_j)|\,\dd\boldsymbol v\,\dd\boldsymbol t.
\end{equation*}
For fixed $\boldsymbol t\in [(\log\log a_r)^{-2},(\log\log a_r)^2]^n$, enlarging each range of $v_j$ to $[t_j-\omega_r,t_j+\omega_r]$ gives $$\int \prod_{j=1}^n|K_r(v_j-t_j)|\,\dd\boldsymbol v\leq \prod_{j=1}^n \int_{t_j-\omega_r}^{t_j+\omega_r}|K_r(v_j-t_j)|\,\dd v_j=\Lambda_r^n,$$ where $$\Lambda_r:=\int_{-\omega_r}^{\omega_r}|K_r(u)|\,\dd u.$$
Thus, simultaneously for all $x\in\mathcal X_r$,
\begin{equation}\label{eq:moment-majorant}
\sum_{y\in\mathcal X_r}|J_r(x,y)|^n \leq \pi^{-2n}\Lambda_r^n\int_{\big[(\log\log a_r)^{-2}, (\log\log a_r)^2\big]^n}\prod_{j=1}^n\big(\1_{\mathcal A_r(t_j)}|F_{a_r}(t_j)|^2\big)
\mathcal K_{r,n}(s(\boldsymbol t))\,\dd\boldsymbol t.
\end{equation}
The right-hand side does not depend on $x$. Taking expectations in \eqref{eq:positive-part-moment} and invoking \eqref{eq:moment-majorant} and Fubini's theorem, we infer
\begin{multline}\label{eq:expected-majorant}
\me\Big[\max_{x\in\mathcal X_r}\sum_{y\in\mathcal X_r}|I_r(x,y)|^n\Big]\\\leq(2/\pi^2)^n\Lambda_r^n\int_{\big[(\log\log a_r)^{-2}, (\log\log a_r)^2\big]^n}\me\Big[\prod_{j=1}^n\big(\1_{\mathcal A_r(t_j)}|F_{a_r}(t_j)|^2\big)\Big]\mathcal K_{r,n}(s(\boldsymbol t))\,\dd\boldsymbol t.
\end{multline}

For $s\in\mn$, put $\Omega_s:=\{\boldsymbol t\in \big[(\log\log a_r)^{-2}, (\log\log a_r)^2\big]^s:~t_1<\ldots<t_s\}$. For $\boldsymbol\varepsilon=(\varepsilon_1,\ldots,\varepsilon_n)\in\{-1,1\}^n$, define
\begin{equation}\label{eq:ordered-integral-definition}
\mathfrak J_{r,n}(\boldsymbol\varepsilon):=\int_{\Omega_n}\mathcal K_{r,n}\Big(\sum_{j=1}^n\varepsilon_jt_j\Big)\prod_{j=1}^{n-1}\chi_r(t_{j+1}-t_j)\,\dd\boldsymbol t.
\end{equation}
Now we estimate the integral in \eqref{eq:expected-majorant}. To this end, we partition $\big[(\log\log a_r)^{-2}, (\log\log a_r)^2\big]^n$, apart from the sets where two coordinates coincide, into the $n!$ regions corresponding to the possible orderings of the coordinates. On each region, we relabel the coordinates in increasing order. The product inside the expectation is unchanged, whereas the signs $\sigma_j$ are permuted. Lemma~\ref{lem:cell} therefore gives
\begin{multline}\label{eq:ordered-moment-bound}
\int_{\big[(\log\log a_r)^{-2}, (\log\log a_r)^2\big]^n}\me\Big[\prod_{j=1}^n\big(\1_{\mathcal A_r(t_j)}|F_{a_r}(t_j)|^2\big)\Big]\mathcal K_{r,n}(s(\boldsymbol t))\,\dd\boldsymbol t\\
\leq n!(Ck)^{Ck}(\log a_r)^n\max_{\boldsymbol\varepsilon\in\{-1,1\}^n}\mathfrak J_{r,n}(\boldsymbol\varepsilon).
\end{multline}

Next, we work towards estimating \eqref{eq:ordered-integral-definition}. For $\boldsymbol t\in\Omega_n$, put $x_j:=t_{j+1}-t_j$ and, for $1\leq i\leq n$, put $$E_i:=\{i-1,i\}\cap\{1,\ldots,n-1\}$$ and $$\Omega_{n,i}:=\{\boldsymbol t\in\Omega_n:~x_j\geq (\log a_r)^{-b}~~ \text{for all }j\in E_i\}.$$ Thus $|E_i|=1$ if $i=1$ or $i=n$, and $|E_i|=2$ otherwise. Finally, put $\Omega_{n,0}:=\Omega_n\setminus\bigcup_{i=1}^n\Omega_{n,i}$. Since the integrand in \eqref{eq:ordered-integral-definition} is nonnegative,
\begin{equation}\label{eq:geometric-splitting}
\mathfrak J_{r,n}(\boldsymbol\varepsilon)\leq \sum_{i=1}^n\int_{\Omega_{n,i}}\mathcal K_{r,n}\Big(\sum_{j=1}^n\varepsilon_jt_j\Big)\prod_{j=1}^{n-1}\chi_r(x_j)\,\dd\boldsymbol t
+\int_{\Omega_{n,0}}\mathcal K_{r,n}\left(\sum_{j=1}^n\varepsilon_jt_j\right)\prod_{j=1}^{n-1}\chi_r(x_j)\,\dd\boldsymbol t.
\end{equation}
Fix $1\leq i\leq n$. On $\Omega_{n,i}$, we have $\prod_{j\in E_i}\chi_r(x_j)\leq (\max_{(\log a_r)^{-b}\leq x\leq (\log\log a_r)^2}\chi_r(x))^{|E_i|}$. After using this bound, the remaining factors $\chi_r$ do not depend on $t_i$. For fixed values of the other coordinates, \eqref{eq:kernel-for-moments} gives $$\int_{[(\log\log a_r)^{-2}, (\log\log a_r)^2]}\mathcal K_{r,n}\left(\varepsilon_i t_i+\sum_{j\neq i}\varepsilon_jt_j\right)\,\dd t_i\leq C_1(\log\log a_r)^3.$$ Here, the substitution $u=\varepsilon_i t_i+\sum_{j\neq i}\varepsilon_jt_j$ maps the integration interval onto an interval of the same length.

For $s\in\mn$, define $$V_{r,s}:=\int_{\Omega_s}\prod_{j=1}^{s-1}\chi_r(t_{j+1}-t_j)\,\dd t_1\ldots\dd t_s$$ and put $V_{r,0}:=1$. The change of variables $(t_1,\ldots,t_s)\mapsto (t_1,t_2-t_1,\ldots,t_s-t_{s-1})$ has unit Jacobian. Therefore,
\begin{multline}\label{eq:chain-integral}
V_{r,s}\leq \big((\log\log a_r)^2-(\log\log a_r)^{-2}\big)\Big(\int_0^{(\log\log a_r)^2}\chi_r(x)\,{\rm d}x\Big)^{s-1}\\\leq (\log\log a_r)^2 \Big(\int_0^{(\log\log a_r)^2}\chi_r(x)\,{\rm d}x\Big)^{s-1}, \quad s\in\mn.
\end{multline}

After integrating over $t_i$, we discard the restrictions involving the relative positions of the coordinates to its left and to its right. The remaining domain is contained in $\Omega_{i-1}\times\Omega_{n-i}$. Consequently,
\begin{multline*}
\int_{\Omega_{n,i}}\mathcal K_{r,n}\Big(\sum_{j=1}^n\varepsilon_jt_j\Big)\prod_{j=1}^{n-1}\chi_r(x_j)\,\dd\boldsymbol t\\ \leq C_1(\log\log a_r)^3 \Big(\max_{(\log a_r)^{-b}\leq x\leq(\log\log a_r)^2}\chi_r(x)\Big)^{|E_i|}
V_{r,i-1}V_{r,n-i}\\ \leq C_1(\log\log a_r)^{3+2|E_i|}\Big(\max_{(\log a_r)^{-b}\leq x\leq(\log\log a_r)^2}\chi_r(x)\Big)^{|E_i|}\Big(\int_0^{(\log\log a_r)^2}\chi_r(x)\,\dd x\Big)^{n-1-|E_i|}.
\end{multline*}
The last inequality follows from \eqref{eq:chain-integral}, because there are exactly $|E_i|$ nonempty groups of remaining coordinates. Using \eqref{eq:gap-integrals} and
\eqref{eq:large-gap-bound}, we obtain
\begin{multline*}
C_1(\log\log a_r)^{3+2|E_i|}\Big(\max_{(\log a_r)^{-b}\leq x\leq(\log\log a_r)^2}\chi_r(x)\Big)^{|E_i|}\Big(\int_0^{(\log\log a_r)^2}\chi_r(x)\,\dd x\Big)^{n-1-|E_i|}\\
\leq C_1c_2^{n-1-|E_i|}(\log a_r)^{b|E_i|}(\log\log a_r)^{3+6|E_i|+8(n-1-|E_i|)}\leq C_3^n(\log a_r)^{2b}(\log\log a_r)^{8n},
\end{multline*}
where $C_3:=\max(1,C_1,c_2)$. Here, we have used $1\leq|E_i|\leq 2$ and $$3+6|E_i|+8(n-1-|E_i|)=8n-5-2|E_i|\leq 8n.$$ Summing over $i$ gives
\begin{equation}\label{eq:isolated-integral}
\sum_{i=1}^n\int_{\Omega_{n,i}}\mathcal K_{r,n}\Big(\sum_{j=1}^n\varepsilon_jt_j\Big)\prod_{j=1}^{n-1}\chi_r(x_j)\,\dd\boldsymbol t\leq C_3^n n(\log a_r)^{2b}(\log\log a_r)^{8n}.
\end{equation}

Now consider $\boldsymbol t\in\Omega_{n,0}$. For each $i$, at least one gap indexed by $E_i$ is shorter than $(\log a_r)^{-b}$, whence $$n\leq \sum_{i=1}^n\sum_{j\in E_i}\1_{\{x_j<(\log a_r)^{-b}\}}=2\sum_{j=1}^{n-1}\1_{\{x_j<(\log a_r)^{-b}\}}.$$ Since $n=2k$, at least $k$ gaps are shorter than $(\log a_r)^{-b}$.
In particular, $$\1_{\Omega_{n,0}}\leq \sum_{\substack{J\subset\{1,\ldots,n-1\}\\|J|=k}}\prod_{j\in J}\1_{\{x_j<(\log a_r)^{-b}\}}\quad\text{on }\Omega_n.$$ Using \eqref{eq:kernel-pointwise-moments} and then changing
variables to $t_1,x_1,\ldots,x_{n-1}$, we obtain
\begin{multline*}
\int_{\Omega_{n,0}}\mathcal K_{r,n}\Big(\sum_{j=1}^n\varepsilon_jt_j\Big)\prod_{j=1}^{n-1}\chi_r(x_j)\,\dd\boldsymbol t \leq |\mathcal X_r|\sum_{\substack{J\subset\{1,\ldots,n-1\}\\|J|=k}}
\int_{\Omega_n}\prod_{j=1}^{n-1}\chi_r(x_j)\prod_{j\in J}\1_{\{x_j<(\log a_r)^{-b}\}}\,\dd\boldsymbol t\\\leq |\mathcal X_r| (\log\log a_r)^2 \binom{n-1}{k} (\int_0^{(\log a_r)^{-b}}\chi_r(x){\rm d}x)^k (\int_0^{(\log\log a_r)^2}\chi_r(x){\rm d}x)^{n-1-k}.
\end{multline*}
For the last inequality, we enlarged the range of each gap indexed by $J$ to $(0,(\log a_r)^{-b})$ and that of each remaining gap to $\big(0,(\log\log a_r)^2\big)$. Since $\binom{n-1}{k}\leq 2^{n-1}$ and $|\mathcal X_r|\leq 3\log a_r$, formula \eqref{eq:gap-integrals} now yields
\begin{multline}\label{eq:nonisolated-integral}
\int_{\Omega_{n,0}}\mathcal K_{r,n}\Big(\sum_{j=1}^n\varepsilon_jt_j\Big)\prod_{j=1}^{n-1}\chi_r(x_j)\,\dd\boldsymbol t\leq C_4^n\log a_r\,(\log\log a_r)^{2+(1-2B)k+8(n-1-k)}\\
\leq C_4^n\log a_r\,(\log\log a_r)^{8n-2Bk}.
\end{multline}
where $C_4:=2c_1\max(c_2,3)$.

Combining \eqref{eq:geometric-splitting}, \eqref{eq:isolated-integral} and \eqref{eq:nonisolated-integral}, we conclude that
\begin{equation}\label{eq:ordered-integral-bound}
\mathfrak J_{r,n}(\boldsymbol\varepsilon)\leq C_5^n n(\log\log a_r)^{8n}\big((\log a_r)^{2b}+\log a_r\,(\log\log a_r)^{-2Bk}\big)
\end{equation}
uniformly in $\boldsymbol\varepsilon$, where $C_5:=\max(C_3,C_4)$. Substituting \eqref{eq:ordered-moment-bound} and \eqref{eq:ordered-integral-bound} into \eqref{eq:expected-majorant}, we obtain
\begin{multline*}
\me\Big[\max_{x\in\mathcal X_r}\sum_{y\in\mathcal X_r}|I_r(x,y)|^{2k}\Big]\leq C_6^n n\,n!(Ck)^{Ck}(\Lambda_r\log a_r)^n (\log\log a_r)^{8n}\\{}\times \big((\log a_r)^{2b}+\log a_r\,(\log\log a_r)^{-2Bk}\big),
\end{multline*}
where $C_6:=2\pi^{-2}C_5$. By \eqref{eq:kernelBounds2} and $\log a_r\leq T_r\leq2\log a_r$, we have $\Lambda_r\log a_r\leq C_7\log\log\log a_r$. Finally, recalling that $n=2k$ and using $n!\leq n^n$ and $\log\log\log a_r\leq\log\log a_r$ for all sufficiently large $r$, we conclude that, for an appropriate absolute constant $C_0\geq 1$, $$C_6^n n\,n!(Ck)^{Ck}(\Lambda_r\log a_r)^n(\log\log a_r)^{8n}\leq(C_0k)^{C_0k}(\log\log a_r)^{C_0k}.$$
This proves \eqref{eq:highmoment}.
\end{proof}

\section{Proof of \eqref{eq:covinput}}\label{sect:proofcov}

Fix $\kappa\in(0,1)$. With $C_0$ as in Lemma~\ref{lem:highmoment}, let $C_1:=3C_0+10$, choose\footnote{The constants $C_1$ and $B$ are larger than needed for the subsequent proof. No attempt has been made to choose them optimally.} $b=\kappa/8$, $B=100+\lceil 16(C_1+4)/\kappa\rceil$ and then put $$k_r:=\left\lfloor\frac{\log\log a_r}{B\log\log\log a_r}\right\rfloor.$$ Observe that $1\leq k_r\leq\lfloor\log\log a_r\rfloor$ for all sufficiently large $r$ and thereupon
\begin{multline*}
\log\left[\big(2(\log\log a_r)^{3/5}\big)^{2k_r}(C_0k_r)^{C_0k_r}(\log\log a_r)^{C_0k_r}\right]\leq (2C_0+4)k_r\log\log\log a_r\\ \leq C_1k_r\log\log\log a_r\leq (C_1/B)\log\log a_r.
\end{multline*}
We have used the definition of $k_r$ for the last inequality. Furthermore, $(\log\log a_r)^{-2Bk_r}=(\log a_r)^{-2+o(1)}$. With these at hand, an application of Lemma~\ref{lem:highmoment} yields
\begin{equation}\label{eq:countMoment}
\me\Big[\big(2(\log\log a_r)^{3/5}\big)^{2k_r}\max_{x\in\mathcal X_r}\sum_{y\in\mathcal X_r}|I_r(x,y)|^{2k_r}\Big]\leq (\log a_r)^{\kappa/2}+(\log a_r)^{-1/2}
\end{equation}
for all sufficiently large $r$. Indeed, the exponents arising from the two terms in \eqref{eq:highmoment} are $$2b+C_1/B+o(1)<\kappa/2\quad\text{and}\quad -1+ C_1/B+o(1)<-1/2,$$ respectively.

Denote by $\mathcal T_r$ the event in Lemma~\ref{lem:truncation}, with the chosen $B$ and $b$. Then $\lim_{r\to\infty}\mmp(\mathcal T_r)=1$ and, on $\mathcal T_r$, $$\max_{x,y\in\mathcal X_r}\left|\frac{C_r(x,y)}{(xy)^{1/2}}-I_r(x,y)\right|\leq(\log\log a_r)^{-5/8}.$$ Since $(\log\log a_r)^{-5/8}=o((\log\log a_r)^{-3/5})$, for all sufficiently large
$r$, for all $x,y\in\mathcal X_r$, $$\mathcal T_r\cap \big\{|C_r(x,y)|>(\log\log a_r)^{-3/5}(xy)^{1/2}\big\}\subseteq \{|I_r(x,y)|>(\log\log a_r)^{-3/5}/2\}.$$ Therefore, on $\mathcal T_r$,
\begin{multline*}
\max_{x\in\mathcal X_r}\#\left\{y\in\mathcal X_r\setminus\{x\}:~|C_r(x,y)|>(\log\log a_r)^{-3/5}(xy)^{1/2}\right\}\\\leq \max_{x\in\mathcal X_r}\sum_{y\in\mathcal X_r\setminus\{x\}}
\1_{\{|I_r(x,y)|>(\log\log a_r)^{-3/5}/2\}}\leq \big(2(\log\log a_r)^{3/5}\big)^{2k_r}\max_{x\in\mathcal X_r}\sum_{y\in\mathcal X_r}|I_r(x,y)|^{2k_r}.
\end{multline*}
Markov's inequality and \eqref{eq:countMoment} now give
\begin{multline*}
\mmp\big\{\max_{x\in\mathcal X_r}\#\big\{y\in\mathcal X_r\setminus\{x\}:~|C_r(x,y)|>(\log\log a_r)^{-3/5}(xy)^{1/2}\big\}>(\log a_r)^\kappa\big\}\\
\leq \mmp(\mathcal T_r^c)+(\log a_r)^{-\kappa/2}+(\log a_r)^{-\kappa-1/2}~\to~ 0, \quad r\to\infty
\end{multline*}
thereby proving \eqref{eq:covinput}.

\section*{Statement on the use of AI tools}

The development of this paper involved an extended series of
mathematical discussions with GPT-6 Astra (OpenAI). These discussions contributed in particular
to the uniform fourth-moment estimate of Lemma~\ref{lem:fourth} and to
the conditional covariance estimate \eqref{eq:covinput} in
Proposition~\ref{prop:novel}. GPT-6 Astra was also used during the
preparation of the manuscript, for discussing the presentation,
checking drafts and language editing. Claude (Anthropic) was used for
proofreading and checking the final version.

The choice of the problem, the overall strategy of the proof and the formulation of the main results are due to the authors.
The authors take full responsibility for the content of the paper, the correctness of its mathematical results and proofs, and any remaining errors.

\end{document}